\documentclass{amsart}

\usepackage{amssymb}
\usepackage{amsmath}
\usepackage{amsfonts}
\usepackage{geometry}
\usepackage{bbm}
\usepackage{hyperref}
\usepackage{tikz}
\usepackage{mathrsfs}
\usepackage{multirow}
\usetikzlibrary{matrix,arrows,decorations.pathmorphing}
\usepackage{verbatim }
\usepackage{mathtools}
\usepackage[author={Sebastian}]{pdfcomment}
\usepackage{todonotes}

\newcommand{\myquad}[1][1]{\hspace*{#1em}\ignorespaces}

\newcounter{cprop}[section]

\newtheorem{theorem}[cprop]{Theorem}
\theoremstyle{plain}

\newtheorem{corollary}[cprop]{Corollary}

\newtheorem{lemma}[cprop]{Lemma}
\newtheorem{proposition}[cprop]{Proposition}

\newtheorem{assumption}[cprop]{Assumption}

\numberwithin{equation}{section}

\theoremstyle{definition}
\newtheorem{definition}[cprop]{Definition}

\theoremstyle{remark}
\newtheorem{remark}[cprop]{Remark}

\newcommand{\E}{\mathbb{E}}
\renewcommand{\P}{\mathbb{P}}

\newcommand{\R}{\mathbb{R}}
\newcommand{\N}{\mathbb{N}}
\newcommand{\Z}{\mathbb{Z}}

\renewcommand{\d}{\mathrm{d}}

\newcommand{\vertiii}[1]{{\left\vert\kern-0.25ex\left\vert\kern-0.25ex\left\vert #1 
		\right\vert\kern-0.25ex\right\vert\kern-0.25ex\right\vert}}

\begin{document}
	\title[Delay equations driven by fBm]{Singular stochastic delay equations driven by fractional Brownian motion: Dynamics, longtime behaviour, and pathwise stability}
	
	\author{M. Ghani Varzaneh}
	\address{Mazyar Ghani Varzaneh\\
		Fachbereich Mathematik und Statistik, Universität Konstanz, Konstanz, Germany}
	\email{mazyarghani69@gmail.com}

	\author{S. Riedel}
	\address{Sebastian Riedel \\
		Fakult\"at f\"ur  Mathematik und Informatik, FernUniversit\"at in Hagen, Hagen, Germany}
	\email{sebastian.riedel@fernuni-hagen.de }

	\keywords{rough paths, stochastic delay differential equations, fractional Brownian motion, random dynamical systems, invariant manifolds, pathwise stability, Multiplicative Rrgodic Theorem, semigroup theory}
	
	\subjclass[2020]{60L20, 60L50, 37H10, 37H15, 37H30, 34K19, 34K20, 34K50, 60G22}
	
	\begin{abstract}
		We study differential equations with a linear, path dependent drift and discrete delay in the diffusion term driven by a $\gamma$-H\"older rough path for $\gamma > \frac{1}{3}$. We prove well-posedness of these systems and establish a priori bounds for their solutions. Applying these results to an equation driven by a multidimensional fractional Brownian motion with Hurst paramter $H \in \big(\frac{1}{3},1\big)$, we can prove the existence of a Lyapunov spectrum for the linearized system that describes its long-time behaviour. Furthermore, we can deduce the existence of local stable, unstable, and center manifolds for the nonlinear equation. As an application, we can prove that under suitable conditions, the solution exhibits pathwise local exponential stability. En passant, we present new, concise and relatively short proofs for some classical results for $\mathcal{C}_0$-semigroups that build on the Multiplicative Ergodic Theorem formulated on Banach spaces.
	\end{abstract}
	
	\maketitle
	
	\section*{Introduction}
	
	Delay differential equations (DDEs) are used to model complex processes for which the dynamics include aftereffect phenomena. In its most general formulation, they can be written as
	\begin{align*}
		\d y_t = f(t,y_t,y |_{(-\infty,t]})\, \d t
	\end{align*}
	where the solution $y$ is a function taking values in $\R^n$. In many cases, delay equations offer a more accurate description of a real-world phenomenon than a simple ordinary differential equation. Applications of delay differential equations exist in various areas including biology, epidemiology, population dynamics, chemistry, nonlinear optics and lasers, fluid dynamics, economics, and mechanical engineering (cf. \cite{Ern09} where various applications are discussed in more detail). From a mathematical point of view, a delay differential equation is usually understood as an evolution equation on an infinite dimensional path space (cf. \cite{HL93}) which renders its analysis significantly more complicated compared to an ordinary differential equation. Furthermore, effects like oscillatory instabilities are often observed within delayed systems that induce severe challenges regarding their control and observability \cite{Bri15} or their numerical approximation \cite{BZ13}. \smallskip
	
	Stochastic DDEs constitute a natural generalization of the deterministic case. Not surprisingly, these models have numerous applications as well and we refer the reader to \cite{Moh84, Moh98, Buc00} and the references therein where examples of stochastic DDEs are discussed in more detail. The noise in a stochastic DDE is typically modelled by a Brownian motion. This is justified in many cases where the noise is supposed to be \emph{white}, i.e. if all time instances are assumed to be uncorrelated. From a mathematical point of view, this is also convenient because standard It\=o theory applies in this case \cite{Mao08}. However, it is not hard to imagine cases (e.g. in financial modelling) where considering memory within the noise is more natural. A typical process that is used for these purposes is the \emph{fractional Brownian motion} (fBm), a generalization of the Brownian motion that was introduced by Mandelbrot and Van Ness \cite{MvN68}. This process has the advantage that it is still Gaussian and furthermore possesses stationary increments, leading to a stationary noise process. However, on the mathematical side, solving a stochastic differential equation driven by a fractional Brownian motion is significantly more challenging than in the classical Brownian motion case. Indeed, since the fBm is not a semimartingale, It\=o's theory of stochastic integration completely breaks down. Fortunately, since the fBm is a Gaussian process, Malliavin calculus can still be applied \cite{Nua05, BHOZ08} and in some cases, the Skorokhod integral can be used as a substitute for the It\=o integral. However, for generic stochastic differential equation, let alone equations with time delay, this theory is not easily applicable. A more elegant solution is provided by Lyons' theory of \emph{rough paths} \cite{LQ02, LCL07, FV10, FH20}. In fact, rough path theory clearly separates the two problems of solving a differential equation (typically by a fixed point argument) and making sense of a stochastic integral. Without going too much into detail here, rough paths theory requires to make sense of iterated stochastic integrals only which simplifies the problem a lot. In recent years, rough path theory became the standard tool for studying stochastic differential equations driven by a fractional Brownian motion, and the reader may consult \cite[Chapter 15]{FV10}, \cite[Chapter 10]{FH20}, and \cite[Chapter 8]{GVS23} for an overview of various problems where this theory was successfully applied to. Despite the success of rough path theory to study differential equations driven by an fBm, there are relatively few results that use this approach to study equations containing a time delay. The article that laid the groundwork to do this was written by Neuenkirch, Nourdin, and Tindel \cite{NNT08}. These authors study equations with finitely many discrete delays, e.g.
	\begin{align}\label{eqn:rough_delay_NNT}
		\begin{split}
			\mathrm{d}y_{t} &= b(y_{t},y_{t-r_1}, \ldots, y_{t-r_k})\, \d t + G (y_{t},y_{t-r_1}, \ldots, y_{t-r_k}) \, \mathrm{d}\mathbf{X}_{t}; \quad t \geq 0\\
			\qquad y_{s} &= \xi_{s}; \quad -r_k\leq s\leq 0,\ \ 
		\end{split}
	\end{align}
	where $0 < r_1 < \ldots < r_k < \infty$ and $\mathbf{X}$ is a $\gamma$-H\"older rough path, $\gamma > \frac{1}{3}$, possessing a so-called \emph{delayed L\'evy area} (see Definition \ref{SAZAS} below). They are able to prove existence and uniqueness of solutions to \eqref{eqn:rough_delay_NNT} and can show that an fBm can be enhanced with a delayed L\'evy area, resulting in a theory capable to deal with stochastic DDEs with discrete delay driven by an fBm. However, one weakness of the results obtained in \cite{NNT08} is that the authors have to assume that the drift term $b$ in \eqref{eqn:rough_delay_NNT} is globally bounded. In applications, this is hardly ever the case as the reader may see by studying the examples presented in the references above. Furthermore, the theory in \cite{NNT08} is formulated for discrete delays only, whereas in many applications, the dynamics may be influenced by the whole path of the solution until present time or at least by a path segment. The first main contribution of this work is an extension of the results in \cite{NNT08} that allows to study a significant larger class of equations. To wit, we will be able to prove existence and uniqueness for equations of the form
	\begin{align}\label{eqn:rough_delay_with_linear_drift}
		\begin{split}
			\mathrm{d}y_{t} &=  A \bigg( y_t,\int_{-r_1}^{0} y_{\theta+t} \, \pi_1(\mathrm{d}\theta), \ldots, \int_{-r_k}^{0}y_{\theta+t} \, \pi_k(\mathrm{d}\theta) \bigg) \, \mathrm{d}t + b(y_{t},y_{t-r_1}, \ldots, y_{t-r_k})\, \d t \\
			&\qquad  + G (y_{t},y_{t-r_1}, \ldots, y_{t-r_k}) \, \mathrm{d}\mathbf{X}_{t}; \quad t \geq 0\\
			\qquad y_{s} &= \xi_{s}; \quad -r_k\leq s\leq 0,\ \ 
		\end{split}
	\end{align}
	where $A \colon \R^{n \times (1 + k)} \to \R^n$ is linear and $\pi_1,\ldots,\pi_k$ are finite signed measures, see Theorem \ref{SSS}. We furthermore provide a priori bounds for the solution that are useful in applications.\smallskip
	
	Having established well-posedness for \eqref{eqn:rough_delay_with_linear_drift} makes it possible to study a much more general class of stochastic DDEs driven by an fBm than in \cite{NNT08}. However, our theorem is even interesting for the Brownian motion case since our result is the first one that gives a \emph{pathwise} meaning to equations of the form \eqref{eqn:rough_delay_with_linear_drift}. In \cite{GVRS22}, we observed that in the Brownian motion case, equations of the form \eqref{eqn:rough_delay_with_linear_drift} without the linear term $A$ induce a stochastic semiflow on random spaces. This, in turn, can be reformulated within L.~Arnold's theory of random dynamical systems (RDS; cf. \cite{Arn98}) by saying that \eqref{eqn:rough_delay_with_linear_drift} induces a \emph{cocycle} on a fiber-type bundle of Banach spaces, called a \emph{field of Banach spaces} (see Definition \ref{MESSSAA}). It turned out that this structure is extremely useful to study the long-time behaviour of stochastic DDEs. Let us mention here that for a long time, it was believed that DDEs of the form \eqref{eqn:rough_delay_with_linear_drift} do not generate a stochastic semiflow (or a cocycle), cf. e.g. the introduction in \cite{MS03}. It was understood in \cite{MS03} that there is an issue caused by the type of delay in the diffusion term $G$ in the equation. A stochastic DDE with a discrete delay in the diffusion term, as it is considered here, was therefore called \emph{singular} in \cite{MS03}, and we adopt this term. On the other hand, if the diffusion term depends smoothly on the past, the equation was called \emph{regular} in \cite{MS03}, and in this case, a stochastic semiflow indeed exists on a deterministic space. The articles \cite{MS03, MS04} exclusively study the regular case, whereas this article and its predecessors \cite{GVRS22, GVR21} concentrate on singular equations. Using these terms, the second main contribution of our article is that also the solution to the singular equation \eqref{eqn:rough_delay_with_linear_drift} driven by a fractional Brownian motion induces a cocycle on a field of Banach spaces, cf. Theorem \ref{SSRR}, and therefore generates a random dynamical system. We emphasize here that fractional Brownian motion introduces additional challenges for this objective. \smallskip
	
	As we already mentioned above, random dynamical systems are extremely useful when studying the long-time behaviour of a stochastic flow. They also have the great advantage that their usage does not depend on the Markov property which does not hold in the case of a fractional noise \cite{BRS17}. In the linear case, crucial information about the long time behaviour is encoded in the \emph{Lyapunov spectrum}. The fundamental theorem that assures its existence is the \emph{Multiplicative Ergodic Theorem} (MET), originally proved by Oseledets for matrix cocycles \cite{Ose68}. Using a version of the MET we proved for RDS defined on fields of Banach spaces \cite{GVR23}, we can deduce that linear stochastic DDEs possess a Lyapunov spectrum. For a general nonlinear equation as considered in \eqref{eqn:rough_delay_with_linear_drift}, we can apply the MET to the linearized equation provided a \emph{stationary point} (which can be seen as a generalized concept of a fixed point for a determinstic equation, cf. Definition \ref{STATA}) exists. The spectrum of the linearized cocycle can be used to prove the existence of invariant manifolds, i.e. stable, unstable, and center manifolds, for the nonlinear equation around a stationary point. These are the next important results in this paper, and we formulate them in Theorem \ref{thm:stable_manifold}, Theorem \ref{thm:unstable_manifold}, and Theorem \ref{thm:center_manifold}. Let us briefly comment on the existing literature in this field now. In the Brownian motion case, the existence of stable and unstable manifolds for regular stochastic DDEs were aready proven in \cite{MS04} and in the singular case in \cite{GVR21}. We believe that Theorem \ref{thm:center_manifold} is the first center manifold theorem formulated for stochastic DDEs even in the Brownian case. In this context, we want to underline the importance of center manifold theory since it is frequently used to study \emph{bifurcation}, a very interesting topic for DDEs due to their unexpected oscillatory instabilities. For a fractional Brownian motion, stable and unstable manifolds for rough ordinary differential equations were recently obtained in \cite{GVR23A} and center manifolds are studied in \cite{NK21, GVR23A}. Invariant manifold theorems for stochastic partial differential equations driven by a fractional Brownian motion can be found in \cite{GLS10, MG22, KN23, GVR23C, LNZ24}. We are not aware of any invariant manifold theorem for stochastic DDEs driven by an fBm up to now. \smallskip
	
	The last part of our paper deals with the question of pathwise exponential stability for stochastic DDEs, a very interesting topic in the field of delay differential equations. Using the stable manifold theorem, we can formulate several results in this direction both for linear and for nonlinear equations, cf. Section \ref{sec:stability}. In this context, we are able to reprove some classical results concerning exponential stability of abstract $\mathcal{C}_0$-semigroups that are well-known, yet difficult to obtain. We present new, short, and concise proofs for these theorems by using the MET. Since we believe that this new approach is interesting in its own right, we devoted a whole section to it, cf. Section \ref{sec:MET_approach}. \smallskip
	
	Let us summarize the main contributions of this article:
	\begin{itemize}
		\item The results in \cite{NNT08} are substantially extended by establishing the existence and uniqueness of solutions for the rough delay differential equation \eqref{eqn:rough_delay_with_linear_drift}. Additionally, we derive bounds for the solution (see Theorem \ref{SSS}). To control these bounds and prevent them from becoming excessively large—especially in the presence of an unbounded linear drift—we employ indirect techniques. Standard direct calculations often yield bounds that are too large for our intended future applications.
		\item 
		We demonstrate that the solution to \eqref{eqn:rough_delay_with_linear_drift} induces a random dynamical system on a continuous (measurable) field of Banach spaces (see Theorem \ref{SSRR}). Constructing such a system for delay equations driven by fractional noise is notably more challenging than for those driven by Brownian motion. In this work, we overcome this difficulty through the application of Malliavin calculus.
		\item The existence of stable, unstable, and center manifolds around stationary points are deduced for the semiflow induced by \eqref{eqn:rough_delay_with_linear_drift} (Theorem \ref{thm:stable_manifold}, \ref{thm:unstable_manifold}, and \ref{thm:center_manifold}). Our calculations enable these results to be derived using certain abstract findings. Direct approaches to constructing invariant manifolds are typically lengthy and complex.
		\item Several results regarding almost sure exponential stability for both linear and nonlinear stochastic delay differential equations are established (see Section Section \ref{sec:stability}). Proving exponential stability in non-Markovian cases is generally challenging, even in finite-dimensional settings. In our context, these difficulties are compounded by the fact that the equation is formulated in infinite-dimensional Banach spaces, where a classical dynamical system (typically defined on a fixed Banach space) cannot be established. To address these challenges, consistent with our other results, we rely on additional insights into the geometry of the underlying Banach spaces, which we refer to as fields of Banach spaces. In particular, we have removed the parameter restriction that we previously imposed for the special case of the Brownian motion in \cite{GVR21} and \cite{GVRS22}.
		\item Using the Multiplicative Ergodic Theorem for Banach spaces, we present new and short proofs of classical results about the exponential stability of abstract $\mathcal{C}_0$-semigroups (Section \ref{sec:MET_approach}).
	\end{itemize}

	The article is structured as follows: In Section \ref{sec:preliminaries}, we introduce the notation used throughout our work and present some basic notions in rough paths theory and random dynamical systems. Furthermore, we recall the approach to rough delay differential equations introduced in \cite{NNT08}. Section \ref{sec:delay_levy_area} discusses the delayed L\'evy area for the fractional Brownian motion and proves an approximation result with Malliavin calculus techniques. Afterwards, we prove existence, uniqueness, and an a priori bound for solutions to \eqref{eqn:rough_delay_with_linear_drift} in Section \ref{sec:well-posedness}. We furthermore discuss Fr\'echet-differentiability of the solution and some technical bounds for it. In Section \ref{sec:RDS}, we prove that \eqref{eqn:rough_delay_with_linear_drift} generates a random dynamical system on a continuous (measurable) field of Banach spaces and prove the invariant manifold theorems. Section \ref{sec:MET_approach} contains stability results for abstract $\mathcal{C}_0$-semigroups, proven with the MET, that are used in Section  \ref{sec:stability} to deduce stability results for linear and nonlinear stochastic delay differential equations. In the appendix, we collect some basic facts about Malliavin calculus and fractional calculus, introduce some related stochastic integrals that are used to define the delayed L\'evy area in Section \ref{sec:delay_levy_area} and prove a technical Lemma.
	\section{Preliminaries}\label{sec:preliminaries}
	\subsection*{Notation and basic definitions}
	In this part, we present some notations that will be used throughout the paper.
	\begin{itemize}
		\item Unless otherwise specified, \( U \),  \( W \) will always represent finite-dimensional, normed vector spaces over the real numbers, with the norm denoted by \( |\cdot| \). The notation \( L(U, W) \) refers to the set of linear and continuous functions from \( U \) to \( W \), equipped with the standard operator norm.
		\item In estimates, the notation \( a \lesssim b \) means that there exists a deterministic constant \( C \), which may depend on some irrelevant parameters, such that \( a \leq C b \).
		\item Let $J$ be an interval in $\mathbb{R}$. A map $m: J \to U$ is referred to as a \emph{path}. For a path $m$, its increment is denoted by $(\delta m)_{s,t} := m_{t} - m_{s}$, where $m_{t}$ represents $m(t)$. We define the supremum norm of $m$ over the interval $J$ as
		\[
		\Vert m\Vert_{\infty; J} := \sup_{s \in J} |m_t|.
		\]
		Additionally, the $\gamma$-H\"older seminorm, for $\gamma \in (0,1]$, is defined by 
		\[
		\Vert m \Vert_{\gamma;J} := \sup_{s, t \in J; s \neq t} \frac{\vert m_t - m_s \vert}{\vert t - s \vert^{\gamma}}.
		\]
		For a general \( 2 \)-parameter function \( m^{\#} \colon J \times J \to U \), we use similar notation. If the domain is clear from the context, we may omit \( J \) as a subscript. The space \( C^0(I, U) \) consists of all continuous paths \( m \colon J \to U \) equipped with the uniform norm, while \( C^{\gamma}(J, U) \) denotes the space of \( \gamma \)-Hölder continuous functions, equipped with the norm
		\[
		\Vert \cdot \Vert_{C^{\gamma};J} := \Vert \cdot \Vert_{\infty;J} + \Vert \cdot \Vert_{\gamma;J}.
		\]
		The space \( C^{\infty}(J, U) \) consists of functions that are infinitely differentiable. If \( 0 \in J \), using \( 0 \) as a subscript, such as in \( C^{\gamma}_0(J, U) \), denotes the subspace of functions for which \( x_0 = 0 \). An upper index, such as \( C^{0, \gamma}(J, U) \), indicates the closure of smooth functions under the corresponding norms.
		\item Let $C^n_b(U, {W})$ denote the space of functions $G \colon U \to {W}$ that have $n$ bounded derivatives, where each derivative up to the $n$-th order is continuous and bounded, with $n \geq 0$. In many cases, we will omit the specific domain and codomain and use the notation $C^n_b$ for simplicity.
	\end{itemize}
	We now provide some fundamental concepts from rough paths theory that are essential for this article. For a comprehensive overview, we direct the reader to \cite{FH20}.
	\begin{itemize}
		\item Let $X \colon \R \to \R^d$ be a locally $\gamma$-H\"older continuous path, with $\gamma \in (0,1]$. A \emph{L\'evy area} associated with $X$ is a continuous function $\mathbb{X} \colon \R \times \R \to \R^d \otimes \R^d$ satisfying the algebraic identity
		\[
		\mathbb{X}_{s,t} = \mathbb{X}_{s,u} + \mathbb{X}_{u,t} + (\delta X)_{s,u} \otimes (\delta X)_{u,t}
		\]
		for every $s, u, t \in \R$ and for which $\|\mathbb{X}\|_{2\gamma; J} < \infty$ holds on every compact interval $I \subset \R$. If $\gamma \in (1/3, 1/2]$ and $X$ admits a L\'evy area $\mathbb{X}$, we refer to the pair $ (X, \mathbb{X})$ as a \emph{$\gamma$-rough path}. 
		\begin{comment}
			
			If $\mathbf{X}$ and $\mathbf{Y}$ are $\gamma$-rough paths, the distance between them is defined by
			\[
			\varrho_{\gamma;I}(\mathbf{X}, \mathbf{Y}) := \sup_{s,t \in I; s \neq t} \frac{|X_{s,t} - Y_{s,t}|}{|t-s|^{\gamma}} + \sup_{s,t \in I; s \neq t} \frac{|\mathbb{X}_{s,t} - \mathbb{Y}_{s,t}|}{|t-s|^{2\gamma}}.
			\]
		\end{comment}
		\item  Let $I = [a,b]$ be a compact interval. A path $m \colon J \to U$ is called a \emph{controlled path} based on $X$ over the interval $J$ if there exists a $\gamma$-H\"older continuous path $m' \colon J \to L(\R^d, U)$ such that, for all $s, t \in J$, the relation
		\[
		(\delta m)_{s,t} = m'_s (\delta X)_{s,t} + m_{s,t}^{\#}
		\]
		holds, where $m^{\#} \colon J \times J \to U$ satisfies $\| m^{\#} \|_{2\gamma; J} < \infty$. The function $m'$ is referred to as the \emph{Gubinelli derivative} of $m$. 
		We denote the space of controlled paths based on $X$ over the interval $J$ by $\mathscr{D}_{X}^{\gamma}(J, U)$. This space is a Banach space with the norm
		\begin{align}\label{FAS}
			\|m\|_{\mathscr{D}_{X}^{\gamma}} := \|(m,m')\|_{\mathscr{D}_{X}^{\gamma}} := |m|_{\infty,J} + |m'|_{\infty,J} + \|m'\|_{\gamma;J} + \| m^{\#} \|_{2\gamma; J}.
		\end{align}
		In this paper, when the context makes it clear what $U$ represents, we may use the notation $\mathscr{D}_{X}^{\gamma}(J)$ instead of $\mathscr{D}_{X}^{\gamma}(J, U)$. Additionally, when we refer to $m_{s}^{\prime}X_{s,t}$, we actually mean $m_{s}^{\prime} \circ X_{s,t}$. When it is clear from the context, we may omit the symbol $\circ$ for function composition, in accordance with our conventions.
	\end{itemize}
	Finally, we recall the definition of a random dynamical system as introduced by L.~Arnold in \cite{Arn98}. 
	\begin{itemize}
		\item Let $(\Omega, \mathcal{F})$ and $(X, \mathcal{B})$ be measurable spaces. Let $\mathbb{T}$ be either $\R$ or $\Z$, equipped with a $\sigma$-algebra $\mathcal{I}$, which is the Borel $\sigma$-algebra $\mathcal{B}(\R)$ when $\mathbb{T} = \R$ and the power set $\mathcal{P}(\Z)$ when $\mathbb{T} = \Z$. 
		
		A family $\theta = (\theta_t)_{t \in \mathbb{T}}$ of maps from $\Omega$ to itself is called a \emph{measurable dynamical system} if:
		\begin{itemize}
			\item[(i)] The map $(\omega, t) \mapsto \theta_t \omega$ is $\mathcal{F} \otimes \mathcal{I} / \mathcal{F}$-measurable, \vspace{0.07cm}
			\item[(ii)] $\theta_0 = \operatorname{Id}$, \vspace{0.07cm}
			\item[(iii)] $\theta_{s + t} = \theta_s \circ \theta_t$, for all $s, t \in \mathbb{T}$. \vspace{0.1cm}
		\end{itemize}
		If $\P$ is a probability measure on $(\Omega, \mathcal{F})$ that is invariant under each $\theta_t$, i.e., $\P \circ \theta_t^{-1} = \P$ for every $t \in \mathbb{T}$, then the tuple $\big(\Omega, \mathcal{F}, \P, \theta\big)$ is called a \emph{measurable metric dynamical system}. The system is called \emph{ergodic} if every $\theta$-invariant set has probability $0$ or $1$.
		
		\item Let $\mathbb{T}^+ := \{t \in \mathbb{T} \, : \, t \geq 0\}$, equipped with the trace $\sigma$-algebra. An \emph{(ergodic) measurable random dynamical system} on $(X, \mathcal{B})$ is an (ergodic) measurable metric dynamical system $\big(\Omega, \mathcal{F}, \P, \theta\big)$ along with a measurable map $\varphi \colon \mathbb{T}^+ \times \Omega \times X \to X$ that satisfies the \emph{cocycle property}, meaning $\varphi(0, \omega, \cdot) = \operatorname{Id}_X$ for all $\omega \in \Omega$, and
		\[
		\varphi(t+s, \omega, \cdot) = \varphi(t, \theta_s \omega, \cdot) \circ \varphi(s, \omega, \cdot),
		\]
		for all $s, t \in \mathbb{T}^+$ and $\omega \in \Omega$. The map $\varphi$ is referred to as a \emph{cocycle}. If $X$ is a topological space with $\mathcal{B}$ as its Borel $\sigma$-algebra, and the map $\varphi_{\cdot}(\omega, \cdot) \colon \mathbb{T}^+ \times X \to X$ is continuous for every $\omega \in \Omega$, then $\varphi$ is called a \emph{continuous (ergodic) random dynamical system}. Generally, we say that \emph{$\varphi$ possesses property $P$} if and only if $\varphi(t, \omega, \cdot) \colon X \to X$ exhibits property $P$ for every $t \in \mathbb{T}^+$ and $\omega \in \Omega$, provided this statement is well-defined. For simplicity and clarity, we will also use the notation $\varphi^s_\omega(\cdot)$ in place of $\varphi(s, \omega, \cdot)$ to streamline our presentation.
	\end{itemize}

	%   More generally, if $(G,\circ)$ is a group, we call a map $F \colon \R \times \Omega \to G$ a \emph{cocycle} if
	%   \begin{align*}
		%    F(t + s,\omega) = F(t,\theta_s \omega) \circ F(s,\omega)
		%   \end{align*}
	%   holds for all $s,t \in \R$ and all $\omega \in \Omega$. If the group $(G,\circ)$ is Abelian, we will often use the term \emph{helix} instead of cocycle.
	
	\subsection{Basic properties of rough delay equations}\label{sec:basics_rough_delay}
	The main equation that we are interested in for this paper is
	\begin{align}\label{equation+drift}
		\begin{split}
			\mathrm{d}y_{t} &= A \bigg( y_t,\int_{-r}^{0}y_{\theta+t} \, \pi(\mathrm{d}\theta) \bigg) \, \mathrm{d}t + G \big(y_{t},y_{t-r}\big) \, \mathrm{d}\mathbf{X}_{t},\\
			y_{s} &= \xi_{s}, \quad -r\leq s\leq 0,
		\end{split}
	\end{align}
	where \(\mathbf{X}\) is a \(\gamma\)-Hölder rough path with \(\gamma > \frac{1}{3}\) (and later a \emph{multidimensional fractional Brownian motion} with Hurst parameter \(H > \frac{1}{3}\)). Here, \(A\) represents a linear drift, and \(\pi\) is a finite signed measure on \([-r,0]\) with \(r > 0\).
	Note that we can add more delay components in the diffusion part in terms of Dirac measures or in the drift part in terms of finite signed measures. However, we do not explore this generality in order to keep the results more focused. The assumptions on \(A\), \(\pi\), $\mathbf{X}$, and \(G\) are detailed in Section \ref{sec:well-posedness}. We will study several aspects of this equation, including well-posedness and dynamics.
	We provide a brief overview of the concept of rough delay equations, along with the fundamental objects and definitions required to solve this class of equations.
	\begin{remark}
	We work with Equation \eqref{equation+drift} instead of \eqref{eqn:rough_delay_with_linear_drift} to enhance the readability of our results. Note that, for the linear drift term, the drift component of \eqref{eqn:rough_delay_with_linear_drift} can be transformed into the drift part of \eqref{equation+drift} using an appropriate measure. For the remaining terms in Equation \eqref{eqn:rough_delay_with_linear_drift}, by assuming \(0 < r_1 < \ldots < r_k\), we first solve the equation on \([0, r_1]\) and then extend the solution successively to \([0, r_2]\), and so forth, until reaching \([0, r_k]\). Moreover, by considering the new path \( t \mapsto (t, X_t) \), we can incorporate \( b \) into the diffusion term by a new rough path generated by \( \mathbf{X} \) and the identity path. Since the main idea remains consistent with the case where \( k = 1 \) and \( b = 0 \), we will focus on Equation \eqref{equation+drift} throughout this work.
	\end{remark}
	\begin{definition}\label{SAZAS}
		Let $\frac{1}{3} < \gamma \leq \frac{1}{2}$, and let $X:\mathbb{R} \to \mathbb{R}^d$ be such that $(X, \mathbb{X})$ is a \emph{$\gamma$-rough path}. Fix $r > 0$. We say that $\mathbf{X}$ is a \emph{delayed $\gamma$-rough path} if there exists a continuous function $\mathbb{X}(-r)$ given by
		\begin{align*}
			\mathbb{X}(-r): \mathbb{R} \times \mathbb{R} \to \mathbb{R}^d \otimes \mathbb{R}^d
		\end{align*}
		such that for every interval $[a,b] \subset \mathbb{R}$,
		\begin{align}\label{VBBN}
			\begin{split}
				&\Vert\mathbb{X}(-r)\Vert_{2\gamma,[a,b]}< \infty, \\
				& \quad \mathbb{X}_{s,t}(-r) = \mathbb{X}_{s,u}(-r) + \mathbb{X}_{u,t}(-r) + X_{s-r,u-r} \otimes X_{u,t}, \quad \forall s, u, t \in \mathbb{R}, \ s < u < t.
			\end{split}
		\end{align}
		We denote a delayed $\gamma$-rough path by $\mathbf{X} = (X, \mathbb{X}, \mathbb{X}(-r))$. For an interval $J = [a,b]$, we define the norm of $\mathbf{X}$ as
		\begin{align*}
			\|\mathbf{X}\|_{\gamma, J} := \|X\|_{\gamma,J} +\Vert\mathbb{X}\Vert_{2\gamma,J}+ \Vert\mathbb{X}(-r)\Vert_{2\gamma,J}.
			% + \int_{-r}^{0} \sup_{s,t \in J} \frac{\|\mathbb{X}_{s,t}(\theta)\|}{(t-s)^{2\gamma}} \, |\pi| \, \mathrm{d}(\theta).
		\end{align*}
		We use $\mathscr{C}^{\gamma}_2(\mathbb{R}^d)$ to denote the set of delayed $\gamma$-rough paths.
	\end{definition} 
	\begin{comment}
		\begin{definition}
			Assume $J = [a,b]$ be a compact interval and $\frac{1}{3}<\beta\leq \gamma$. Let $(X,\mathbb{X})$ is a $\gamma$-rough path. We say a $\beta$-H\"older path $\xi \colon J \to W$ is a controlled by  $\mathbf{X}$ in $J$, if exists a $\beta$-H\"older path $\xi' \colon J \to \mathcal{L}(\mathbb{R}^d,W)$ such that
			\begin{align*}%\label{eqn:controlled}
				(\delta \xi)_{s,t} = \xi'_s (\delta X)_{s,t} + \xi_{s,t}^{\#},
			\end{align*}
			for all $s,t \in J$ also $\Vert\xi^{\#}\Vert_{2\beta; J}:=\sup_{\substack{s,t\in J,\\ s<t}}\frac{\Vert \xi_{s,t}^{\#}\Vert}{(t-s)^{2\gamma}}<\infty$. The path $\xi'$ is called a \emph{Gubinelli-derivative} of $\xi$ and $\xi^{\#}$, the remainder. We call the set of all paths $(\xi,\xi^{\prime})$ which are controlled by $\mathbf{X}$, the space of \emph{controlled path} by $\mathbf{X}$. For every  $0<\beta\leq\gamma$,  we use $\mathscr{D}_{\mathbf{X}}^{\beta}(J,W)$ to denote this space with the following norm
			\begin{align*}
				\Vert(\xi,\xi^\prime)\Vert_{\mathscr{D}_{\mathbf{X}}^{\beta}(J,W)}:=\Vert\xi_a\Vert+\Vert\xi^{\prime}_a\Vert+\Vert\xi\Vert_{\beta;J}+\Vert\xi^\prime\Vert_{\beta;J}+\Vert\xi^{\#}\Vert_{2\beta; J}.
			\end{align*} 
		\end{definition}
	\end{comment}
	Next, let's recall what a delayed controlled path is.
	\begin{definition}
		Assume $J = [a,b]$ is a compact interval and $\frac{1}{3} < \beta \leq \gamma$. Let $\mathbf{X}$ be a delayed $\gamma$-rough path. We say that a $\beta$-H\"older path $\zeta \colon J \to W$ is a \emph{delayed controlled path} based on $\mathbf{X}$ over the interval $J$ if there exist $\gamma$-H\"older paths $\zeta^0, \zeta^1 \colon J \to L(\mathbb{R}^d, W)$ such that
		\begin{align}\label{eqn:delay_controlled}
			\zeta_{s,t} = \zeta^0_s (\delta X)_{s,t} + \zeta^1_s(\delta X)_{s-r,t-r} + \zeta^{\#}_{s,t}
		\end{align}
		for all $s,t \in J$, where $\zeta^{\#} \colon J \times J \to W$ satisfies $\|\zeta^{\#}\|_{2\gamma ; J} < \infty$.
		% \begin{align*}
			% \| m^{\#} \|_{2\gamma ; I} := \sup_{s < t \in I} \frac{| m^{\#}_{s,t} |}{(t-s)^{2\gamma}} < \infty .
			% \end{align*}
		The paths $\zeta^0$ and $\zeta^1$ are referred to as the \emph{Gubinelli derivatives} of $\zeta$. We denote the space of delayed controlled paths based on $\mathbf{X}$ over the interval $J$ by $\mathcal{D}_{\mathbf{X}}^\beta(J, W)$. A norm on this space is defined as
		\begin{align}\label{eqn:norm_delayed_controlled_path}
			\|\zeta\|_{\mathcal{D}_{X}^{\beta}(J,W)} := \|(\zeta,\zeta^0,\zeta^1)\|_{\mathcal{D}_{X}^{\beta}(J,W)} := |\zeta_a| + |\zeta^0_a| + |\zeta^1_a| + \|\zeta^0\|_{\beta; J} + \|\zeta^1\|_{\beta; J} + \|\zeta^{\#}\|_{2\beta; J}.
		\end{align}
	\end{definition}
	\begin{remark}
		Assume $J = [a,b]$ and let $G: W \times W \to L(\mathbb{R}^d, W)$ be a $C^2$ function. Suppose $\frac{1}{3} < \beta \leq \gamma$, $\xi \in \mathscr{D}_{\mathbf{X}}^{\beta}([a,b], W)$, and $\tilde{\xi} \in \mathscr{D}_{\mathbf{X}}^{\beta}([a-r, b-r], W)$. Then, for $\zeta: [a,b] \to L(\mathbb{R}^d, W)$ defined by $\zeta_{t} := G(\xi_{t}, \tilde{\xi}_{t-r})$, we have $\zeta \in \mathcal{D}_{X}^{\beta}(J, L(\mathbb{R}^d, W))$. Furthermore,
		\begin{align}\label{TYCZ}
			\begin{split}
				\zeta^{0}_s &= \frac{\partial G}{\partial x}(\xi_s, \tilde{\xi}_{s-r}) \xi^{\prime}_s, \\
				\zeta^{1}_s &= \frac{\partial G}{\partial y}(\xi_s, \tilde{\xi}_{s-r}) \tilde{\xi}^{\prime}_{s-r},
			\end{split}
		\end{align}
		for all $s \in J$. Additionally, for $\theta(\xi,\tilde{\xi})(s,t) := \theta(\xi_{t}, \tilde{\xi}_{t-r}) + (1-\theta)(\xi_{s}, \tilde{\xi}_{s-r})$, the remainder term of $\zeta$, denoted by $\zeta^{\#}_{s,t}$, is given by
		\begin{align*}
			\zeta^{\#}_{s,t} &= \underbrace{D_{(\xi_s, \tilde{\xi}_{s-r})}G \, (\xi^{\#}_{s,t}, \tilde{\xi}^{\#}_{s-r,t-r})}_{\frac{\partial G}{\partial x}(\xi_s, \tilde{\xi}_{s-r}) \xi^{\#}_{s,t} + \frac{\partial G}{\partial y}(\xi_s, \tilde{\xi}_{s-r}) \tilde{\xi}^{\#}_{s-r,t-r}} \\
			&\quad + \int_{0}^{1} (1-\theta) D^{2}G_{\theta(\xi, \tilde{\xi})(s,t)}\left((\delta (\xi, \tilde{\xi}_{\cdot-r}))_{s,t}, (\delta (\xi, \tilde{\xi}_{\cdot-r}))_{s,t}\right) \, \mathrm{d}\theta,
		\end{align*}
		for $s, t \in J$, where $(\delta (\xi, \tilde{\xi}_{\cdot-r}))_{s,t} = (\xi_{t}, \tilde{\xi}_{t-r}) - (\xi_{s}, \tilde{\xi}_{s-r})$. Note that the proof is a straightforward consequence of the following Taylor expansion:
		\begin{align*}
			G(\xi_t,\tilde{\xi}_t)-G(\xi_s,\tilde{\xi}_s)=D_{(\xi_s,\tilde{\xi}_{s-r})}G(\delta (\xi, \tilde{\xi}_{\cdot-r})_{s,t})+\int_{0}^{1} (1-\theta) D^{2}G_{\theta(\xi, \tilde{\xi})(s,t)}\left((\delta (\xi, \tilde{\xi}_{\cdot-r}))_{s,t}, (\delta (\xi, \tilde{\xi}_{\cdot-r}))_{s,t}\right) \, \mathrm{d}\theta.
		\end{align*}
	\end{remark}
	The following result allows us to extend the concept of rough integration to $\mathcal{D}_{X}^\gamma(J, U)$.
	\begin{theorem}\label{DCVZVXZ}
		Assume $\mathbf{X} = \big(X, \mathbb{X}, \mathbb{X}(-r)\big)$ is a delayed $\gamma$-rough path, and let $\frac{1}{3} < \beta \leq \gamma$ such that $2\beta + \gamma > 1$. Further assume that $\zeta$ is a $\beta$-H\"older, $L(\mathbb{R}^d, U)$-valued delayed controlled path based on $\mathbf{X}$, with the decomposition given by \eqref{eqn:delay_controlled} for $W = L(\mathbb{R}^d, U)$ on the interval $J = [a, b]$. Then the following limit
		\begin{align}\label{dfn}
			\int_{a}^{b} \zeta_{s}\, d\mathbf{X}_s := \lim_{|\Pi| \to 0}  \sum_{t_j \in \Pi} \left( \zeta_{t_{j}} X_{t_{j}, t_{j+1}} + \zeta^{0}_{t_{j}} \mathbb{X}_{t_{j}, t_{j+1}} + \zeta^{1}_{t_{j}} \mathbb{X}_{t_{j}, t_{j+1}}(-r) \right)
		\end{align}
		exists, where $\Pi$ is a partition of $[a, b]$. Moreover, there exists a constant $C$, depending only on $\gamma$, $\beta$, and $(b-a)$, such that for all $s < t \in [a, b]$, the estimate
		\begin{align}\label{REER}
			\begin{split}
				&\left| \int_s^t \zeta_{u}\, d\mathbf{X}_u - \zeta_s X_{s,t} - \zeta^{0}_s \mathbb{X}_{s,t} - \zeta^{1}_{s} \mathbb{X}_{s,t}(-r) \right| \\
				&\quad \leq C\left( \|\zeta^{\#} \|_{2\beta} \|X\|_{\gamma}  + \|\zeta^0\|_{\beta} \|\mathbb{X}\|_{2\gamma}  +  \|\zeta^1\|_{\beta} \|\mathbb{X}(-r)\|_{2\gamma} \right) |t-s|^{2\beta + \gamma}
			\end{split}
		\end{align}
		holds. In particular,
		\begin{align*}
			t \mapsto \int_s^t \zeta_{u}\, d\mathbf{X}_u
		\end{align*}
		is controlled by $X$ with Gubinelli derivative $\zeta$.
	\end{theorem}
	\begin{proof}
		Cf.\cite[Theorem 1.5.]{GVRS22}
	\end{proof}

	\section{Delayed Lévy Area for Fractional Brownian Motion with Smooth Approximations}\label{sec:delay_levy_area}
	Recall from Definition \ref{SAZAS} that a delayed rough path, in comparison to a classical rough path, includes an additional term which we refer to as the \emph{delayed Lévy area}. This extra element must satisfy the condition given in \eqref{VBBN}.
	\begin{comment}
		\begin{definition}
			Assume $\frac{1}{3}<\alpha\leq\frac{1}{2}$ and  $X \colon \R \to \mathbb{R}^d$ be a locally $\alpha$-H\"older path and $r > 0$. For $r_1\in\lbrace 0,r\rbrace$ , a \emph{delayed L\'evy area} for $X$ is a continuous function
			\begin{align*}
				\mathbb{X}(-r_1) \colon \R \times \R \to \mathbb{R}^d \otimes \mathbb{R}^d
			\end{align*}
			for which the algebraic identity
			\begin{align*}
				\mathbb{X}_{s,t}(-r_1)=\mathbb{X}_{s,u}(-r_1) + \mathbb{X}_{u,t}(-r_1)+X_{s-r_1,u-r_1}\otimes X_{u,t} 
			\end{align*}
			holds for every $s,u,t \in \R$ and for which $\| \mathbb{X}(-r_1) \|_{2\gamma ; I} < \infty$
			%  \begin{align*}
				%   \| \mathbb{X}(-r) \|_{2\gamma ; I} = \sup_{s < t \in I} \frac{| \mathbb{X}_{s,t}(-r) |}{(t-s)^{2\gamma}} < \infty
				%  \end{align*}
			holds on every compact interval $I \subset \R$. If $\gamma \in (1/3,1/2]$ and $X$ admits L\'evy- and delayed L\'evy area $\mathbb{X}$ and $\mathbb{X}(-r)$, we call $ \mathbf{X}= \big{(}X, \mathbb{X}, \mathbb{X}(-r)\big{)}$ a \emph{delayed $\gamma$-rough path with delay $r > 0$}. We use $\mathscr{C}^{\gamma}_2(\mathbb{R}^d)$, to denote the space of (delayed) L\'evy areas.
		\end{definition}
	\end{comment}
	Assume that \((B^i)_{1 \leq i \leq d}\) are independent fractional Brownian motions. While the existence of the Lévy area for this process is well established, the existence of the delayed Lévy area \(\mathbb{B}(-r)\) is more intricate. Indeed, as outlined in \cite{NNT08} [Pages 34-37], if we use the symmetric integral (cf. Definition \ref{symmetric_integral}) and define
	\begin{align}\label{DElY}
		\mathbb{B}_{s,t}(-r):=\int_{s}^{t}B_{s-r,u-r}\otimes \mathrm{d}^{\circ} B_{u},
	\end{align}
	then $\mathbb{B}_{s,t}(-r) $ is a process that belongs to the second chaos of the fractional Brownian motion $B$. Also
	\begin{align}\label{AAA_1}
		\begin{split}
			&\int_{s}^{t}B_{s-r,u-r}\otimes \mathrm{d}^{\circ} B_{u}\\&\quad=\sum_{1\leq i,j\leq d}\left(\delta^{B^{j}}(B^{i}_{s-r,.-r}\chi_{[s,t]}(.))+\delta_{i=j}\left(-Hr^{2H-1}(t-s)+\frac{1}{2}((t-s+r)^{2H}-r^{2H})\right)\right)e_i\otimes e_j
		\end{split}
	\end{align}
	and
	\begin{align}\label{BBB_1}
		\sup_{r\in [0,1]}\E | \delta^{B^{i}}(B^{j}_{s-r,.-r}\chi_{[s,t]}(.))|^2\lesssim (t-s)^{4H}.
	\end{align}
	Here, \(\delta_{i=j}\) denotes the Kronecker delta, which is equal to one if \(i = j\) and zero otherwise.
	By the Kolmogorov continuity theorem, we can demonstrate that $\mathbb{B}(-r)$ has a continuous version that satisfies the condition \eqref{VBBN} for every $\gamma < H$. Moreover, for this version, we can establish that
	\begin{align}\label{key}
		\Vert\mathbb{B}(-r)\Vert_{2\gamma,[0,r]} \in \bigcap_{p \geq 0} L^{p}(\Omega).
	\end{align}
	To define a random dynamical system, we need to ensure that \(\mathbb{B}_{s,t}(-r)\) can be approximated by smooth functions. We will use tools from Malliavin calculus to achieve this. The appendix provides basic background information and definitions for the notations used in this section. For the remainder of this section, we accept the following assumption:	
	\begin{assumption}\label{FFFDDCCA}
		We assume that \(\frac{1}{3} < H < \frac{1}{2}\) and that \(B = (B^1, \ldots, B^d) \colon \mathbb{R} \to \mathbb{R}^d\) is an \(\mathbb{R}^d\)-valued two-sided fractional Brownian motion with Hurst parameter \(H\). This process is defined on a probability space \((\Omega, \mathcal{F}, \mathbb{P})\) and is adapted to some filtration \((\mathcal{F}^t)_{t \in \mathbb{R}}\), with \(B_0 = 0\) almost surely.
	\end{assumption}
	Let us start with the following definition:
	\begin{definition}\label{REGU}
		Let $ \rho :\mathbb{R}\rightarrow [0,2] $ be a smooth function such that $\operatorname{supp}(\rho)\subset [0,1] $ and  $\int_{\mathbb{R}}\rho(z)\mathrm{d}z=1$. We set
		\begin{align*}
			B^{\epsilon}_{t}:=\int_{0}^1B_{-\epsilon z,t-\epsilon z}\rho(z)d z, \quad \epsilon \in (0,1].
		\end{align*}
	\end{definition}
	It is not hard to see that
	\begin{align}\label{AZAZ}
		\E| B^{\epsilon}_{s,t}|^{2}\leqslant M(t-s)^{2H} \ \ \text{and} \ \ \ \lim_{\epsilon \rightarrow 0}\E| B^{\epsilon}_{s,t}-B_{s,t}|^{2} =0
	\end{align}
	It is natural to expect that \(\int_{s}^{t} B^{\epsilon}_{s-r, u-r} \otimes \mathrm{d}B_{u}^{\epsilon}\) converges to \(\mathbb{B}_{s,t}(-r)\). We will make this statement precise through several intermediate steps.
	\begin{lemma}\label{representation}
		We have the following pathwise identity on a set of full measure:
		\begin{align}\label{eqn:pathwise}
			\int _{s}^{t}B^{\epsilon}_{s-r,u-r}\otimes \mathrm{d} B_{u}^{\epsilon}=\int_{0}^1\rho(z)\int _{s-\epsilon z}^{t-\epsilon z}B^{\epsilon}_{s-r,u+\epsilon z-r}\otimes \mathrm{d}B_{u} \mathrm{d}z=\int _{0}^1\rho(z)\int _{s-\epsilon z}^{t-\epsilon z}B^{\epsilon}_{s-r,u+\epsilon z-r}\otimes \mathrm{d}^{\circ}B_{u} \mathrm{d}z.
		\end{align}
		\begin{proof}
			Note that the integrals in \eqref{eqn:pathwise} are defined pathwise, since \(B^{\epsilon}\) is smooth and \(B\) is a Hölder continuous path (almost surely for \(\gamma < H\)). Using integration by parts, we obtain:
			\begin{align*}
				\int _{s-\epsilon z}^{t-\epsilon z}(B^{\epsilon})^{i}_{s-r,u+\epsilon z-r}\mathrm{d}B^{j}_{u}=(B^{\epsilon})^{i}_{s-r,t-r}B^{j}_{t-\epsilon z}-\int_{s-r}^{t-r}B^{j}_{u+r-\epsilon z}\mathrm{d}(B^{\epsilon})^{i}_{u}.
			\end{align*}
			Consequently,
			\begin{align*}
				\int _{\mathbb{R}}\rho(z)&\int _{s-\epsilon z}^{t-\epsilon z}(B^{\epsilon})^{i}_{s-r,u+\epsilon z-r}\mathrm{d} B^{j}_{u} \mathrm{d}z\\&=\int_{\mathbb{R}}\rho(z)(B^{\epsilon})^{i}_{s-r,t-r}B^{j}_{t-\epsilon z}\mathrm{d}z-\int_{\mathbb{R}}\int_{s-r}^{t-r}\rho(z)B^{j}_{u+r-\epsilon z}\mathrm{d}(B^{\epsilon})^{i}_{u} \mathrm{d}z\\
				&=(B^{\epsilon})^{i}_{s-r,t-r}(B^{\epsilon})^{j}_{t}-\int_{s-r}^{t-r}(B^{\epsilon})^{j}_{u+r}\mathrm{d}(B^{\epsilon})^{i}_{u}.
			\end{align*}
			From integration by parts,
			\begin{align*}
				(B^{\epsilon})^{i}_{s-r,t-r}(B^{\epsilon})^{j}_{t}-\int_{s-r}^{t-r}(B^{\epsilon})^{j}_{u+r}\mathrm{d}(B^{\epsilon})^{i}_{u}
				=\int _{s}^{t}(B^{\epsilon})^{i}_{s-r,u-r}\mathrm{d} (B^{\epsilon})_{u}^{j},
			\end{align*}
			which implies the claim.
		\end{proof}
	\end{lemma}
	Before proceeding, we recall the following simple inequality that holds for every measure space \((Y, \mathcal{Y}, \mu)\) and every measurable function \(f: Y \times \Omega \to \mathbb{R}\):
	\begin{align}\label{ZVX5466}
		\left(\mathbb{E} \left( \int_{Y} f(y,\omega) \, \mathrm{d}\mu(y) \right)^2 \right)^{\frac{1}{2}} \leq \int_{Y} \left( \mathbb{E} \left( f^2(y,\omega) \right) \right)^{\frac{1}{2}} \, \mathrm{d}\mu(y).
	\end{align}
	The symmetric integral defined in Definition \ref{symmetric_integral} is closely related to the Stratonovich integral for Brownian motion. We can then derive the following useful bound, which is crucial for our subsequent objectives.
	
	\begin{lemma}\label{malliavin_2}
		For $\tau \in (-\infty,\sigma)$, we have
		\begin{align}\label{A_MALL}
			\E\left(\left\vert \int_{s}^{t}B_{u-\sigma,u-\tau}\otimes \mathrm{d}^{\circ}B_{u}\right\vert^{2}\right)\lesssim \big{(}sgn(\sigma)\vert \sigma\vert ^{2H-1}-sgn(\tau)\vert\tau\vert^{2H-1}\big{)}^2(t-s)^{2}+\sum_{1\leqslant i\leqslant 7}(t-s)^{\frac{i}{2}H}(\sigma-\tau)^{(4-\frac{i}{2})H} .
		\end{align}
	\end{lemma}
	\begin{proof}
		Set \(\Phi^{j}(u) := B_{u-\sigma, u-\tau}^j \chi_{[s,t]}(u)\) and \(\Phi(u) = (\Phi^{j}(u))_{1 \leq j \leq d}\). Note that for \(u \leq x\),
		\begin{align}\label{VACZXA54}
			\E\left(\left| \Phi(u) - \Phi(x) \right|^2\right) \lesssim (x-u)^{H}(\sigma - \tau)^H.
		\end{align}
		From \eqref{BB_1}, \eqref{ZVX5466}, and \eqref{VACZXA54},
		\begin{align}\label{F_MALL}
			\begin{split}
				&\E(\vert\Phi\vert_{\vert\mathcal{H}\vert^{ d}}^{2})\lesssim \int_{s}^{t}(\sigma-\tau)^{2H}(t-s)^{2H-1}\mathrm{d}u+\int_{-\infty}^{s}\left(\int_{s}^{t}\frac{(\sigma-\tau)^{H}}{(x-u)^{3/2-H}}\mathrm{d}x\right)^{2}\mathrm{d}u+\\&\int_{s}^{t}\left( \int_{u}^{t}\frac{\left(\E(\vert\Phi(x)-\Phi(u)\vert ^{2})\right)^{\frac{1}{2}}}{(x -u)^{\frac{3}{2}-H}}\mathrm{d}x\right)^{2}\mathrm{d}u\lesssim (\sigma-\tau)^{2H}(t-s)^{2H}+(\sigma-\tau)^{H}(t-s)^{3H} .
			\end{split}
		\end{align}
		Also, \((D^B \Phi(u))(x) = \chi_{[u-\sigma, u-\tau]}(x) \chi_{[s,t]}(u) I_{d \times d}\), where \(I_{d \times d}\) denotes the identity matrix. It is also easy to see 
		\begin{align*}
			\forall u\in [s,t]: \ \ \Vert D^{B}\Phi(u)\Vert_{\vert\mathcal{H}\vert^{ d}}^{2}:=\sum_{1\leq i\leq d}\Vert D^{B}\Phi^i(u)\Vert_{\vert\mathcal{H}\vert^d}^{2}\lesssim (\sigma-\tau)^{2H}.
		\end{align*}
		Consequently,	from Proposition \ref{MAL}
		\begin{align}\label{B_MALL}
			\int _{s}^{t}\Phi(u)\otimes\mathrm{d}^{\circ}{B_{u}}= \sum_{1\leq i,j\leq d}\big(\delta^{B^{j}}(\Phi^{i})+Tr_{[s,t]}(D^{B^{j}}\Phi^{i})) e_{i}\otimes e_j .
		\end{align}
		From \eqref{Trace}
		\begin{align}\label{trace}
			\begin{split}
				Tr_{[s,t]}(D^{B^{j}}\Phi^{i}) &=\lim_{\epsilon \rightarrow 0} \int_{s}^{t} \frac{\langle D^{B^j} \Phi^i(u), \chi_{[u-\epsilon, u+\epsilon] \cap [s,t]} \rangle_{\mathcal{H}}}{2\epsilon} \, \mathrm{d}u\\&\quad=\delta_{i=j}H\left(sgn(\sigma)\vert \sigma\vert ^{2H-1}-sgn(\tau)\vert\tau\vert^{2H-1}\right)(t-s).
			\end{split}
		\end{align}
		Recall that \((D^{B^{i}}\Phi^j)(x)(u) := (D^{B^{i}}\Phi^j(u))(x)\). Therefore, from \eqref{MAL_B},
		\begin{align}\label{C_MALL}
			\begin{split}
				&\delta^{B^{i}}((D^{B^{i}}\Phi^j)(x))=\delta_{i=j}\delta^{B^{i}}(\chi_{[s,t]\cap[x+\tau,x+\sigma]}(.))=\delta_{i=j}\chi_{[s-\sigma,t-\tau]}(x)B^{i}_{\max(s,x+\tau),\min(t,x+\sigma)}.
				\
				\\  %\delta^{B^{i}}_{[s,t]}\big{(}\chi_{[s-\tau,t-\tau]}(x)\chi_{[x+\tau ,t]}(.)+\chi_{[t-r,s-\tau]}(x)\chi_{[s,t]}(.)+\chi_{[s-r,t-r]}(x)\chi_{[s,x +r]}(.)\big{)}
				%	=\\ &\ \   \delta_{i=j}[\chi_{[s-\tau,t-\tau]}(x)B_{u +\tau ,t}+\chi_{[t-r,s-\tau]}(u)B_{s,t}+\chi_{[s-r,t-r]}(x)B_{s,x +r}] .
			\end{split}
		\end{align}
		Note that
		\begin{align}\label{maliavin_3}
			\begin{split}
				&\chi_{[s-\sigma,t-\tau]}(x)B^{i}_{\max(s,x+\tau),\min(t,x+\sigma)} = B^{i}_{\max(s,x+\tau),t}\chi_{[t-\sigma,t-\tau]}(x) + B^{i}_{\max(s,x+\tau),x+\sigma}\chi_{[s-\sigma,t-\sigma]}(x) \\
				&\quad = \begin{cases}
					B^{i}_{x+\tau,t}\chi_{[t-\sigma,t-\tau]}(x) + B^{i}_{x+\tau,x+\sigma}\chi_{[s-\tau,t-\sigma]}(x) + B^{i}_{s,x+\sigma}\chi_{[s-\sigma,s-\tau]}(x), & \text{if } \sigma-\tau \leq t-s, \\
					B^{i}_{x+\tau,t}\chi_{[s-\tau,t-\tau]}(x) + B^{i}_{s,t}\chi_{[t-\sigma,s-\tau]}(x) + B^{i}_{s,x+\sigma}\chi_{[s-\sigma,s-\tau]}(x), & \text{if } t-s \leq \sigma-\tau.
				\end{cases}
			\end{split}
		\end{align}
		%	For $ r-\tau\leqslant t-s$
		%	\begin{align}\label{D_MALL}
			%	\begin{split}
				%		&\delta^{B^{j}}_{[s,t]}(D_{u}^{B^{j}}\Phi_{.})=\\&\ \delta_{i=j}\delta^{B^{j}}_{[s,t]}\big{(}\chi_{[t-r,t-\tau]}(u)\chi_{[u +\tau ,t]}(.)+\chi_{[s-r,t-r]}(u)\chi_{[s,\sigma +r]}(.)-\chi_{[s-\tau ,t-r]}(u)\chi_{[s,u +\tau]}(.)\big{)}=\\ 
				%		&\ \ \delta_{i=j}[\chi_{[t-r,t-\tau]}(u)B_{u +\tau ,t}+\chi_{[s-r,t-r]}(u)B_{s,u +r}-\chi_{[s-\tau ,t-r]}(u)B_{s,u +\tau}].  
				%	\end{split}
			%\end{align}
			In addition, from \eqref{MAL_A} and \eqref{MAL_B}, we have
			\begin{align}\label{E_MALL}
				&\E\left((\delta^{B^{i}}(\Phi^j))^{2}\right)\E\left((\delta^{B^{i}}(\Phi^j))^{2}\right) = \E\left(\langle \Phi^{j}, D^{B^{i}}(\delta^{B^{i}}(\Phi^j))\rangle_{\mathcal{H}}\right) = \E\left(\|\Phi^{j}\|_{\mathcal{H}}^{2}\right) + \E\left(\langle \delta^{B^{i}}(D^{B^{i}}\Phi^j), \Phi^j \rangle_{\mathcal{H}}\right)\\ \nonumber 
				&\leq \E\left(\|\Phi^j\|_{\mathcal{H}}^{2}\right) + \E\left(\|\delta^{B^{i}}(D^{B^{i}}\Phi^j)\|_{\mathcal{H}}^{2}\right)^{\frac{1}{2}} \E\left(\|\Phi^j\|_{\mathcal{H}}^{2}\right)^{\frac{1}{2}}.
			\end{align}
			Similar to \eqref{F_MALL}, from \eqref{maliavin_3}, we can obtain 
			\begin{align}\label{F_MALL2}
				\E\left(\|\delta^{B^{i}}(D^{B^{i}}\Phi^j)\|_{\vert\mathcal{H}\vert}^{2}\right) \lesssim \sum_{1\leqslant k\leqslant 3} (t-s)^{kH}(\sigma-\tau)^{(4-k)H}.
			\end{align}
			\begin{comment}
				content...
				
				\textcolor{red}{
					Note that, we used  the following inequality
					\begin{align*}
						0\leqslant 2x^{2H}+2y^{2H}-\vert x+y\vert^{2H}-\vert x-y\vert^{2H}\leqslant 4x^{H}y, \ \ \ \ \   x,y\geqslant 0 .  
					\end{align*}
				}
			\end{comment}
			Now, \eqref{A_MALL} can be derived from \eqref{F_MALL}-\eqref{F_MALL2}.
		\end{proof}
		\begin{remark}
			During the proof of Lemma \ref{malliavin_2}, when $i \neq j$, due to the independence of $B^{i}$ and $B^{j}$, our calculation in \eqref{E_MALL} can be further simplified as follows: Let $\mathcal{F}_j$ be the filtration (on $\mathbb{R}$) to which $B^{j}$ is adapted. Then, by conditioning on $\mathcal{F}_j$ and using the fact that $B^{i}$ and $B^{j}$ are independent, we obtain:
			\[
			\E\left((\delta^{B^{i}}(\Phi^j))^{2}\right)=\E\left( \E\left( (\delta^{B^{i}}(\Phi^j))^{2} \mid \mathcal{F}_j \right) \right) = \E\left( \Vert \Phi^{j} \Vert_{\mathcal{H}}^2 \right).
			\]
			This coincides with the first line of our calculation in \eqref{malliavin_2}, since $D^{B^{i}}(\Phi^j)=0$ when $i \neq j$.
		\end{remark}
		\begin{proposition}\label{thm:approx_strat_delay_f}
			Let \(B^{\epsilon}\) be the regularized path in Definition \ref{REGU} and 	
			\begin{align*}
				\mathbf{B}_{s,t}^{\epsilon} := \left( B^{\epsilon}_{s,t}, \mathbb{B}_{s,t}^{\epsilon}, \mathbb{B}_{s,t}^{\epsilon}(-r) \right) := \left(B^{\epsilon}_{s,t}, \int_s^t B_{s,u}^{\epsilon}\, \otimes \mathrm{d}B^{\epsilon}_u, \int_s^t B^{\epsilon}_{s-r,u - r}\, \otimes \mathrm{d} B^{\epsilon}_u \right),
			\end{align*}
			Then 
			\begin{align*}
				\lim_{\epsilon \rightarrow\infty}\sup_{q\geqslant 1}\frac{\left\|d_{\gamma;J}\big{(} \mathbf{B}^{\epsilon} , \mathbf{B} \big{)}\right\|_{L^{q}}}{\sqrt{q}} = 0
			\end{align*}
			for each $\gamma <H$ and every compact interval $J \subset \R$ where $d_{\gamma;J}$ represents the homogeneous metric
			\begin{align*}
				d_{\gamma;J}(\mathbf{X},\mathbf{Y}) = \sup_{s,t \in J; s \neq t} \frac{\vert X_{s,t} - Y_{s,t}\vert}{|t-s|^{\gamma}} + \sqrt{ \sup_{s,t \in J; s \neq t} \frac{\vert\mathbb{X}_{s,t} - \mathbb{Y}_{s,t}\vert}{|t-s|^{2 \gamma}}} + \sqrt{ \sup_{s,t \in J; s \neq t} \frac{\vert\mathbb{X}_{s,t}(-r) - \mathbb{Y}_{s,t}(-r)\vert}{|t-s|^{2 \gamma}}} 
			\end{align*}
			and by \(\Vert \cdot \Vert_{L^{q}}\), we mean the usual \(L^{q}(\Omega)\) norm.
		\end{proposition}
		\begin{proof}
			Note that from \cite[Theorem 15.33 and Theorem 15.37]{FV10}, it is know that for every $\gamma<H$ and compact interval $J\subset\R$
			\begin{align*}
				\tilde{d}_{\gamma;J}(\mathbf{X},\mathbf{Y}) = \sup_{s,t \in J; s \neq t} \frac{\vert X_{s,t} - Y_{s,t}\vert}{|t-s|^{\gamma}} + \sqrt{ \sup_{s,t \in J; s \neq t} \frac{\vert\mathbb{X}_{s,t} - \mathbb{Y}_{s,t}\vert}{|t-s|^{2 \gamma}}},
			\end{align*}
			we have
			\begin{align*}
				\lim_{\epsilon \rightarrow\infty}\sup_{q\geqslant 1}\frac{\left\|d_{\gamma;J}\big{(} \mathbf{B}^{\epsilon} , \mathbf{B}^{\text{Strat}} \big{)}\right\|_{L^{q}}}{\sqrt{q}} = 0
			\end{align*}
			Therefore, we need to address the third item, in which the delay term appears.
			We claim  for every $\gamma<H$
			\begin{align}\label{III}
				\lim_{\epsilon\rightarrow 0}\sup_{s,t\in J}	\frac{\E\bigg{\vert}\int _{s}^{t}B^{\epsilon}_{s-r,u-r}\otimes \mathrm{d} B^{\epsilon}_{u}-\int_{s}^{t}B_{s-r,u-r}\mathrm{d}^{\circ}B_{u}\bigg{\vert}^{2}}{(t-s)^{4\gamma}}=0.
			\end{align}
			From Lemma \eqref{representation} 
			\begin{align*}
				&\E\left|\int _{s}^{t}B^{\epsilon}_{s-r,u-r} \otimes \mathrm{d}B^{\epsilon}_{u}-\int_{s}^{t}B_{s-r,u-r}\otimes \mathrm{d}^{\circ}B_{u}\right|^{2}\lesssim
				\\ &\underbrace{\E\left|\int_{0}^1\rho(z)\left(\int_{s-\epsilon z} ^{t-\epsilon z}\left(B^{\epsilon}_{s-r,u+\epsilon z-r}-B_{s-r,u+\epsilon z-r}\right)\otimes \mathrm{d}^{\circ}\ B_{u}\right)\mathrm{d}z \right|^{2}}_{\text{I}^{\epsilon}(s,t)}+\\  &\ \underbrace{\E\left|\int_{0}^1\rho(z)\left(\int_{s-\epsilon z}^{t-\epsilon z}B_{s-r,u+\epsilon z-r}\otimes \mathrm{d}^{\circ}B_{u} -\int_{s}^{t}B_{s-r,u-r}\otimes \mathrm{d}^{\circ}B_{u}\right)\mathrm{d}z\right|^{2}}_{\text{II}^{\epsilon}(s,t)}.
			\end{align*}
			The following identity can be easily verified using integration by parts:
			\begin{align*}
				\int_{s-\epsilon z}^{t-\epsilon z} B_{s-r, u+\epsilon z - r}^{i} \, \mathrm{d}^{\circ} B_{u}^{j} - \int_{s}^{t} B_{s-r, u-r}^{i} \, \mathrm{d}^{\circ} B_{u}^{j} &= \\
				-B_{s-r, t-r}^{i} B_{t-\epsilon z, t}^{j} + \int_{s-r}^{t-r} B_{u+r-\epsilon z, u+r}^{j} \, \mathrm{d}^{\circ} B_{u}^{i}.
			\end{align*}
			From \eqref{ZVX5466} and Lemma \ref{malliavin_2}
			\begin{align}\label{II}
				\begin{split}
					\lim_{\epsilon \to 0} \sup_{s,t \in J} \text{II}^\epsilon(s,t) 
					&\lesssim \lim_{\epsilon \to 0} \sup_{s,t \in J} \sum_{1 \leq i,j \leq d} \bigg\vert \int_{0}^1 \rho(z) \bigg( \E \big{(} -B_{s-r, t-r}^{i} B_{t - \epsilon z, t}^{j} \\
					&\quad + \int_{s-r}^{t-r} B_{u + r - \epsilon z, u + r}^{j} \, \mathrm{d}^{\circ} B_{u}^{i} \big{)}^{2} \bigg)^{\frac{1}{2}} \, \mathrm{d}z \bigg\vert^2 
					= 0.
				\end{split}
			\end{align}
			Also, from \eqref{AAA_1} and \eqref{BBB_1}, we have
			\[
			\sup_{s,t \in J} \frac{\text{II}^\epsilon(s,t)}{(t-s)^{4H}} < \infty.
			\]
			Therefore, using a simple interpolation technique, we obtain that for every \(\gamma < H\),
			\begin{align}\label{J_1}
				\lim_{\epsilon\rightarrow 0}\sup_{s,t\in J}\frac{\text{II}^\epsilon(s,t)}{(t-s)^{4\gamma}}= 0.
			\end{align}
			Similarly
			\begin{align*}
				\sup_{s,t\in J}\text{I}^\epsilon(s,t)\lesssim \sup_{s,t\in J}\sum_{1\leq i,j\leq d}\bigg|\int_{0}^1\rho(z)\bigg(\E\big(\int_{s-\epsilon z} ^{t-\epsilon z}((B^{\epsilon})^{i}_{s-r,u+\epsilon z-r}-B_{s-r,u+\epsilon z-r}^i)\mathrm{d}^{\circ}\ B_{u}^j\big)^{2}\bigg)^{\frac{1}{2}}\mathrm{d}z \bigg{\vert}^{2}.
			\end{align*}
			Set $ s_{1}=s-\epsilon z$ and $ t_{1}=t-\epsilon z $, then 
			\begin{align*}
				\E&\bigg{(} \int_{s_{1}}^{t_{1}}\big{(}(B^{\epsilon})^{i}_{s_{1},u+\epsilon z-r}-B_{s_{1},u+\epsilon z-r}^{i}\big{)}\mathrm{d}^{\circ}B_{u}^{j} \bigg{)}^{2}\\&\lesssim\bigg|\int_{0}^1\rho(y)\bigg(\E\big{(}\int_{s_{1}}^{t_{1}}\big{(}B_{s_{1}-\epsilon y,u+\epsilon z-r-\epsilon y}^{i}-B_{s_{1},u+\epsilon z-r}^{i} \big{)}\mathrm{d}^{\circ}B_{u}^{j} \big{)}^{2}\bigg)^{\frac{1}{2}}\mathrm{d}y\bigg|^2
			\end{align*} 
			We proceed similarly to the previous case and conclude that
			\begin{align}\label{J_2}
				\lim_{\epsilon\rightarrow 0}\sup_{s,t\in J}\frac{\text{I}^\epsilon(s,t)}{(t-s)^{4\gamma}}= 0.
			\end{align}
			Finally, from \eqref{J_1}, \eqref{J_2}, and the second chaos property of  $\mathbb{B}_{s,t}^{\epsilon}(-r), \mathbb{B}_{s,t}(-r) $, we conclude that
			\begin{align*}
				\lim_{\epsilon \rightarrow\infty}\sup_{q\geqslant 1}\frac{\left\vert\sqrt{ \sup_{s,t \in J; s \neq t} \frac{\vert\mathbb{X}_{s,t}(-r) - \mathbb{Y}_{s,t}(-r)\vert}{|t-s|^{2 \gamma}}}\right\Vert_{{L^{q}}}}{\sqrt{q}} = 0 .
			\end{align*}
			For more details, see \cite[Chapter 15]{FV10}.
		\end{proof}
		We will revisit this result later when we aim to prove that our solutions to Equation \eqref{equation+drift} generate a random dynamical system.
		
		\section{Existence, uniqueness, and some a priori bounds}\label{sec:well-posedness}
		In this section, we establish an existence and uniqueness result for Equation \eqref{equation+drift} driven by a general delayed rough path. We also obtain an a priori bound for the solution. For our future purposes, this bound must increase polynomially with respect to the initial value and the rough driver. Additionally, we derive several useful estimates for the Fréchet derivative of the solution.
		To obtain an a priori bound for our solution, there are several approaches available. We utilize a technique known as the decomposition of flows, which provides more straightforward estimates. Although it is possible to use a semigroup approach, this method is too complex for delay equations, and the bound obtained is difficult to control with a polynomial function. Let us start with the following assumption about the drift part.
		\begin{assumption}\label{SDDRRFF}
			\begin{itemize}
				\item Assume $A \in {L}(\mathbb{R}^n \times \mathbb{R}^n, \mathbb{R}^n)$ and let $\pi = (\pi^{i,j})_{1 \leq i,j \leq n}$ be a family of finite signed measures on $[-r,0]$. For $y \in C([-r,0], \mathbb{R}^n)$, assume that $\int_{-r}^{0} y_{\theta+t} \, \pi(\mathrm{d}\theta)$ denotes the vector
				\[
				\left(\sum_{1 \leq i \leq n} \int_{-r}^{0} y^i_{\theta+t} \, \pi^{i,j}(\mathrm{d}\theta)\right)_{1 \leq j \leq n}.
				\]
				\item Assume that for $\gamma > \frac{1}{3}$, $\mathbf{X}$ is a delayed $\gamma$-rough path.
			\end{itemize}
		\end{assumption}
		The following theorem is our main result in this section, which also generalizes \cite[Theorem 2.10]{GVR21}.
		\begin{theorem}\label{SSS}
			Let Assumption \ref{SDDRRFF} hold, and let \( \frac{1}{3} < \beta \leq \gamma \) such that \( 2\beta + \gamma > 1 \). Assume \( \xi \in \mathscr{D}_{\mathbf{X}}^{\beta}([-r,0]) \) and \( G \in C^{4}_{b}(\mathbb{R}^n \times \mathbb{R}^n, L(\mathbb{R}^d, \mathbb{R}^n)) \). Then, there exists a unique path \( y \in \mathscr{D}_{\mathbf{X}}^{\beta}([0, \infty)) \) such that for every \( 0 \leq s < t \), we have
			\begin{align*}
				(\delta y)_{s,t} &= \int_{s}^t A\left(y_\tau, \int_{-r}^{0} y_{\theta + \tau} \pi(\mathrm{d}\theta)\right) \mathrm{d}\tau \\
				&\quad + \int_{s}^{t} G \left(y_{\tau}, y_{\tau - r}\right) \mathrm{d}\mathbf{X}_{\tau}.
			\end{align*}
			Here, the last integral is defined in the sense of Theorem \ref{DCVZVXZ}. If \( A = 0 \), then the unique solution exists under weaker assumptions, specifically \( G \in C^{3}_{b} \). We denote the solution in both cases by \( {\phi}^{t}_{\mathbf{X}}(\xi) \). Also, there exists a polynomial \( Q \) (which only depends on \( \Vert A \Vert, G, \) and \( \pi \)) such that
			\begin{align}\label{BUN}
				\left\| \phi_{\mathbf{X}}(\xi) \right\|_{\mathscr{D}_{\mathbf{X}}^{\beta}([0,r])} \leq Q \left( \left\| \xi \right\|_{\mathscr{D}_{\mathbf{X}}^{\beta}([-r,0])}, \left\| \mathbf{X} \right\|_{\gamma,[0,r]} \right),
			\end{align}
			where \( \phi_{\mathbf{X}}(\xi) = (\phi_{\mathbf{X}}^{t}(\xi))_{0 \leq t \leq r} \). Additionally, considering the map
			\begin{align}\label{STTAASX}
				\begin{split}
					\phi_{\mathbf{X}}: \mathscr{D}_{\mathbf{X}}^{\beta}([-r,0]) &\rightarrow \mathscr{D}_{\mathbf{X}}^{\beta}([0,r]), \\
					\xi &\mapsto \phi_{\mathbf{X}}(\xi),
				\end{split}
			\end{align}
			we have that \( \phi_{\mathbf{X}} \) is continuous.
		\end{theorem}
		\begin{proof}
			%\todo{If I understand correctly, you are only proving the claimed bound for the solution. If this is the case, please say why existence and uniqueness of the equation and continuity of the initial condition follow.}
			For the case that $A=0$ and $G\in C^{3}_{b}$, our claim follows from \cite [Theorem 2.8]{GVR21}.
			To incorporate the case where \( A \neq 0 \), we employ the technique known as \emph{decomposition of flows}, as discussed in \cite{RS17}. Our proof will be divided into several steps and involves establishing well-posedness and obtaining a priori bounds for the solution. During the proof, when we refer to a polynomial bound, we always mean a two-variable polynomial in terms of $\Vert \xi \Vert_{\mathscr{D}_{\mathbf{X}}^{\beta}([-r,0])}$ and $\Vert \mathbf{X} \Vert_{\gamma,[0,r]}$ that bounds our term. Our bounds might also depend on \( G \), \( \pi \), and \( \Vert A \Vert \) without mentioning them explicitly.
			\smallskip
			
			\textbf{Step 1.} In this step, we first estimate  
			\[
			\sup_{t \in [0, r]} \|\phi^{t}_{\mathbf{X}}(\xi)\|
			\]
			and show that it can be bounded polynomially, as in \eqref{BUN}. Let's fix our feedback control, denoted by $(\xi_t)_{-r\leq t\leq0}$, and further assume $G\in C^{4}_{b}$. 
			For $0\leq s\leq t \leq r$ let $\tilde\phi_{\mathbf{X}}(s,t,y,\xi)$,  denote the solution to the following equation:
			\begin{align}\label{AZXSE}
				\begin{split}
					&\mathrm{d}z_{t}= G\big(z_{t},\xi_{t-r}\big)\mathrm{d}\mathbf{X}_{t},\\
					&\quad z_{s}=y ,
				\end{split}
			\end{align}
			for $t\geq s$.
			Using a standard fixed-point argument for rough differential equations, we can establish both the existence and uniqueness of a solution to the given equation. Additionally, for \(0 \leq s \leq u \leq t \leq r\), the flow property is satisfied, which means that
			
			\begin{align}\label{FFLLOOWW}
				\tilde{\phi}_{\mathbf{X}}(s, t, y, \xi) &= \tilde{\phi}_{\mathbf{X}}(u, t, \tilde{\phi}_{\mathbf{X}}(s, u, y, \xi), \xi).
			\end{align}
			Applying a standard argument for estimating (linear) rough differential equations, as detailed in \cite[Theorem 10.53]{FV10} and \cite[Pages 10-11]{GVR21}, we can show that for the polynomials \(P_1\) and \(P_2\), and a constant \(M\), the following bounds hold for every \(s \leq u \leq t \leq r\):
			\begin{align}\label{dffvv}
				\begin{split}
					&\sup_{\substack{y\in\mathbb{R}^n}}\vert\tilde{\phi}_{\mathbf{X}}(s,t,y,\xi)-y\vert\leq (t-s)^{\beta}P_{1}(\Vert\xi\Vert_{\mathscr{D}_{\mathbf{X}}^{\beta}([-r,0])},\Vert\mathbf{X}\Vert_{\gamma,[0,r] }),\\
					& \sup_{y\in\mathbb{ R}^n}\vert\tilde{\phi}_{\mathbf{X}}(s,t,y,\xi)-\tilde{\phi}_{\mathbf{X}}(s,u,y,\xi) -G(\tilde{\phi}_{\mathbf{X}}(s,u,y,\xi),\xi_{u-r})X_{u,t} \vert\leq (t-u)^{2\beta}P_{1}(\Vert\xi\Vert_{\mathscr{D}_{\mathbf{X}}^{\beta}([-r,0])},\Vert\mathbf{X}\Vert_{\gamma,[0,r] }),
				\end{split}
			\end{align}
			and 
			\begin{align}\label{inverse}
				\begin{split}
					&\max\big(\sup_{\substack{y\in\mathbb{R}^n}}\Vert(\frac{\partial \tilde{\phi}_{\mathbf{X}}}{\partial y}(s,t,y,\xi)-I\Vert,\sup_{\substack{y\in\mathbb{R}^n}}\Vert(\frac{\partial \tilde{\phi}_{\mathbf{X}}}{\partial y}(s,t,y,\xi))^{-1}-I\Vert\big)\leq\\&\quad M (t-s)^\beta P_{2}(\Vert\xi\Vert_{\mathscr{D}_{\mathbf{X}}^{\beta}([-r,0])},\Vert\mathbf{X}\Vert_{\gamma,[0,r]})\exp((t-s)P_{2}(\Vert\xi\Vert_{,\mathscr{D}_{\mathbf{X}}^{\beta}([-r,0])},\Vert\mathbf{X}\Vert_{\gamma,[0,r]})),
				\end{split}
			\end{align}
			for $	0\leq s\leq t \leq r$ , where \(I\) denotes the identity matrix in \(\mathbb{R}^n\). Note that the regularity assumption on \( G \) implies the differentiability of \( \tilde{\phi}_{\mathbf{X}}(s, t, y, \xi) \) with respect to \( y \), as stated in \cite[Theorem 2.6]{GVR21}.
			Let \( p_1 : [0, r] \rightarrow \mathbb{R}^n \) be a function that satisfies the following equation:
			\begin{align}\label{flow-decomposition_1}
				\begin{split}
					&\frac{\mathrm{d}p_1(t)}{\mathrm{d}t}=\big(\frac{\partial \tilde{\phi}_{\mathbf{X}}}{\partial y}(0,t,p_1(t),\xi)\big)^{-1}A\big(\tilde{\phi}_{\mathbf{X}}(0,t,p_1(t),\xi_t),\int_{-r}^{0}{\xi}_{\theta+t}^1\pi(\mathrm{d}\theta)\big),\\
					&\quad p_1(0)=\xi_0,
				\end{split}
			\end{align}
			where
			\begin{align*}
				{\xi}_t^1:=\begin{cases}
					\xi_t  & \text{for } -r\leq t<0,\\
					\tilde{\phi}_{\mathbf{X}}(0,t,p_1(t),\xi) & \text{for }  t\geq 0.
				\end{cases} \ \ 
			\end{align*}
			From the chain rule formula, we observe that \( \tilde{\phi}_{\mathbf{X}}(0, t, p_1(t), \xi) \) is a (local) solution to the equation \eqref{equation+drift}. 
			Let
			\begin{align}\label{B_DE}
				\begin{split}
					&B(\tau, \Vert \xi \Vert_{\mathscr{D}_{\mathbf{X}}^{\beta}([-r, 0])}, \Vert \mathbf{X} \Vert_{\gamma, [0, r]}) \\
					&\quad :=  M\tau^\beta P_{2}(\Vert \xi \Vert_{\mathscr{D}_{\mathbf{X}}^{\beta}([-r, 0])}, \Vert \mathbf{X} \Vert_{\gamma, [0, r]}) \exp\left(\tau P_{2}(\Vert \xi \Vert_{\mathscr{D}_{\mathbf{X}}^{\beta}([-r, 0])}, \Vert \mathbf{X} \Vert_{\gamma, [0, r]})\right).
				\end{split}
			\end{align}
			and define
			\begin{align}\label{t_1}
				\begin{split}
					&t_1^\beta := \frac{1}{P_{2}(\Vert \xi \Vert_{\mathscr{D}_{\mathbf{X}}^{\beta}([-r, 0])}, \Vert \mathbf{X} \Vert_{\gamma, [0, r]}) \vee P_{1}(\Vert \xi \Vert_{\mathscr{D}_{\mathbf{X}}^{\beta}([-r, 0])}, \Vert \mathbf{X} \Vert_{\gamma, [0, r]})} \wedge r, \\
					&N = \left\lfloor \frac{1}{t_1} \right\rfloor + 1.
				\end{split}
			\end{align}
			Note that with this choice for \( t_1 \), we have
			\begin{align}\label{ADCZx}
				B(t_1, \Vert \xi \Vert_{\mathscr{D}_{\mathbf{X}}^{\beta}([-r, 0])}, \Vert \mathbf{X} \Vert_{\gamma, [0, r]}) \leq  M \exp(r^{1 - \beta}),
			\end{align}
			therefore, from \eqref{inverse}, we get
			\begin{align}\label{inverse22}
				\begin{split}
					& s, t \in [0, r] \ \text{and} \ |t - s| \leq t_1: \\
					& \quad \max \left( \sup_{\substack{y \in \mathbb{R}^n}} \left\| \frac{\partial \tilde{\phi}_{\mathbf{X}}}{\partial y}(s, t, y, \xi)  \right\|, \sup_{\substack{y \in \mathbb{R}^n}} \left\| \left(\frac{\partial \tilde{\phi}_{\mathbf{X}}}{\partial y}(s, t, y, \xi)\right)^{-1}  \right\| \right) \leq 1 + M \exp(r^{1 - \beta}).
				\end{split}
			\end{align}
			Assume for $1\leq i<m\leq N$ , $p_{i}:[(i-1)t_1,it_1]\rightarrow\mathbb{R}^d$ is defined. We set
			\begin{align}\label{supp}
				\begin{split}
					&\tilde{p}_i:=:[(i-1)t_1,it_1]\rightarrow\mathbb{R},\\
					& \tilde{p}_i(t):=\sup_{(i-1)t_1\leq \tau\leq it_1 }\vert p_{i}(\tau)\vert.
				\end{split}
			\end{align}
			Let
			\begin{align}\label{inductively}
				\xi^m_t := \begin{cases}
					\xi_t, & \text{for } -r \leq t < 0, \\
					\tilde{\phi}_{\mathbf{X}}(it_1, t, p_{i+1}(t), \xi), & \text{for } i t_1 \leq t \leq (i+1)t_1 \text{ and } 0 \leq i \leq (m-2), \\
					\tilde{\phi}_{\mathbf{X}}((m-1)t_1, t, p_m(t), \xi), & \text{for } (m-1)t_1 \leq t \leq m t_1.
				\end{cases}
			\end{align}
			where \( p_m: [(m-1)t_1, mt_1] \to \mathbb{R}^d \) is given by the following equation:
			\begin{align}\label{flow-decomposition_m}
				\begin{split}
					&\frac{\mathrm{d}p_m(t)}{\mathrm{d}t}=\big(\frac{\partial\tilde{\phi}_{\mathbf{X}}}{\partial y}((m-1)t_1,t,p_m(t),\xi)\big)^{-1}A\big(\tilde{\phi}_{\mathbf{X}}((m-1)t_1,t,p_m(t),\xi),\int_{-r}^{0}{\xi}_{\theta+t}^m\pi(\mathrm{d}\theta)\big),\\
					&\quad p_m((m-1)t_1)=\tilde{\phi}_{\mathbf{X}}\big((m-2)t_1,(m-1)t_1,p_{m-1}((m-1)t_1),\xi\big).
				\end{split}
			\end{align}
			
			Similar to \eqref{supp}, we define \(\tilde{p}_m: [(m-1)t_1, mt_1] \to \mathbb{R}\) accordingly. Given the definition of $({\xi}^{m}_t)_{-r\leq 0\leq mt_1}$ and \eqref{dffvv}, for a constant ${M}_1$ depending on $\pi$ and for $(m-1)t_1 \leq t \leq mt_1$,
			\begin{align}\label{XII_m}
				\begin{split}
					&  \vert \int_{-r}^{0}{\xi}_{\theta+t}^m\pi(\mathrm{d}\theta)\vert\\ &\quad\leq {M_1}\left(\Vert\xi\Vert_{\mathscr{D}_{\mathbf{X}}^{\beta}([-r,0])}+ t_{1}^\beta P_{1}(\Vert\xi\Vert_{\mathscr{D}_{\mathbf{X}}^{\beta}([-r,0])},\Vert\mathbf{X}\Vert_{\gamma,[0,r]})+\sup_{1\leq i<m}\tilde{p}_{i}(it_1)+\tilde{p}_{m}(t)\right).
				\end{split}
			\end{align} 
			Therefore, based on our selection of \( t_1 \) in \eqref{t_1} and the bounds given in \eqref{dffvv}, we have:
			\begin{align}\label{delly}
				\vert \int_{-r}^{0}{\xi}_{\theta+t}^m\pi(\mathrm{d}\theta)\vert\leq M_1\big(1+\Vert\xi\Vert_{\mathscr{D}_{\mathbf{X}}^{\beta}([-r,0])}+\sup_{0\leq i<m}\tilde{p}_{i}(it_1)+\tilde{p}_{m}(t)\big).
			\end{align}
			From \eqref{dffvv} and \eqref{flow-decomposition_m} for a constant $M_2$ (which depends on the $\Vert A\Vert$)
			\begin{align}\label{BBVA}
				\begin{split}
					&	\sup_{(m-1)t_1 \leq \sigma \leq t}\vert p_m(\sigma)-p_{m}((m-1)t_1)\vert\leq \int_{(m-1)t_1}^{t}M_2\sup_{\substack{y \in \mathbb{R}^n}} \left\| \left(\frac{\partial \tilde{\phi}_{\mathbf{X}}}{\partial y}((m-1)t_1, \tau, y, \xi)\right)^{-1}\right\Vert \\&\quad\times \left(t_1^\beta P_{1}(\Vert\xi\Vert_{\mathscr{D}_{\mathbf{X}}^{\beta}([-r,0])},\Vert\mathbf{X}\Vert_{\gamma,[0,r]})+\tilde{p}_{m}(\tau)+\vert \int_{-r}^{0}{\xi}_{\theta+\tau}^m\pi(\mathrm{d}\theta)\vert\right)\mathrm{d}\tau
				\end{split}
			\end{align}
			Note that
			\begin{align*}
				\tilde{p}_{m}(t)-p_{m}((m-1)t_1)\leq \sup_{(m-1)t_1 \leq \sigma \leq t}\Vert p_m(\sigma)-p_{m}((m-1)t_1)\Vert
			\end{align*}
			Consequently, from \eqref{t_1}, \eqref{inverse22}, \eqref{delly}, and \eqref{BBVA}, and considering a constant \( M_3 \), the following inequalities hold:
			\begin{itemize}
				\item For \( m \geq 2 \) and \( (m-1)t_1 \leq t \leq mt_1 \):
				\begin{align}\label{REEDDSSAA}
					\tilde{p}_{m}(t) - p_{m}((m-1)t_1) \leq M_3 \int_{(m-1)t_1}^{t} \left( 1 + \|\xi\|_{\mathscr{D}_{\mathbf{X}}^{\beta}([-r,0])} + \sup_{1 \leq i < m} \tilde{p}_{i}(it_1) + \tilde{p}_{m}(\tau) \right) \mathrm{d}\tau.
				\end{align}
				\item 	For \( m = 1 \) and \( 0 \leq t \leq t_1 \):
				\begin{align}\label{RERT}
					\tilde{p}_{1}(t) - \|\xi_0\| \leq M_3 \int_{0}^{t} \left( 1 + \|\xi\|_{\mathscr{D}_{\mathbf{X}}^{\beta}([-r,0])} + \tilde{p}_{1}(\tau) \right) \mathrm{d}\tau.
				\end{align}
			\end{itemize}
			From Gronwall's Lemma applied to the inequality in \eqref{RERT}
			\begin{align}\label{HHGG}
				\tilde{p}_{1}(t_1)\leq \vert\xi_0\vert\exp(M_3t_1)+\big(1+\Vert\xi\Vert_{\mathscr{D}_{\mathbf{X}}^{\beta}([-r,0])}\big)(\exp(M_3t_1)-1).
			\end{align}
			For $m\geq 2$, first we set
			\begin{align}\label{aloh_m}
				\alpha_m(t)= p_{m}((m-1)t_1)+M_3\left(1+\Vert\xi\Vert_{\mathscr{D}_{\mathbf{X}}^{\beta}([-r,0])}+\sup_{1\leq i<m}\tilde{p}_{i}(it_1)\right)\left(t-(m-1)t_1\right).
			\end{align}
			Therefore, by substituting into \eqref{REEDDSSAA}, we obtain
			\begin{align*}
				\tilde{p}_{m}(t)\leq \alpha_{m}(t)+M_3\int_{(m-1)t_1}^{t}\tilde{p}_{m}(\tau)\mathrm{d}\tau.
			\end{align*}
			From Gronwall's Lemma
			\begin{align*}
				\tilde{p}_{m}(t)\leq \alpha_m(t)+\int_{(m-1)t_1}^{t}M_3\alpha_m(s)\exp(M_3(t-s))\mathrm{d}s .
			\end{align*}
			Consequently, from \eqref{aloh_m}
			\begin{align}\label{VBVNN}
				\begin{split}
					&\tilde{p}_{m}(mt_1)\leq p_{m}((m-1)t_1)+M_3 t_1\big(1+\Vert\xi\Vert_{\mathscr{D}_{\mathbf{X}}^{\beta}([-r,0])}+\sup_{1\leq i<m}\tilde{p}_{i}(it_1)\big)\\& + p_{m}((m-1)t_1)(\exp(M_3t_1)-1)+\big(1+\Vert\xi\Vert_{\mathscr{D}_{\mathbf{X}}^{\beta}([-r,0])}+\sup_{1\leq i<m}\tilde{p}_{i}(it_1)\big)(\exp(M_3t_1)-1-M_3t_1)\\
					&\quad	=p_{m}((m-1)t_1)\exp(M_3t_1)+\big(1+\Vert\xi\Vert_{\mathscr{D}_{\mathbf{X}}^{\beta}([-r,0])}+\sup_{1\leq i<m}\tilde{p}_{i}(it_1)\big)(\exp(M_3t_1)-1).
				\end{split}
			\end{align}
			Also, from and \eqref{dffvv} and \eqref{t_1}
			\begin{align*}
				\Vert p_m((m-1)t_1)\Vert=\Vert\tilde{\phi}_{\mathbf{X}}\big((m-2)t_1,(m-1)t_1,p_{m-1}((m-1)t_1),\xi\big)\Vert\leq 1+\tilde{p}_{m-1}((m-1)t_1).
			\end{align*}
			Therefore, by substituting the latter inequality into \eqref{VBVNN}, we conclude that
			\begin{align*}
				&\tilde{p}_{m}(mt_1)\leq 2\exp(M_3t_1)-1+\Vert\xi\Vert_{\mathscr{D}_{\mathbf{X}}^{\beta}([-r,0])}(\exp(M_3t_1)-1)\\&\quad+ \tilde{p}_{m-1}((m-1)t_1)\exp(M_3t_1)+\sup_{1\leq i<m}\tilde{p}_{i}(it_1)(\exp(M_3t_1)-1).
			\end{align*}
			This yields
			\begin{align*}
				&\sup_{1\leq i<m+1} \tilde{p}_{i}(it_1)\leq 2\exp(M_3t_1)-1+\Vert\xi\Vert_{\mathscr{D}_{\mathbf{X}}^{\beta}([-r,0])}(\exp(M_3t_1)-1)\\&\quad+(2\exp(M_3t_1)-1)\sup_{1\leq i<m}\tilde{p}_{i}(it_1)=R_{1}+R_{2}\sup_{1\leq i<m}\tilde{p}_{i}(it_1).
			\end{align*}
			By applying Gronwall's Lemma in the discrete case, we obtain
			\begin{align*}
				\sup_{1\leq i<m+1} \tilde{p}_{i}(it_1)\leq R_{1}\frac{R_{2}^{m-1}-1}{R_{2}-1}+R_{2}^{m-1}\tilde{p}_{1}(t_1).
			\end{align*}
			By substituting the values of \( R_1 \) and \( R_2 \), along with the fact that \( \frac{1}{R_{2}-1} \leq \frac{1}{2M_3t_1} \) and using \eqref{HHGG}, we obtain
			% $\sup_{t> 0}(2\exp(M_{3}t)-1)^{\frac{1}{t}}=\exp(2M_{3})$ %
			\begin{align}\label{NNBBNNN}
				\begin{split}
					&\sup_{1\leq i<m+1} \tilde{p}_{i}(it_1)\leq \frac{2\exp(M_3t_1)-1+\Vert\xi\Vert_{\mathscr{D}_{\mathbf{X}}^{\beta}([-r,0])}(\exp(M_3t_1)-1)}{2M_{3}t_1}\big((2\exp(M_3t_1)-1)^{m-1}-1\big)\\&\quad+ (2\exp(M_3t_1)-1)^{m-1} \big(\vert\xi_0\vert\exp(M_3t_1)+\big(1+\Vert\xi\Vert_{\mathscr{D}_{\mathbf{X}}^{\beta}([-r,0])}\big)(\exp(M_3t_1)-1)\big).
				\end{split}
			\end{align}
			Note that: \(\sup_{t > 0}\left(2\exp(M_3 t) - 1\right)^{\frac{1}{t}} = \exp(2M_3)\). Therefore, based on our selection of \(t_1\) and \(N\) in \eqref{t_1}, and considering the fact that \(m \leq N\), as well as \eqref{NNBBNNN}, we conclude that \(\sup_{1 \leq m \leq N} \tilde{p}_m(m t_1)\) can be bounded by a polynomial. Finally, through simple concatenation, we can conclude that the solution to equation \eqref{equation+drift} for \((m-1)t_1 \leq t \leq \min(r, mt_1)\) with \(1 \leq m \leq N\) is given by
			\begin{align}\label{VBBZAw}
				\phi^{t}_{\mathbf{X}}(\xi)=\tilde{\phi}_{\mathbf{X}}((m-1)t_1,t,p_{m}(t),\xi).
			\end{align} 
			Also, from \eqref{dffvv} for $t\in [(m-1)t_1,mt_1]$, we have 
			\begin{align}\label{VBZCXSE}
				\vert\tilde{\phi}_{\mathbf{X}}((m-1)t_1,t,p_{m}(t),\xi)\vert\leq \tilde{p}_{m}(mt_1)+ t_{1}^{\beta}P_{1}(\Vert\xi\Vert_{\mathscr{D}_{\mathbf{X}}^{\beta}([-r,0])},\Vert\mathbf{X}\Vert_{\gamma,[0,r] }).
			\end{align}
			Since \(\sup_{1 \leq m \leq N} \tilde{p}_m(m t_1)\) is dominated by a polynomial, from \eqref{VBBZAw} and \eqref{VBZCXSE}, we conclude that \(\sup_{t \in [0, r]} \Vert \phi^{t}_{\mathbf{X}}(\xi) \Vert\) also has the same a priori bound; that is, it is almost polynomially bounded.\smallskip
			
			\textbf{Step 2.} In this step, we obtain the same type of a priori bound (polynomially) for the $\beta$-H\"older norm of the Gubinelli derivative. Since the Gubinelli derivative of $\phi^{\cdot}_{\mathbf{X}}(\xi)$ is equal to $G(\phi^{\cdot}_{\mathbf{X}}(\xi),\xi_{\cdot-r})$, and based on our assumption about $G$, to guarantee that $\|G(\phi^{\cdot}_{\mathbf{X}}(\xi),\xi_{\cdot-r})\|_{\beta,[0,r]}$ has a polynomial bound, it is sufficient to obtain the same type of bound for $\|\phi^{\cdot}_{\mathbf{X}}(\xi)\|_{\beta,[0,r]}$. \smallskip
			
			Assume $s,t\in [(m-1)t_1,mt_1]$.  Using the flow property \eqref{FFLLOOWW},
			\begin{align*}
				&\left\vert\phi^{t}_{\mathbf{X}}(\xi)-\phi^{s}_{\mathbf{X}}(\xi)\right\vert=\left\vert\tilde{\phi}_{\mathbf{X}}((m-1)t_1,t,p_{m}(t),\xi)-\tilde{\phi}_{\mathbf{X}}((m-1)t_1,s,p_{m}(s),\xi) \right\vert\\&\quad\leq \left\vert\tilde{\phi}_{\mathbf{X}}((m-1)t_1,t,p_{m}(t),\xi)-\tilde{\phi}_{\mathbf{X}}((m-1)t_1,t,p_{m}(s),\xi)  \right\vert\\&\quad +\left\vert \tilde{\phi}_{\mathbf{X}}((m-1)t_1,t,p_{m}(s),\xi)-\tilde{\phi}_{\mathbf{X}}((m-1)t_1,s,p_{m}(s),\xi)\right\vert\\&\qquad=\underbrace{\left\vert\int_{0}^{1}\frac{\partial\tilde{\phi}_{\mathbf{X}} }{\partial y}\big((m-1)t_1,t,\theta p_{m}(t)+(1-\theta)p_{m}(s),\xi\big)(p_{m}(t)-p_{m}(s))\mathrm{d}\theta\right\vert}_{\text{I}(s,t)}\\&\myquad[3]+\underbrace{\left\vert\tilde{\phi}_{\mathbf{X}}\big(s,t,\tilde{\phi}_{\mathbf{X}}((m-1)t_1,s,p_{m}(s),\xi),\xi\big)-\tilde{\phi}_{\mathbf{X}}((m-1)t_1,s,p_{m}(s),\xi)\right\vert}_{\text{II}(s,t)}.
			\end{align*}
			Note that
			\begin{align*}
				\sup_{s, t \in [(m-1)t_1, mt_1]} \frac{\text{I}(s,t)}{(t-s)^{\beta}} &\leq \sup_{y \in \mathbb{R}^n} \left\lvert \frac{\partial \tilde{\phi}_{\mathbf{X}}}{\partial y}((m-1)t_1, t, y, \xi) \right\rvert \frac{\left\lvert p_{m}(t) - p_{m}(s) \right\rvert}{(t-s)^{\beta}} \\
				&\leq r^{1-\beta} \sup_{y \in \mathbb{R}^n} \left\lvert \frac{\partial \tilde{\phi}_{\mathbf{X}}}{\partial y}((m-1)t_1, t, y, \xi) \right\rvert \sup_{\tau \in [(m-1)t_1, mt_1]} \left\lvert \frac{\mathrm{d} p_m(\tau)}{\mathrm{d} \tau} \right\rvert.
			\end{align*}
			Recall that \( p_{m} \) is defined in \eqref{flow-decomposition_m}. Using our conclusions about the a priori bounds for \( \sup_{1 \leq i \leq N} \tilde{p}_{i}(i t_i) \) and \( \sup_{t \in [0, r]} \vert \phi^{t}_{\mathbf{X}}(\xi) \vert \), and employing \eqref{dffvv}, \eqref{inverse22}, \eqref{flow-decomposition_m}, \eqref{XII_m}, and \eqref{VBBZAw}, we conclude that
			\[
			\sup_{\tau \in [(m-1)t_1, mt_1]} \left\lvert \frac{\mathrm{d} p_m(\tau)}{\mathrm{d} \tau} \right\rvert
			\]
			grows almost polynomially (independent of \( m \)). Therefore, from \eqref{inverse22}, we can obtain that
			\[
			\sup_{s, t \in [(m-1)t_1, mt_1]} \frac{\text{I}(s,t)}{(t-s)^{\beta}},
			\]
			is dominated by a polynomial that is independent of \( m \). Obtaining a polynomial a priori bound for \( \frac{\text{II}(s,t)}{(t-s)^{\beta}} \) follows from \eqref{dffvv}. Thus, we have shown that \( \sup_{1 \leq m \leq N} \Vert \phi^{\cdot}_{\mathbf{X}}(\xi) \Vert_{\beta, [(m-1)t_1, \min(mt_1, r)]} \) is dominated by a polynomial. Since \( N \) also has similar growth as obtained in \eqref{t_1}, using the fact that
			\begin{align*}
				\Vert \phi^{\cdot}_{\mathbf{X}}(\xi) \Vert_{\beta, [0, r]} \leq \sum_{1 \leq m \leq N} \Vert \phi^{\cdot}_{\mathbf{X}}(\xi) \Vert_{\beta, [(m-1)t_1, \min(mt_1, r)]},
			\end{align*}
			We can obtain the same type of bound (polynomially) for \( \Vert \phi^{\cdot}_{\mathbf{X}}(\xi) \Vert_{\beta, [0, r]} \). Note that since the Gubinelli derivative of \(\phi^{\cdot}_{\mathbf{X}}(\xi)\) is equal to \(G(\phi^{\cdot}_{\mathbf{X}}(\xi), \xi_{\cdot-r})\), this directly yields the boundedness of the derivative, which depends only on \(G\).
			\smallskip
			
			\textbf{Step 3.} In this step, we show that the remainder of the solution also grows almost polynomially. The remainder term is given by
			\begin{align*}
				&(\phi_{\mathbf{X}}(\xi))^{\#}_{s,t}:=\phi^{t}_{\mathbf{X}}(\xi)-\phi^{s}_{\mathbf{X}}(\xi)-G(\phi^{s}_{\mathbf{X}}(\xi),\xi_{s-r})X_{s,t}.
				%\\&=\int_{s}^t A\big(\phi_{\mathbf{X}}^{\tau}(\xi),\int_{-r}^{0}\phi_{\mathbf{X}}^{\theta+\tau}(\xi)\pi(\mathrm{d}\theta)\big)\mathrm{d}\tau+\int_{s}^{t}V \big(\phi_{\mathbf{X}}^{\tau}(\xi),\xi_{\tau-r})\mathrm{d}\mathbf{X}_{\tau}-G(\phi^{s}_{\mathbf{X}}(\xi),\xi_{s-r})X_{s,t},
			\end{align*}
			%	where $\phi_{\mathbf{X}}^{\theta+\tau}(\xi)=\xi_{\theta+\tau}$ for $\theta+\tau<0$. 
			Assume first $s,t\in [(m-1)t_1,mt_1]$, then
			\begin{align*}
				&\Vert(\phi_{\mathbf{X}}^{.}(\xi))^{\#}_{s,t}\Vert\\&=\Vert\tilde{\phi}_{\mathbf{X}}((m-1)t_1,t,p_{m}(t),\xi)-\tilde{\phi}_{\mathbf{X}}((m-1)t_1,s,p_{m}(s),\xi)-G(\tilde{\phi}_{\mathbf{X}}((m-1)t_1,s,p_{m}(s),\xi),\xi_{s-r})X_{s,t}\Vert\\&\leq \underbrace{\Vert\tilde{\phi}_{\mathbf{X}}((m-1)t_1,t,p_{m}(t),\xi)-\tilde{\phi}_{\mathbf{X}}((m-1)t_1,t,p_{m}(s),\xi)\Vert}_{\text{{III}}(s,t)}\\ &+\underbrace{\Vert \tilde{\phi}_{\mathbf{X}}((m-1)t_1,t,p_{m}(s),\xi)-\tilde{\phi}_{\mathbf{X}}((m-1)t_1,s,p_{m}(s),\xi)-G(\tilde{\phi}_{\mathbf{X}}((m-1)t_1,s,p_{m}(s),\xi),\xi_{s-r})X_{s,t}\Vert}_{\text{{IV}}(s,t)}.
			\end{align*}
			Note that
			\begin{align}\label{BNZUJS}
				\begin{split}
					&\frac{\text{{III}}(s,t)}{(t-s)^{2\beta}}\leq \sup_{y\in\mathbb{ R}^n}\left\vert \frac{\partial \tilde{\phi}_{\mathbf{X}}}{\partial y}((m-1)t_1,t,y,\xi)\right\vert\frac{\vert p_{m}(t)-p_{m}(s)\vert}{(t-s)^{2\beta}}\\&\leq r^{1-2\beta}\sup_{y\in\mathbb{ R}^n}\left\vert \frac{\partial \tilde{\phi}_{\mathbf{X}}}{\partial y}((m-1)t_1,t,y,\xi)\right\vert\sup_{\tau\in [(m-1)t_1,mt_1]}\left\vert \frac{\mathrm{d}p_m(\tau)}{\mathrm{d}\tau}\right\vert.
				\end{split}
			\end{align}
			So it is enough to repeat the same argument that we made in the previous step for estimating \(\sup_{s, t \in [(m-1)t_1, mt_1]} \frac{\text{I}(s,t)}{(t-s)^{\beta}}\), to obtain a polynomial bound (independent of \( m \)) for
			\[
			\sup_{s, t \in [(m-1)t_1, mt_1]} \frac{\text{III}(s,t)}{(t-s)^{\beta}}.
			\]
			\begin{comment}
				
				\begin{align*}
					&\sup_{y\in\mathbb{ R}^n}\Vert \frac{\partial \tilde{\phi}_{\mathbf{X}}}{\partial y}((m-1)t_1,t,y,\xi)\Vert\sup_{\tau\in [(m-1)t_1,mt_1]}\vert \frac{\mathrm{d}p_m(\tau)}{\mathrm{d}\tau}\vert\\&\leq L(\Vert A\Vert,\pi)\big(\Vert\xi\Vert_{\mathscr{D}_{\mathbf{X}}^{\beta}([-r,0])}+\underbrace{\sup_{\tau\in [(m-1)t_1,mt_1]}\Vert\tilde{\phi}_{\mathbf{X}}((m-1)t_1,\tau,p_{m}(t),\xi)\Vert}_{\sup_{\tau\in [(m-1)t_1,mt_1]}\Vert\phi^{t}_{\mathbf{X}}(\xi)\Vert} \big)
				\end{align*}
			\end{comment}
			Also, from \eqref{dffvv}
			\begin{align*}
				\text{{IV}}(s,t)\leq (t-s)^{2\beta} P_{1}(\Vert\xi\Vert_{\mathscr{D}_{\mathbf{X}}^{\beta}([-r,0])},\Vert\mathbf{X}\Vert_{\gamma,[0,r] }).
			\end{align*}
			Therefore, we conclude that the remainder term for \( s, t \in [(m-1)t_1, mt_1] \) can be bounded polynomially. To obtain the final estimate on \([0, r]\), we proceed as in the previous step. Indeed, since for \( s < u < t \), we have
			\[
			(\phi_{\mathbf{X}}(\xi))^{\#}_{s,t} = (\phi_{\mathbf{X}}(\xi))^{\#}_{s,u} + (\phi_{\mathbf{X}}(\xi))^{\#}_{u,t} + (\phi_{\mathbf{X}}(\xi))^{\#}_{s,u}X_{u,t},
			\]
			we can easily conclude that
			\begin{align*}
				\Vert\phi_{\mathbf{X}}(\xi)^{\#}\Vert_{2\beta,[0,r]}\leq (1+\Vert X \Vert_{\beta,[0,r]})\sum_{1\leq m\leq N}\Vert\phi_{\mathbf{X}}(\xi)^{\#}\Vert_{\beta,[(m-1)t_1,\min(mt_1,r)]}.
			\end{align*}
			As in the previous step from \( t_1 \), we know that \( N \) grows polynomially. This completes our argument for estimating the remainder term.
			\smallskip
			
			\textbf{Step 4.}
			In the previous steps, we proved the well-posedness of the solution and also found polynomial bounds for the elements of the norm \(\|\cdot\|_{\mathscr{D}_{X}^{\beta}}\) as defined in \eqref{FAS}. As argued, all of them are bounded by a polynomial, which yields \eqref{BUN}.
			So far, we have solved the equation up to time \( r \); however, we can easily use \((\varphi_{\mathbf{X}}^{t}(\xi))_{0\leq t\leq r}\) as a new initial value and solve the equation for \( r \leq t \leq 2r \), repeating the same procedure to obtain a global solution. Additionally, the continuity of the map defined in \eqref{STTAASX} can be easily derived from the method used in \eqref{flow-decomposition_m} to construct the solution. Indeed, since the construction method relies on a standard fixed-point argument over a short interval and then extends the solution to a larger interval, the continuity of the map \eqref{STTAASX} is evident. This continuity was even established (though not explicitly mentioned) during the first step.
		\end{proof}
		As mentioned in Theorem \ref{SSS}, our solution to equation \eqref{SSS} exhibits a polynomial growth with respect to the initial value and noise. This polynomial growth is crucial for our future purposes. In \cite[Theorem 2.10]{GVR21}, a similar result is obtained under the assumption that \( \pi \) is a Dirac measure. Our result in Theorem \ref{SSS} is more general as it does not impose such an assumption on \( \pi \) and involves more technical considerations. Establishing this growth allows us to derive two key bounds for the linearized equations, which are instrumental in applying the Multiplicative Ergodic Theorem. This, in turn, enables us to deduce the existence of invariant manifolds from abstract results.	Unlike Theorem \ref{SSS}, which necessitated numerous technical steps and subtle refinements beyond those found in \cite[Theorem 1.10]{GVR21}, the proofs of the next two theorems are quite similar to the case where $\pi$ is a Dirac measure, as investigated in \cite{GVR21}. The most challenging aspect of proving the upcoming results lies in ensuring that our solution exhibits the polynomial growth established in Theorem \ref{SSS}. Consequently, we omit the proofs, as they involve only minor modifications.
		%We will use $y^{\xi}=(\phi^{t}_{\mathbf{X}}(\xi,0\,y_0))_{0\leq t\leq r}$ to denote the solution of \eqref{equation+drift}. note that, $y^{\xi}\in\mathscr{D}_{\mathbf{X}}^{\gamma}([0,r]$.
		\begin{theorem}\label{DCXSZ}
			Assume the same hypotheses as in Theorem \ref{SSS}. Then $\phi_{\mathbf{X}}(\xi)$ is is Fr\'echet differentiable (with respect to $\xi$). In addition, for 
			$D_{\xi}\phi_{\mathbf{X}}:\mathscr{D}_{\mathbf{X}}^{\gamma}([0,r]\rightarrow \mathscr{D}_{\mathbf{X}}^{\gamma}([0,r] $, we can find a polynomial $Q_1$ such that
			\begin{align*}
				&\Vert D_{\xi}\phi_{\mathbf{X}}[\zeta]\Vert_{\mathscr{D}_{\mathbf{X}}^{\gamma}([0,r]}\leq \exp\big(Q_{1}(\Vert\xi\Vert_{\mathscr{D}_{\mathbf{X}}^{\gamma}([-r,0])},\Vert\mathbf{X}\Vert_{\gamma,[0,r]})\big),
			\end{align*}
			for every $\xi,\zeta \in\mathscr{D}_{\mathbf{X}}^{\gamma}([-r,0])$.
		\end{theorem}
		\begin{proof}
			Similar to the proof of \cite[Theorem 2.12]{GVR21}, with some minor modifications.
		\end{proof}
		We also have:
		\begin{theorem}\label{ASXXZA}
			Assume the same hypotheses as in Theorem \ref{SSS}. Then in case $A\neq 0$ (and therefore $G\in C^{4}_{b}$), there exists a polynomial $Q_2$ such that
			\begin{align*}
				&\Vert \Vert D_{\xi}\phi_{\mathbf{X}}[\zeta]- D_{\tilde\xi}\phi_{\mathbf{X}}[\zeta]\Vert_{\mathscr{D}_{\mathbf{X}}^{\gamma}([0,r]}\leq \\ &\quad\Vert \xi-\tilde{\xi}\Vert_{\mathscr{D}_{\mathbf{X}}^{\gamma}([-r,0]}\Vert\zeta\Vert_{\mathscr{D}_{\mathbf{X}}^{\gamma}([-r,0]}\exp\big(Q_{2}(\Vert\xi\Vert_{\mathscr{D}_{\mathbf{X}}^{\gamma}([-r,0])},\Vert\tilde{\xi}\Vert_{\mathscr{D}_{\mathbf{X}}^{\gamma}([-r,0]},\Vert\mathbf{X}\Vert_{\gamma,[0,r]})\big),
			\end{align*}
			for every \( \xi, \tilde{\xi}, \zeta \in \mathscr{D}_{\mathbf{X}}^{\gamma}([-r,0]) \).
			In case \( A=0 \), and additionally, if for some \( 0 < \vartheta \leq 1 \), \( G \in C^{3+\vartheta}_{b} \), then there exists a polynomial \( Q_3 \) such that
			\begin{align*}
				&\Vert \Vert D_{\xi}\phi_{\mathbf{X}}[\zeta]- D_{\tilde\xi}\phi_{\mathbf{X}}[\zeta]\Vert_{\mathscr{D}_{\mathbf{X}}^{\gamma}([0,r]}\leq \\ &\quad\Vert \xi-\tilde{\xi}\Vert_{\mathscr{D}_{\mathbf{X}}^{\gamma}([-r,0]}^{\vartheta}\Vert\zeta\Vert_{\mathscr{D}_{\mathbf{X}}^{\gamma}([-r,0]}\exp\big(Q_{3}(\Vert\xi\Vert_{\mathscr{D}_{\mathbf{X}}^{\gamma}([-r,0])},\Vert\tilde{\xi}\Vert_{\mathscr{D}_{\mathbf{X}}^{\gamma}([-r,0])},\Vert\mathbf{X}\Vert_{\gamma,[0,r]})\big),
			\end{align*}
			for every \( \xi, \tilde{\xi}, \zeta \in \mathscr{D}_{\mathbf{X}}^{\gamma}([-r,0]) \).
		\end{theorem}
		\begin{proof}
			Similar to the proof of \cite[Theorem 2.13]{GVR21}, with some minor modifications.
		\end{proof}
		\begin{remark}
			In Theorem \ref{SSS}, when \( A \neq 0 \), we assumed that \( G \in C^4_b \). This assumption is not optimal, and we believe it can be weakened to \( C^3_b \). However, the reason we initially made this assumption was primarily to obtain a polynomial a priori bound. In the case where \( G \in C^3_b \), the method we used (i.e., decomposition of flows) is not applicable. Alternatively, we might switch to the semigroup approach. While the semigroup approach could establish well-posedness under weaker assumptions, the bound may not be polynomial, and obtaining an integrable bound could require more sophisticated techniques such as those involving greedy points (cf. \cite{CLL13, GVR23C}). We wanted to avoid this complexity and make our article more straightforward and readable.
		\end{remark}
		%Now we prove how our center manifold theorem result can be applied to the general delay equations governed by \eqref{equation+drift}. We worked with delay equations admitting discrete delays driven by Brownian motions in \cite{GVR23} and \cite{GVR21}. In this part, besides our new center manifold claim, we argue how all results in these two papers can be generalized to the fractional Brownian motions(do it in this paper????). Let us first start with a definition
		\section{Random dynamical systems and invariant manifolds}\label{sec:RDS}
		In this section, we establish the connection between stochastic delay equations and Arnold's concept of a random dynamical system (RDS). Indeed, due to the deterministic nature of rough path theory and the processes involved in our problem, the RDS provides a natural framework for studying the dynamics of the solutions. We specifically address delay equations with fractional Brownian motions, defining integrals using rough path theory, and describing the delayed Levy area as given by \eqref{DElY}. After proving that our equations generate a random dynamical system, we will demonstrate the existence of invariant manifolds. In the next section, we specifically focus on stable manifolds and explore how our theory can yield several insightful stability results. For the remainder of this section, we assume that Assumption \ref{FFFDDCCA} holds.
		We can now prove that our fractional Brownian motion for every \(\frac{1}{3} < \gamma < H\) can be lifted to a  a delayed \(\gamma\)-rough path. We begin with the following lemma.
		\begin{lemma}\label{FRFR}. 
			For \( s, t \in \mathbb{R} \), set
			\[
			\mathbb{B}_{s,t}(-r) := \int_{s}^{t} B_{s-r,u-r} \otimes \mathrm{d}^{\circ} B_{u}.
			\] 
			Then for every closed interval \( J \subset \mathbb{R} \), on a set of full measure \( \tilde{\Omega} \subset \Omega \), we can obtain a delayed \( \gamma \)-rough path \( \overline{\mathbf{B}}(\omega) = (B(\omega), \mathbb{B}(\omega), \mathbb{B}(-r)(\omega)) \), such that \( (B(\omega), \mathbb{B}(\omega)) \) is a geometric rough path for every \( \omega \in \tilde{\Omega} \).
		\end{lemma}
		\begin{proof}
			This result already appears in \cite{NNT08}, but for the reader's convenience, we repeat the main ideas here. It is known that we can lift \( B \) to a geometric rough path; cf. \cite[Corollary 10.10]{FH20}. Also, from \eqref{AAA_1} and \eqref{BBB_1}, we have  
			\[
			\sup_{s, t \in J} \mathbb{E} \vert \mathbb{B}_{s,t}(-r) \vert^2 \leq C (t-s)^{4H},
			\]
			where \(C > 0\) is a constant and \(H\) is the Hurst parameter of the fractional Brownian motion \(B\). Since \( \mathbb{B}(-r) \) is an element in the second Wiener chaos, this bound holds in the \( L^q \)-norm for any \( q \geq 1 \); see \cite[Proposition 15.22]{FV10}. Thus, we can apply the Kolmogorov–Chentsov theorem; cf. \cite[Theorem 3.1]{FH20}, to conclude that, on a set of full measure \( \tilde{\Omega} \), for the process \( \mathbb{B}(-r) \) and for every closed interval \( J \):
			\[
			\sup_{s,t \in J} \frac{|\mathbb{B}(-r)|}{(t-s)^{2\gamma}} < \infty,
			\]
			where \( \gamma \in (0, \frac{1}{2}) \) depends on \(H\).
			This completes our proof.
		\end{proof}
		\begin{remark}\label{REDDFDF}
			With a similar argument, we can also define \(\mathbb{B}\) as a stochastic process like \(\mathbb{B}(-r)\) using the symmetric integral (cf. Definition \ref{symmetric_integral}). Since \(\mathbb{B}\) is geometric, this definition coincides with the one introduced in \cite{FH20}[Corollary 10.10]. The proofs are analogous to those we provided for \(\mathbb{B}(-r)\). Therefore, we do not repeat the arguments, as only slight changes are required and the proofs are even simpler. For this process, similar to \(\mathbb{B}(-r)\), the following identity holds:
			\begin{align}\label{AAA_9}
				\begin{split}
					&\mathbb{B}_{s,t}:=\int_{s}^{t}B_{s,u}\otimes \mathrm{d}^{\circ} B_{u}\\&\quad=\sum_{1\leq i,j\leq d}\left(\delta^{B^{i}}(B^{j}_{s,.}\chi_{[s,t]}(.))+\delta_{i=j}\left(\frac{1}{2}((t-s)^{2H})\right)\right)e_i\otimes e_j.
				\end{split}
			\end{align}
			Like \(\mathbb{B}(-r)\), we can apply the Kolmogorov-Chentsov theorem to extend this process to the L\'evy area. In addition, from \eqref{key}, for every \(\gamma < H\), we have
			\begin{align}\label{JHNBZS}
				\Vert \mathbf{B} \Vert_{\gamma, [0,r]} \in \bigcap_{p \geq 1} L^{p}(\Omega).
			\end{align}
		\end{remark}
		The interesting aspect of our delayed rough path is its ability to generate an ergodic random dynamical system. Two key features that enable us to demonstrate this are: first, the fact that fractional Brownian motions possess stationary increments; and second, the observation that the Lévy area and the delayed Lévy area can be derived by regularizing the path in a limiting process. First, we need some technical definitions.
		
		\begin{definition}
			Assume $(\Omega, \mathcal{F}, \mathbb{P}, (\theta_t)_{t \in \mathbb{R}})$ is a measurable metric dynamical system and $r > 0$. A \emph{delayed $\gamma$-rough path cocycle $\mathbf{X}$ (with delay $r > 0$)} is a delayed $\gamma$-rough path-valued stochastic process $\mathbf{X}(\omega) = (X(\omega), \mathbb{X}(\omega), \mathbb{X}(-r)(\omega))$ that satisfies
			\begin{align}\label{eqn:cocycle_rough_delay}
				\mathbf{X}_{s,s+t}(\omega) = \mathbf{X}_{0,t}(\theta_s \omega)
			\end{align}
			for every $\omega \in \Omega$ and for all $s, t \in \mathbb{R}$. 
		\end{definition}
		\begin{definition}
			Let
			\begin{align*}
				\tilde{T}^{2}(\mathbb{R}^d):=\big{\lbrace}\big{(}1\oplus(\alpha,\beta)\oplus(\gamma,\theta)\big{)}\ \vert\ \alpha ,\beta\in \mathbb{R}^d \ \text{and} \ \gamma ,\theta\in \mathbb{R}^d\otimes \mathbb{R}^d\big{\rbrace}.
			\end{align*}
			We define projections $\Pi_i^j$ by
			\begin{align*}
				\Pi_i^j\big{(}1\oplus(\alpha,\beta)\oplus(\gamma,\theta)\big{)} := \begin{cases}
					\alpha &\text{if } i = 1,\, j = 1 \\
					\beta &\text{if } i = 1,\, j = 2 \\
					\gamma &\text{if } i = 2,\, j = 1 \\
					\theta &\text{if } i = 2,\, j = 2.
				\end{cases}
			\end{align*}
			We also define
			\begin{align*}
				\big{(}1\oplus(\alpha_{1},\beta_{1})\oplus(\gamma_{1},\theta_{1})\big{)}&\circledast \big{(}1\oplus(\alpha_{2},\beta_{2})\oplus(\gamma_{2},\theta_{2})\big{)}:=\\ &\big{(}1\oplus(\alpha_{1}+\alpha_{2},\beta_{1}+\beta_{2})\oplus(\gamma_{1}+\gamma_{2}+\alpha_{1}\otimes\alpha_{2},\theta_{1}+\theta_{2}+\beta_{1}\otimes\alpha_{2})\big{)}
			\end{align*}
			and $\mathbf{1} := (1,(0,0),(0,0))$. In this case, \((\tilde{T}^{2}(U), \circledast)\) is a topological group with identity \(\mathbf{1}\). In particular, a \emph{delayed \(\gamma\)-rough path} takes values in \(\tilde{T}^{2}(\mathbb{R}^d)\), and the cocycle property \eqref{eqn:cocycle_rough_delay} is equivalent to  $ \mathbf{X}_{0,t}(\theta_{s}(\omega))=\mathbf{X}_{0,s}^{-1}(\omega)\circledast\mathbf{X}_{t+s}(\omega) $ for every $s,t \in \R$ and every $\omega \in \Omega$.	
		\end{definition}
		We need another definition.
		\begin{definition}
			Let \( J \subset \mathbb{R} \) with \( 0 \in J \). The space \( \mathcal{C}_0^{0,1-\text{var}}(J, \mathbb{R}^d) \) is defined as the closure of the set of paths \( x \colon J \to \mathbb{R}^d \) that are arbitrarily often differentiable and satisfy \( x_0 = 0 \), under the \( 1 \)-variation norm.
			Let $p\geq 1$, then by \( C_{0}^{0,p-\text{var}}(J, \mathbb{R}^d) \), we denote the set of continuous maps \( \mathbf{x} \colon J \rightarrow \mathbb{R}^d \) such that \( \mathbf{x}_0 = 0 \), and there exists a sequence \( \{x_n\} \) with \( x_n \in \mathcal{C}_0^{0,1-\text{var}}(J, \mathbb{R}^d) \) satisfying
			\[
			d_{p-\text{var}}\big{(}\mathbf{x}, x_n\big{)} := \left( \sup_{\mathcal{P} \subset J} \sum_{t_k \in \mathcal{P}} \left| (\delta \mathbf{x})_{t_k, t_{k+1}} - (\delta x_n)_{t_k, t_{k+1}} \right|^{p} \right)^{\frac{1}{p}} \to 0
			\]
			as \( n \to \infty \), where \( \mathcal{P} \) denotes the set of finite partitions of \( J \). The space \( C_{0}^{0,p-\text{var}}(\mathbb{R}, \mathbb{R}^d) \) consists of all continuous paths \( \mathbf{x} \colon \mathbb{R} \rightarrow \mathbb{R}^d \) for which \( \mathbf{x}|_J \in C_{0}^{0,p-\text{var}}(J, \mathbb{R}^d) \) for every \( J \) as above.
			Similarly, the space \( C_{0}^{0,p-\text{var}}(J, \tilde{T}^{2}(\mathbb{R}^d)) \) consists of continuous maps \( \mathbf{x} \colon J \to \tilde{T}^{2}(\mathbb{R}^d) \) with \( \mathbf{x}_0 = \mathbf{1} \), for which there exists a sequence \( \{x_n\} \) with \( x_n \in \mathcal{C}_0^{0,1-\text{var}}(J, \mathbb{R}^d) \) such that
			\[
			d_{p-\text{var}}\big{(}\mathbf{x}, \tilde{S}_{2}(x_n)\big{)} := \sup_{i,j \in \{1,2\}} \left( \sup_{\mathcal{P} \subset J} \sum_{t_k \in \mathcal{P}} \left| \Pi_{i}^{j}\left(\mathbf{x}_{t_k, t_{k+1}} - \tilde{S}_{2}(x_n)_{t_k, t_{k+1}}\right) \right|^{\frac{p}{i}} \right)^{\frac{1}{p}} \to 0
			\]
			as \( n \to \infty \). Here, \( \mathcal{P} \) represents the set of finite partitions of \( J \), and \( \mathbf{x}_{s,t} := \mathbf{x}_s^{-1} \circledast \mathbf{x}_t \). The space \( C_{0}^{0,p-\text{var}}(\mathbb{R}, \tilde{T}^{2}(\mathbb{R}^d)) \) consists of all continuous paths \( \mathbf{x} \colon \mathbb{R} \to \tilde{T}^{2}(\mathbb{R}^d) \) such that for every interval \( J \subset \mathbb{R} \), the restriction \( \mathbf{x}|_J \) belongs to \( C_{0}^{0,p-\text{var}}(J, \tilde{T}^{2}(\mathbb{R}^d)) \).
		\end{definition}
		\begin{remark}
			A notable feature of \( C_{0}^{0,p-\text{var}}(\R, \mathbb{R}^d) \) and \( C_{0}^{0,p-\text{var}}(\R, \tilde{T}^{2}(\mathbb{R}^d)) \) is that both are Polish spaces.
		\end{remark}
		We are now ready to prove that $\overline{\mathbf{B}}$, has an indistinguishable version which is delayed $\gamma$-rough path cocycle.
		\begin{theorem}\label{thm:B_delayed_cocycle}
			Let us consider the process $\overline{\mathbf{B}}$ defined in Lemma \ref{FRFR}. For any $p \in \left(\frac{1}{3}, H\right)$, there exists an ergodic metric dynamical system $(\Omega, \mathcal{F}, \mathbb{P}, (\theta_t)_{t \in \mathbb{R}})$ and a random variable $\mathbf{B}$ on $\Omega$, taking values in $C_{0}^{0,p\text{-var}}(\mathbb{R}, \tilde{T}^{2}(U))$, which has the same law as $\overline{\mathbf{B}}$. Furthermore, this process $\mathbf{B}$ satisfies the cocycle property given by equation \eqref{eqn:cocycle_rough_delay}.
		\end{theorem}
		\begin{proof}
			First, note that for \(\frac{1}{H} < p < 3\), we can assume that our fractional Brownian motion is defined on the probability space \(\left(C_{0}^{0,p-\text{var}}(\R, \mathbb{R}^d), \hat{\mathcal{F}}, \hat{\mathbb{P}}\right)\), where \(\hat{\mathcal{F}}\) denotes the Borel \(\sigma\)-algebra on \(C_{0}^{0,p-\text{var}}(\R, \mathbb{R}^d)\), and \(\hat{\mathbb{P}}\) is the probability measure on this space that governs the law of our process.
			Let \(B^{\epsilon}\) be the regularized path in Definition \ref{REGU}. Recall that from Proposition \eqref{thm:approx_strat_delay_f}, for
			\begin{align*}
				\mathbf{B}^{\epsilon}_{s,t}:=\bigg{(}1\oplus\big{(}B^\epsilon_{s,t}, B^\epsilon_{s-r,t-r}\big{)}\oplus\big{(}\int_{s}^{t} B^\epsilon_{s,\tau}\otimes \mathrm{d}B^\epsilon_{\tau},\int_{s}^{t} B^\epsilon_{s-r,\tau-r}\otimes \mathrm{d}B^\epsilon_{\tau}\big{)}\bigg{)} \in \tilde{T}^{2}(\mathbb{R}^d),
			\end{align*}
			we have
			\begin{align}\label{FRFF}
				\lim_{\epsilon \rightarrow\infty}\sup_{q\geqslant 1}\frac{\left\|d_{\gamma;J}\big{(} \mathbf{B}^{\epsilon} , \overline{\mathbf{B}} \big{)}\right\|_{L^{q}}}{\sqrt{q}} = 0 .
			\end{align}
			By the Kolmogorov-Chentsov theorem, for every \(\frac{1}{H} < p < 3\), we can deduce that \(\overline{\mathbf{B}}\) has an indistinguishable version \({\mathbf{B}}\) that takes values in \(C_{0}^{0, p\text{-var}}(\mathbb{R}, \tilde{T}^{2}(\mathbb{R}^d))\). Additionally, \(\mathbf{B}^\epsilon\) exhibits stationary increments, meaning the law of the process \((\mathbf{B}_{t_0, t_0 + h}^\epsilon)_{h \in \mathbb{R}}\) is independent of \(t_0 \in \mathbb{R}\). Consequently, from \eqref{FRFF}, we conclude that \(\overline{\mathbf{B}}\) and therefore \({\mathbf{B}}\) also possess stationary increments. Now we set $\Omega:=C_{0}^{0,p-\text{var}}(\R,\tilde{T}^{2}(\mathbb{R}^d)) $, ${\mathcal{F}}$ the corresponding Borel $\sigma$-algebra and $\P$ be the law of $\overline{\mathbf{B}}$. We also define $\theta$ by
			\begin{align*}
				(\theta_{s}\omega)(t):=\omega(s)^{-1}\circledast\omega(t+s).
			\end{align*}
			The rest of the proof regarding the construction of the probability measure follows similarly to \cite[Theorem 5]{BRS17}. It remains to show the ergodicity of \(\theta\). This (informally) follows from the fact that fractional Brownian motion is ergodic under the classical shift map \(\hat{\theta}_s : C_{0}^{0,p-\text{var}}(\R, \mathbb{R}^d) \rightarrow C_{0}^{0,p-\text{var}}(\R, \mathbb{R}^d)\), defined by \(\hat{\theta}_s(\omega)(t) := \omega(t+s) - \omega(s)\). To make our claim regarding the ergodicity precise, for a \(x\in C_{0}^{0,1-\text{var}}(\R, \mathbb{R}^d) \), we define
			\begin{align*}
				S(x) = \bigg{(}1 \oplus \big{(}x_{u,v}, x_{u-r, v-r}\big{)} \oplus \big{(} \int_{u}^{v} x_{u,\tau} \otimes \mathrm{d}x_{\tau}, \int_{u}^{v} x_{u-r, \tau-r} \otimes \mathrm{d}x_{\tau} \big{)} \bigg{)}_{u \leq v}.
			\end{align*}
			Then, clearly, \( S(\hat{\theta}_s x) = \theta_s(S(x)) \). We extend this map to \(C_{0}^{0,p-\text{var}}(\mathbb{R}, \mathbb{R}^d)\), which can be defined as limits of Riemann sums on compact sets in the symmetric sense (cf. Definition \eqref{symmetric_integral}); otherwise, we set \( S(x) = \mathbf{1} \).
			% Note that for Then, clearly, \( S(\hat{\theta}_s x) = \theta_s(S(x)) \). 
			%Assume $\Pi:=\lbrace\frac{k}{2^i}: k\in\Z, i\in \N\cup\lbrace0\rbrace\rbrace$. For every $s\in\Pi$, define
			%	\begin{align*}
				%\hat{\Omega}_s:=\lbrace x\in C_{0}^{0,p-\text{var}}(\R, \mathbb{R}^d): \exists  \( x_n \in \mathcal{C}_0^{0,1-\text{var}}(\R, \mathbb{R}^d)\rbrace
				%	\end{align*}
			Note that from \eqref{FRFF}, for every \( t \in \mathbb{R} \), we can find a subset \( \hat{\Omega}_t \subset C_{0}^{0,p\text{-var}}(\mathbb{R}, \mathbb{R}^d) \) with full measure such that \( S(\hat{\theta}_t x) = \theta_t(S(x)) \) for every \( x \in \hat{\Omega}_t \). Recall that $\left(C_{0}^{0,p\text{-var}}(\mathbb{R}, \mathbb{R}^d), \hat{\mathcal{F}}, \hat{\mathbb{P}}, (\hat{\theta}_t)_{t \in \mathbb{R}} \right)$ is an ergodic dynamical system. Therefore, our ergodicity claim for \( \theta \) follows from \cite[Lemma 3.6]{GVRS22}. See also \cite[Lemma 3]{GS11}.
			\begin{comment}
				In particular, we can take a dense sequence \( (s_{n})_{n \geq 1} \) in \(\mathbb{R}\) and set \(\tilde{\Omega} = \cap_{n \geq 1} \hat{\Omega}_{s_n} \). Then \(\tilde{\Omega}\) has full measure and \( S(\hat{\theta}_{s_n} x) = \theta_{s_n}(S(x)) \) holds for every \( x \in \tilde{\Omega} \).
				Note that since \(\theta = (\theta_s)_{s \in \mathbb{R}}\) and \(\hat{\theta} = (\hat{\theta}_s)_{s \in \mathbb{R}}\) are both continuous (with respect to \(s\) for every path) and \((s_{n})_{n \geq 1}\) is dense, actually \( S(\hat{\theta}_s x) = \theta_s(S(x)) \) holds for every \( x \in \tilde{\Omega} \) and \( s \in \mathbb{R} \).
				Finally, since for every \( s \in \mathbb{R} \), \(\hat{\theta}_s : \tilde{\Omega} \rightarrow \tilde{\Omega}\) is ergodic, and since by construction \(\mathbb{P} = \hat{\mathbb{P}} \circ S^{-1}\), we can conclude the ergodicity of \(\theta\) (cf. \cite[Lemma 3.6]{GVRS22}).
			\end{comment}
		\end{proof}
		Now, we can prove that the solutions of the rough delay equation driven by fractional Brownian motions generate a random dynamical system.
		\begin{theorem}\label{SSRR}
			Assume \((\Omega, \mathcal{F}, \mathbb{P}, \theta)\), \(\frac{1}{3} < \beta \leq \gamma < H\), and let \(\mathbf{B}\) be the delayed \(\gamma\)-rough path cocycle constructed in Theorem \ref{thm:B_delayed_cocycle}. Assume \(\xi \in \mathscr{D}_{\mathbf{B}(\omega)}^{\beta}([-r,0])\) and consider the following equation:
			\begin{align}\label{equation+drift2}
				\begin{split}
					\mathrm{d}y_{t} &= A\left(y_t, \int_{-r}^{0} y_{\theta+t} \pi(\mathrm{d}\theta)\right) \mathrm{d}t + G \left(y_{t}, y_{t-r}\right) \mathrm{d}\mathbf{B}_{t}(\omega), \\
					y_{s} &= \xi_{s} \quad \text{for} \quad -r \leq s \leq 0,
				\end{split}
			\end{align}
			with the same conditions as in Theorem \ref{SSS}. Then this equation admits a unique solution globally in time. Assume \((\phi^{t}_{\mathbf{B}(\omega)}(\xi))_{t \geq 0}\) is this solution. For \(m \in \mathbb{N}\), define
			\[
			\varphi^{m}_{\omega}(\xi) := \left(\phi^{(m-1)r+t}_{\mathbf{B}(\omega)}(\xi)\right)_{0 \leq t \leq r}.
			\]
			In this case, we indeed have \(\varphi^{m}_{\omega}(\xi) \in \mathscr{D}_{\mathbf{B}(\theta_{mr}\omega)}^{\beta}([-r,0])\). Moreover, the following map
			\begin{align}\label{TVBN}
				\begin{split}
					&\varphi^{m}_{\omega}: \mathscr{D}_{\mathbf{B}(\omega)}^{\beta}([-r,0]) \rightarrow \mathscr{D}_{\mathbf{B}(\theta_{mr}\omega)}^{\beta}([-r,0]), \\
					& \xi \longrightarrow \varphi^{m}_{\omega}(\xi),
				\end{split}
			\end{align}
			is continuous and satisfies the cocycle property, i.e.,
			\begin{align}\label{VCCA}
				\varphi^{p+q}_{\omega} = \varphi^{p}_{\theta_{qr} \omega} \circ \varphi^{q}_{\omega},
			\end{align}
			for all \(p, q \in \mathbb{N}\).
		\end{theorem}
		\begin{proof}
			First, note that from Theorem \ref{SSS}, the solution exists. The rest of the proof is almost identical to the proof in \cite[Theorem 3.12]{GVRS22}.
		\end{proof}
		\begin{remark}
			In equation \eqref{equation+drift2}, the diffusion term, denoted by \(G\), can be assumed to be a linear function. Specifically, we can assume that \(G \in L(\mathbb{R}^n \times \mathbb{R}^n, L(\mathbb{R}^d, \mathbb{R}^n))\). Under this assumption, our solution exists globally in time for every initial value \(\xi \in \mathscr{D}_{\mathbf{B}(\omega)}^{\beta}[-r,0]\), and the cocycle property \eqref{VCCA} holds. We do not provide an explicit proof here, as this result is a straightforward modification of Theorem \ref{SSS} and Theorem \ref{SSRR}.
		\end{remark}
		While the collection of spaces \((\mathscr{D}_{\mathbf{B}(\omega)}^{\beta}[-r,0])_{\omega \in \Omega}\) allows us to define a generalized cocycle, these spaces are not separable and thus are not suitable for further investigations. To remedy this problem, we need to slightly modify these spaces and define our cocycle on some invariant subspace of \(\mathscr{D}_{\mathbf{B}(\omega)}^{\beta}[-r,0]\) (under \(\varphi\)). This technique was first carry out in \cite{GVRS22}, for the Brownian motions. Before introducing these modified spaces, we need to recall the following abstract definition:  
		\begin{definition}\label{MESSSAA}
			Let $(\Omega,\mathcal{F})$ be a measurable space. We call a family of Banach spaces $\{E_{\omega}\}_{\omega \in \Omega}$ a \emph{measurable field of Banach spaces} if there is a set of sections
			\begin{align*}
				\Delta \subset \prod_{\omega \in \Omega} E_{\omega}
			\end{align*}
			with the following properties:
			\begin{itemize}
				\item[(i)] $\Delta$ is a linear subspace of $\prod_{\omega \in \Omega} E_{\omega}$.
				\item[(ii)] There is a countable subset $\Delta_0 \subset \Delta$ such that for every $\omega \in \Omega$, the set $\{g(\omega)\, :\, g \in \Delta_0\}$ is dense in $E_{\omega}$.
				\item[(iii)] For every $g \in \Delta$, the map $\omega \mapsto \| g(\omega) \|_{E_{\omega}}$ is measurable.
			\end{itemize}
			\begin{remark}
				Alternatively, when \(\Omega\) is a metric space, we can define a stronger concept called \emph{continuous fields of Banach spaces}. In this case, the third condition in Definition \ref{MESSSAA} must be replaced by:
				\begin{itemize}
					\item[(iii)$^{\prime}$] For every \(g \in \Delta\), the map \(\omega \mapsto \| g(\omega) \|_{E_{\omega}}\) is continuous.
				\end{itemize}
			\end{remark}
		\end{definition} 
		We have the following theorem.
		\begin{theorem}\label{BETA}
			Assume \((\Omega, \mathcal{F}, \mathbb{P}, \theta)\), \(\frac{1}{3} < \beta < \gamma < H\), and let \(\mathbf{B}\) be the delayed \(\gamma\)-rough path cocycle constructed in Theorem \ref{thm:B_delayed_cocycle}. For every \(\omega \in \Omega\), define
			\[
			\mathscr{D}_{\mathbf{B}(\omega)}^{\beta, \gamma}([-r,0]) := \overline{\mathscr{D}_{\mathbf{B}(\omega)}^{\beta}([-r,0])}^{\mathscr{D}_{\mathbf{B}(\omega)}^{\gamma}([-r,0])},
			\]
			where \(\overline{\mathscr{D}_{\mathbf{B}(\omega)}^{\beta}([-r,0])}^{\mathscr{D}_{\mathbf{B}(\omega)}^{\gamma}([-r,0])}\) denotes the closure of \(\mathscr{D}_{\mathbf{B}(\omega)}^{\beta}([-r,0])\) in \(\mathscr{D}_{\mathbf{B}(\omega)}^{\gamma}([-r,0])\). Then \(\{\mathscr{D}_{\mathbf{B}(\omega)}^{\beta, \gamma}([-r,0])\}_{\omega \in \Omega}\) forms a measurable field of Banach spaces.
		\end{theorem}
		\begin{proof}
			For \(0 < \epsilon < \gamma\), let \(C^{0,1-\epsilon}=C^{0,1-\epsilon}[-r,0]\) denote the space obtained by taking the closure of \(C^\infty\) with respect to the \((1-\epsilon)\)-Hölder norm. It is not hard to see that 
			this space is a separable Banach space. Set
			\begin{align}\label{GGBBGG}
				\begin{split}
					&\Gamma_{\omega}:=(\prod_{d\times n}C^{0,1-\epsilon})\times \prod_{n}C^{0,1-\epsilon}\longrightarrow \mathscr{D}_{\mathbf{B}(\omega)}^{\beta, \gamma}([-r,0]),
					\\& \Gamma_{\omega}((f,g))(t):=\int_{-r}^{t}f(\tau)\mathrm{d}B_\tau+g(\tau):=\sum_{1\leq i\leq d}\int_{-r}^{t}f_{i}(\tau)\mathrm{d}B^{i}_{\tau}(\omega)+g(\tau),
				\end{split}
			\end{align}
			where the integrals are understood in the usual Young's sense and $f_{i}:[-r,0]\rightarrow\mathbb{R}^n$. Note that since \(1 - \epsilon + \gamma > 1\), the integrals are well-defined. Now we define
			\begin{align*}
				\Delta:=\left\lbrace \prod_{\omega\in\Omega}\Gamma_{\omega}(h)\big|\ \  h=(f,g)\in (\prod_{d\times n}C^{0,1-\epsilon})\times \prod_{n}C^{0,1-\epsilon} \right\rbrace.
			\end{align*}
			Select a countable dense subset \((\prod_{d\times n}C^{0,1-\epsilon})\times \prod_{n}C^{0,1-\epsilon}\) and define
			\begin{align*}
				\Delta_0:=\left\lbrace \prod_{\omega\in\Omega}\Gamma_{\omega}(h)\big|\ \  h\in \mathcal{S}_{0} \right\rbrace.
			\end{align*}
			We claim that by these constructions, we can fulfill the conditions of Definition \ref{MESSSAA} for \(\prod_{\omega \in \Omega} \mathscr{D}_{\mathbf{B}(\omega)}^{\beta, \gamma}([-r,0])\).
			\begin{itemize}
				\item[(i)] By definition, it is clear that $\Delta$ is a linear subspace of $\prod_{\omega \in \Omega} \mathscr{D}_{\mathbf{B}(\omega)}^{\beta, \gamma}([-r,0])$.
				\item [(ii)] The density claim for \(\Delta_{0}\) follows from item (ii) in \cite[Theorem 3.10]{GVRST22}.
				\item [(iii)] Let $h\in (\prod_{d\times n}C^{0,1-\epsilon})\times \prod_{n}C^{0,1-\epsilon}$, we must check that the following function is (Borel) measurable
				\begin{align*}
					\omega\longrightarrow\Vert \Gamma_{\omega}(h)\Vert_{\mathscr{D}_{\mathbf{B}(\omega)}^{\beta,\gamma}([-r,0])}.
				\end{align*}
				This is clear since, in \eqref{GGBBGG}, the integrals are defined in the Young sense. By the usual partition approximation, the measurability can be easily verified (see \cite[Proposition 3.15]{GVRS22} for more details).
			\end{itemize}
			So, our claim is proved.
		\end{proof} 
		\begin{remark}
			Since our probability space is \(\Omega := C_{0}^{0,p\text{-var}}(\mathbb{R}, \tilde{T}^{2}(\mathbb{R}^d))\) (cf. Theorem \ref{thm:B_delayed_cocycle}), and this space is a Polish space, one can show that \(\{\mathscr{D}_{\mathbf{B}(\omega)}^{\beta, \gamma}([-r,0])\}_{\omega \in \Omega}\) forms a continuous field of Banach spaces, which is clearly a stronger structure. See \cite[Theorem 3.10]{GVRST22}, where a much more general statement is proven for the space of branched rough paths.
		\end{remark}
		\begin{remark}
			A similar result for the case where \(B\) is a Brownian motion is proven in \cite[Proposition 3.15]{GVRS22}. However, this result is specific to Brownian motion, and unlike Theorem \ref{BETA}, the condition on \(\beta\) is more restrictive (see \cite[Equation 3.2]{GVRS22}). Theorem \ref{BETA} not only generalizes this result to fractional Brownian motion but also relaxes the restriction on the choice of \(\beta\).
		\end{remark}
		\begin{remark}
			In the case where \( H = \frac{1}{2} \), i.e., the Brownian motion case, we can alternatively define
			\begin{align*}
				\mathbf{B}_{s,t}^{\text{It\=o}} := \left( B_{s,t}, \mathbb{B}_{s,t}^{\text{It\=o}}, \mathbb{B}_{s,t}^{\text{It\=o}}(-r) \right) := \left(B_t - B_s, \int_s^t (B_u - B_s) \otimes \mathrm{d}B_u, \int_s^t (B_{u - r} - B_{s - r}) \otimes \mathrm{d}B_u \right),
			\end{align*}
			where the stochastic integrals are understood in the It\=o sense. In this context, the Kolmogorov-Chentsov theorem allows us to recover a delayed \(\gamma\)-rough path, which is not geometric, as discussed in \cite{GVRS22}[Proposition 2.2]. Moreover, as explored in \cite{GVRS22}, all the results in this paper can also be expressed in terms of \(\mathbf{B}^{\text{It\=o}}\).
		\end{remark}
		Recall that our solution to Equation \ref{equation+drift2} defines a cocycle over the family of random Banach spaces \(\{\mathscr{D}_{\mathbf{B}(\omega)}^{\beta}([-r,0])\}_{\omega \in \Omega}\) for every \(\frac{1}{3}<\beta \leq\gamma< H\). As previously mentioned, this family of random Banach spaces possesses advantageous properties that enable us to define a generalized cocycle for the solution to Equation \ref{equation+drift2}. However, these Banach spaces are not separable, which presents a technical challenge: it hinders our ability to prove certain measurability conditions necessary for the Multiplicative Ergodic Theorem.
		To overcome this obstacle, we consider the parameters \(\frac{1}{3} < \beta < \gamma < H\) and work with \(\mathscr{D}_{\mathbf{B}(\omega)}^{\beta, \gamma}([-r,0])\) instead. By definition, these spaces are invariant under the solution, meaning that the properties \eqref{TVBN} and \eqref{VCCA} hold when we replace \(\mathscr{D}_{\mathbf{B}(\omega)}^{\beta}([-r,0])\) with \(\mathscr{D}_{\mathbf{B}(\omega)}^{\beta, \gamma}([-r,0])\). We are now ready to present our main results in this chapter regarding the existence of invariant manifolds. Similar to deterministic dynamical systems, we are interested in certain invariant subsets around the equilibrium. In the context of random dynamical systems, the counterpart to an equilibrium point is the stationary point, which will be defined in the following definition.
		\begin{definition}\label{STATA}
			Let \(\frac{1}{3} < \beta < \gamma < H\) be fixed parameters. We define \(E_\omega := \mathscr{D}_{\mathbf{B}(\omega)}^{\beta, \gamma}([-r, 0])\). Let \(\varphi\) be the solution of equation \ref{equation+drift2}, which defines a cocycle with respect to the metric dynamical system \((\Omega, \mathcal{F}, \P, \theta)\). A collection of random points \(\{Y_\omega\}_{\omega \in \Omega}\) is called a \textit{stationary point} of \(\varphi\) if the following conditions are satisfied:
			\begin{itemize}
				\item[(i)] \(Y_\omega \in E_\omega\),
				\item[(ii)] \(\varphi^{m}_{\omega}(Y_\omega) = Y_{\theta_{mr}\omega}\), 
				\item[(iii)] \(\omega \mapsto \|Y_\omega\|_{E_\omega}\) is measurable.
				%	\item [(iv)]  $\Vert Y_\omega\Vert_{E_\omega}\in \bigcap_{p\geq 1}L^{p}(\Omega)$.\todo{Either remove or call it an integrable stationary point.}
			\end{itemize}
		\end{definition}
		\begin{remark}
			A simple example of a stationary point is the zero path when \( G(0,0) = 0 \).
		\end{remark}
		The following result guarantees that when equation \ref{equation+drift2} admits a stationary point, the linearized equation generates certain deterministic values (Lyapunov exponents) that determine the dynamical behavior of our solution.
		\begin{proposition}\label{MEEETTTT}
			Consider the equation in Theorem \ref{SSRR} and assume that \(\varphi\) is the cocycle defined for this equation (restricted to \((E_{\omega})_{\omega \in \Omega}\)). Further, assume that \(\{Y_\omega\}_{\omega \in \Omega}\) is a stationary point of this cocycle such that
			\begin{align}\label{STATAT}
				\Vert Y_\omega\Vert_{E_\omega} \in \bigcap_{p \geq 1} L^{p}(\Omega).
			\end{align}
			Then \(\varphi\) is differentiable (in the Fréchet sense), and for \(\psi^{m}_{\omega} := D_{Y_{\omega}}\varphi^{m}\), the linear cocycle property holds. That is, for \(p, q \in \mathbb{N}\), we have
			\[
			\psi^{p+q}_{\omega} = \psi^{p}_{\theta_{rq} \omega} \circ \psi^{q}_{\omega}.
			\]
			Furthermore, for every \(m \in \mathbb{N}\), the function \(\psi^{m}_{\omega}: E_\omega \to E_{\theta_{mr} \omega}\) is linear. For \(\mu \in \mathbb{R}\), set
			\[
			F_{\mu}(\omega) := \left\{ x \in E_\omega : \ \limsup_{m \rightarrow \infty} \frac{1}{m} \log \Vert \psi^{m}_{\omega}(x) \Vert_{E_{\theta_{mr}\omega}} \leq \mu \right\}.
			\]
			Then, we have:
			\begin{itemize}
				\item There exists a measurable $\theta_r$-invariant set $\tilde{\Omega} \subset \Omega$ of full measure, along with a decreasing sequence $\{\mu_i\}_{i \geq 1}$ of Lyapunov exponents, $\mu_i \in [-\infty, \infty)$, which have the properties $\lim_{n \to \infty} \mu_n = -\infty$ and either $\mu_i > \mu_{i+1}$ or $\mu_i = \mu_{i+1} = -\infty$. For each $\omega \in \tilde{\Omega}$, we have $F_{\mu_1}(\omega) = E_{\omega}$, and
				\begin{align}
					x \in F_{\mu_{i}}(\omega) \setminus F_{\mu_{i+1}}(\omega) \quad \text{if and only if} \quad \lim_{m \rightarrow \infty} \frac{1}{m} \log \Vert \psi^m_{\omega}(x) \Vert_{E_{\theta_{mr}\omega}} = \mu_{i}.
				\end{align}
				Moreover, there exist integers $m_1, m_2, \ldots$ such that $\operatorname{codim} F_{\mu_j}(\omega) = m_1 + \ldots + m_{j-1}$ for each $\omega \in \tilde{\Omega}$. Furthermore, for every $i \geq 1$ with $\mu_i > \mu_{i+1}$ and $\omega \in \tilde{\Omega}$, there exists an $m_i$-dimensional subspace $H^i_\omega$ with the following characteristics:
				\begin{itemize}
					\item[(i)] \textbf{(Invariance)}\ \ $\psi_{\omega}^k(H^i_{\omega}) = H^i_{\theta_{kr} \omega}$ for each $k \geq 0$.
					\item[(ii)] \textbf{(Splitting)}\ \ $H_{\omega}^i \oplus F_{\mu_{i+1}}(\omega) = F_{\mu_i}(\omega)$. Specifically,
					\begin{align*}
						E_{\omega} = H^1_{\omega} \oplus \cdots \oplus H^i_{\omega} \oplus F_{\mu_{i+1}}(\omega).
					\end{align*}
					\item[(iii)] \textbf{(Fast-growing subspace)}\ \ For every $h_{\omega} \in H^{i}_{\omega} \setminus \{0\}$,
					\begin{align*}
						\lim_{m \rightarrow \infty} \frac{1}{m} \log \Vert \psi^{m}_{\omega}(h_{\omega}) \Vert_{E_{\theta_{mr}\omega}} = \mu_{i},
					\end{align*}
					and
					\begin{align*}
						\lim_{m \rightarrow \infty} \frac{1}{m} \log \Vert (\psi^{m}_{\theta_{-mr}\omega})^{-1}(h_{\omega}) \Vert_{E_{\omega}} = -\mu_{i}.
					\end{align*}
				\end{itemize} 
			\end{itemize}
		\end{proposition}
		\begin{proof}
			From Theorem \ref{BETA}, the family of random spaces \(\lbrace E_\omega \rbrace_{\omega \in \Omega}\) forms a measurable (continuous) field of Banach spaces. Additionally, Theorem \ref{SSRR} establishes that \(\varphi\) is a nonlinear cocycle, and Theorem \ref{DCXSZ} shows that \(\varphi\) is also differentiable. By the definition of a stationary point, it is straightforward to see that \(\psi_{\omega}^{m} := D_{Y_{\omega}}\varphi^{m}_{\omega}\) is also a cocycle. This can be verified using the fact that \(\varphi^{p+q}_{\omega}(x) = \varphi^{p}_{\theta_{qr} \omega} \circ \varphi^{q}_{\omega}(x)\) and applying the chain rule (by taking derivatives around the stationary points). Furthermore, Theorem \ref{DCXSZ} ensures that for a polynomial \(Q_1\), the following holds:
			\begin{align}\label{AZMOPPLLM}
				\log^{+}(\Vert\psi^{1}_{\omega}\Vert_{L(E_\omega,E_{\theta_r\omega})}) \leq Q_{1}(\Vert Y_\omega\Vert_{E_\omega}, \Vert\mathbf{B}(\omega)\Vert_{\gamma,[0,r]}) \in L^{1}(\Omega),
			\end{align}
			where \(\log^+ x \coloneqq \max \{0, \log x\}\). Note that the integrability claim in the above equation follows from \eqref{STATAT} and the fact that \(\Vert \mathbf{B}(\omega) \Vert_{\gamma, [0,r]} \in \bigcap_{p \geq 1} L^{p}(\Omega)\); cf. \eqref{JHNBZS}. Note that, from \cite[Proposition 1.9]{GVRS22}, \(\psi^{1}_{\omega}\) is also a compact map for every \(\omega \in \Omega\). Now our claim follows from \cite{GVR23}[Theorem 1.21].
		\end{proof}
		Having established the Multiplicative Ergodic Theorem, we can infer the existence of invariant manifolds (stable, center, and unstable) around the stationary points. The existence of stable and unstable manifolds in the case of Brownian motion is already established in \cite{GVR21}. The contribution of the following results is twofold: first, we establish the same results (existence of stable and unstable manifolds) in more general cases, including fractional Brownian motion; second, we deduce the existence of center manifolds, which has been missing in the literature.
		Note that, due to the foundational work we have established, we can prove these results with relative ease by applying some abstract results. The nice part of these three results is that their proofs rely on (almost) the same conditions. The next result guarantees the existence of an invariant stable manifold around the stationary point. This result, in particular, is important because it enables us to conclude local exponential stability when all the Lyapunov exponents are negative.
		\begin{theorem}[Local stable manifolds]\label{thm:stable_manifold}
			Consider the equation in Theorem \ref{SSRR}, and assume that \(\varphi\) is the cocycle defined for this equation, restricted to \((E_{\omega})_{\omega \in \Omega}\). Furthermore, suppose that \(\{Y_\omega\}_{\omega \in \Omega}\) is a stationary point of this cocycle, and that the assumptions made in Theorem \ref{ASXXZA} regarding \(A\), $\pi$ and \(G\) are also satisfied. Let \(\{\mu_j\}_{j \geq 1}\) be the Lyapunov exponents corresponding to the linear cocycle \(\psi^{1}_{\omega} := D_{Y_\omega}\varphi^{1}_{\omega}\), and define \(\mu^{-} := \max\{\mu_j : \mu_j < 0\}\).
			Then, there exists a \(\theta_r\)-invariant set of full measure \(\tilde{\Omega}\) and a family of immersed submanifolds \(S^{\upsilon}_{\text{loc}}(\omega)\) of \(E_\omega\), where \(0 < \upsilon < -\mu^{-} \) and \(\omega \in \tilde{\Omega}\), satisfying the following properties for every \(\omega \in \tilde{\Omega}\):
			\begin{itemize}
				\item[(i)] There exist random variables $\rho_{1}^{\upsilon}(\omega)$ and $\rho_{2}^{\upsilon}(\omega)$, which are positive and finite on $\tilde{\Omega}$, such that
				\begin{align}\label{eqn:rho_temp_1}
					\liminf_{p \to \infty} \frac{1}{p} \log \rho_i^{\upsilon}(\theta_{pr} \omega) \geq 0, \quad i = 1,2,
				\end{align}
				and the following inclusions hold:
				\begin{align*}
					\big{\lbrace} \xi_\omega \in E_\omega \, :\, \sup_{m\geqslant 0}\exp(m\upsilon)&\Vert \varphi^{m}_{\omega}(\xi_\omega) - Y_{\theta_{mr}\omega}\Vert_{E_{\theta_{mr}\omega}} < \rho_{1}^{\upsilon}(\omega)\big{\rbrace} \subseteq S^{\upsilon}_{\text{loc}}(\omega)\\
					&\subseteq \big{\lbrace} \xi_\omega \in E_\omega \, :\, \sup_{m\geqslant 0}\exp(m\upsilon)\Vert \varphi^{m}_{\omega}(\xi_\omega) - Y_{\theta_{mr}\omega}\Vert_{E_{\theta_{mr}\omega}}<\rho_{2}^{\upsilon}({\omega})\big{\rbrace}.
				\end{align*}
				
				\item[(ii)] The tangent space at $Y_{\omega}$ satisfies
				\begin{align*}
					T_{Y_{\omega}}S^{\upsilon}_{\text{loc}}(\omega) = S_{\omega}:=F_{\mu^{-} }(\omega).
				\end{align*}
				
				\item[(iii)] For $m \geqslant M(\omega)$,
				\begin{align*}
					\varphi^{m}_{\omega}(S^{\upsilon}_{\text{loc}}(\omega)) \subseteq S^{\upsilon}_{\text{loc}}(\theta_{mr}\omega).
				\end{align*}
				
				\item[(iv)] For $0 < \upsilon_{1} \leqslant \upsilon_{2} < -\mu^{-} $,
				\begin{align*}
					S^{\upsilon_{2}}_{\text{loc}}(\omega) \subseteq S^{\upsilon_{1}}_{\text{loc}}(\omega).
				\end{align*}
				Additionally, for $m \geqslant M(\omega)$,
				\begin{align*}
					\varphi^{m}_{\omega}(S^{\upsilon_{1}}_{\text{loc}}(\omega)) \subseteq S^{\upsilon_{2}}_{\text{loc}}(\theta_{mr}(\omega)),
				\end{align*}
				and consequently, for $\xi_\omega \in S^{\upsilon}_{\text{loc}}(\omega)$,
				\begin{align}\label{eqn:contr_char_1}
					\limsup_{m\rightarrow\infty}\frac{1}{m}\log\Vert\varphi^{m}_{\omega}(\xi_\omega) - Y_{\theta_{mr}\omega}\Vert_{E_{\theta_{mr}\omega}} \leqslant \mu^{-} .
				\end{align}
				
				\item[(v)] The following holds:
				\begin{align*}
					\limsup_{m\rightarrow\infty} \frac{1}{m} \log\bigg{[}\sup\bigg{\lbrace}\frac{\Vert\varphi^{m}_{\omega}(\xi_\omega) - \varphi^{m}_{\omega}(\tilde{\xi_\omega}) \Vert_{E_{\theta_{mr}\omega}} }{\Vert \xi_\omega - \tilde{\xi_\omega} \Vert_{E_{\omega}}},\ \ \xi_\omega \neq\tilde{\xi}_\omega ,\ \xi,\tilde{\xi}_\omega \in S^{\upsilon}_{\text{loc}}(\omega)\bigg{\rbrace}\bigg{]} \leqslant \mu^{-} .
				\end{align*}
			\end{itemize}
		\end{theorem}
		\begin{proof}
			Set 
			\begin{align*}
				&P_{\omega}: E_\omega\rightarrow E_{\theta_{r}\omega},\\
				&P_{\omega}(\xi_\omega):=\varphi^{1}_{\omega}(\xi_\omega+Y_\omega)-\varphi^{1}_{\omega}(Y_\omega)-\psi^{1}_{\omega}(\xi_\omega).
			\end{align*}
			Then from Theorem \ref{ASXXZA}
			\begin{itemize}
				\item [(i)] If $A\neq 0$ and $G\in C^{4}_{b}$
				\begin{align}\label{4512sdds}
					\begin{split}
						&\Vert P_{\omega}(\xi_\omega)-P_{\omega}(\tilde{\xi}_\omega)\Vert_{E_{\theta_{r}\omega}}\leq\int_{0}^{1}\left\Vert (D_{\tau\xi_\omega+(1-\tau)\tilde{\xi}_\omega+Y_\omega}\varphi^{1}_{\omega}-D_{Y_\omega}\varphi^{1}_{\omega})[\xi_\omega-\tilde{\xi}_\omega]\right\Vert_{E_{\theta_{r}\omega}}\mathrm{d}\tau\\&\leq \Vert \xi_\omega-\tilde{\xi}_\omega\Vert_{E_\omega}(\Vert\xi_\omega\Vert_{E_\omega}+\Vert\tilde{\xi}_\omega\Vert_{E_\omega})\exp\left(\sup_{\tau\in [0,1]}Q_{2}\left(\Vert\tau\xi_\omega+(1-\tau)\tilde{\xi}_\omega\Vert_{E_\omega},\Vert Y_\omega\Vert_{E_\omega},\Vert\mathbf{B}(\omega)\Vert_{\gamma,[0,r]}\right)\right)
					\end{split}
				\end{align}
				\item [(ii)] If $A= 0$ and $G\in C^{3+\vartheta}_{b}$, for some $0<\vartheta\leq 1$
				\begin{align}
					\begin{split}
						&\Vert P_{\omega}(\xi_\omega)-P_{\omega}(\tilde{\xi}_\omega)\Vert_{E_{\theta_{r}\omega}}\leq\int_{0}^{1}\Vert (D_{\tau\xi_\omega+(1-\tau)\tilde{\xi}_\omega+Y_\omega}\varphi^{1}_{\omega}-D_{Y_\omega}\varphi^{1}_{\omega})[\xi_\omega-\tilde{\xi}_\omega]\Vert_{E_{\theta_{r}\omega}}\mathrm{d}\tau\\&\leq \Vert \xi_\omega-\tilde{\xi}_\omega\Vert_{E_\omega}(\Vert\xi_\omega\Vert_{E_\omega}^{\vartheta}+\Vert\tilde{\xi_\omega}\Vert_{E_\omega}^{\vartheta})\exp\left(\sup_{\tau\in [0,1]}Q_{3}\left(\Vert\tau\xi_\omega+(1-\tau)\tilde{\xi}_\omega\Vert_{E_\omega},\Vert Y_\omega\Vert_{E_\omega},\Vert\mathbf{B}(\omega)\Vert_{\gamma,[0,r]}\right)\right)
					\end{split}
				\end{align}
			\end{itemize}
			where \(Q_2\) and \(Q_3\) are polynomials. Note that there exist an increasing polynomial \(R_0\) and a polynomial \(R_1\) such that, for \(i=2,3\),
			\begin{align}\label{VBssN}
				\begin{split}
					&\sup_{\tau\in [0,1]}Q_{i}\left(\Vert\tau\xi_\omega+(1-\tau)\tilde{\xi}_\omega\Vert_{E_\omega},\Vert Y_\omega\Vert_{E_\omega},\Vert\mathbf{B}(\omega)\Vert_{\gamma,[0,r]}\right)\\&\leq R_{0}(\Vert\xi_\omega\Vert_{E_\omega}+\Vert\tilde{\xi}_{\omega}\Vert_{E_\omega})+ R_{1}(\Vert Y_\omega\Vert_{E_\omega},\Vert\mathbf{B}(\omega)\Vert_{\gamma,[0,r]}).
				\end{split}
			\end{align}
			Since the rest of the argument is the same in both cases \((i)\) and \((ii)\), we will complete the proof for case \((i)\). We plug inequality \eqref{VBssN} into \eqref{4512sdds}, then 
			\begin{align*}
				\Vert P_{\omega}(\xi_\omega) - P_{\omega}(\tilde{\xi}_\omega)\Vert_{E_{\theta_{r}\omega}} 
				&\leq \Vert \xi_\omega - \tilde{\xi}_\omega \Vert_{E_\omega} \left( \Vert \xi_\omega \Vert_{E_\omega} + \Vert \tilde{\xi}_\omega \Vert_{E_\omega} \right) \exp\left(R_0 \left( \Vert \xi_\omega \Vert_{E_\omega} + \Vert \tilde{\xi}_\omega \Vert_{E_\omega} \right)\right) \\
				&\quad \times \underbrace{\exp\left( R_1 \left(\Vert Y_\omega \Vert_{E_\omega}, \Vert\mathbf{B}(\omega)\Vert_{\gamma,[0,r]} \right)\right)}_{f(\omega)}.
			\end{align*}
			We define \(h(x): \mathbb{R}^+ \rightarrow \mathbb{R}^+\) by setting \(h(x) = x\exp(R_0(x))\). Then, \(h(x)\) is an increasing and differentiable function. Also, since \(R_1\) is a polynomial, similar to \eqref{AZMOPPLLM}, we have
			\begin{align*}
				\log^{+}(f(\omega)) \in L(\Omega).
			\end{align*}
			So, from the Birkhoff ergodic theorem, we have
			\begin{align*}
				\lim_{m \rightarrow \infty} \frac{1}{m} \log^{+}\left(f(\theta_{mr}\omega)\right) = 0, \text{ a.s.}
			\end{align*}
			Therefore, we have verified the condition of \cite[Theorem 2.10]{GVR23} and can apply this result to conclude our claim.
		\end{proof}
		This result immediately gives the following exponential stability result when all the Lyapunov exponents are negative.
		\begin{corollary}\label{STTTAASDSF}
			Assume that \(\mu_{1} < 0\). Then, for every \(0 < \upsilon < -\mu_1\), there exists a random variable \(R^{\upsilon}: \Omega \to (0, \infty)\) such that
			\[
			\liminf_{p \to \infty} \frac{1}{p} \log R^{\upsilon}(\theta_{pr} \omega) \geq 0
			\]
			and
			\begin{align}\label{neighborhood_1}
				\left\{\xi_\omega \in E_\omega \mid \Vert \xi_\omega - Y_{\omega} \Vert_{E_\omega} < R^{\upsilon}(\omega)\right\} = S^{\upsilon}_{loc}(\omega).
			\end{align}
		\end{corollary}
		\begin{proof}
			The proof follows from the fact that, in this case, \( F_{\mu_{1}}(\omega) = E_\omega \) with only minor modifications to \cite[Theorem 2.10]{GVR23}. For details on these modifications, we refer to \cite[Corollary 2.17]{GVR23A}.
		\end{proof}
		The following result guarantees the existence of an unstable manifolds.
		\begin{theorem}[Local unstable manifolds]\label{thm:unstable_manifold}
			Assume the same setting as in Theorem \ref{thm:stable_manifold}. Furthermore, suppose that \( \mu_1 > 0 \) holds for the first Lyapunov exponent and $\mu^{+}:=\min\lbrace\mu_j : \ \mu_j>0\rbrace$. Then there exists a $\theta_r$-invariant set of full measure $\tilde{\Omega}$ and a family of immersed submanifolds \( U^{\upsilon}_{loc}(\omega)\) of \( E_\omega \), for \( 0 < \upsilon < \mu^{+} \) and \( \omega \in \tilde{\Omega} \), satisfying the following properties for every \( \omega \in \tilde{\Omega} \):
			\begin{itemize}
				\item[(i)] There exist random variables \( \tilde{\rho}_{1}^{\upsilon}(\omega) \) and \( \tilde{\rho}_{2}^{\upsilon}(\omega) \), positive and finite on \( \tilde{\Omega} \), such that
				\begin{align*}
					\liminf_{p \to \infty} \frac{1}{p} \log \tilde{\rho}_i^{\upsilon}(\theta_{-pr} \omega) \geq 0, \quad i = 1, 2,
				\end{align*}
				and
				\begin{align*}
					&\left\lbrace \xi_{\omega} \in E_\omega\, :\, \exists \lbrace \xi_{\theta_{-pr}\omega}\rbrace_{p\geqslant 1} \text{ s.t. } \varphi^{m}_{\theta_{-pr}\omega}(\xi_{\theta_{-pr}\omega}) = \xi_{\theta_{(m-p)r}\omega} \text{ for all } 0 \leq m \leq p, \right. \\
					&\quad \left.\sup_{p \geqslant 0} \exp(p\upsilon)\Vert \xi_{\theta_{-pr}\omega} - Y_{\theta_{-pr}\omega} \Vert_{E_{\theta_{-pr}\omega}} < \tilde{\rho}_{1}^{\upsilon}(\omega)\right\rbrace \subseteq U^{\upsilon}_{loc}(\omega) \\
					&\quad \subseteq \left\lbrace\xi_{\omega} \in E_\omega\, :\, \exists \lbrace \xi_{\theta_{-pr}\omega}\rbrace_{p\geqslant 1} \text{ s.t. } \varphi^{m}_{\theta_{-pr}\omega}(\xi_{\theta_{-pr}\omega}) = \xi_{\theta_{(m-p)r}\omega} \text{ for all } 0 \leq m \leq p, \right. \\
					&\quad \left. \sup_{p\geqslant 0} \exp(p\upsilon)\Vert \xi_{\theta_{-pr}\omega} - Y_{\theta_{-pr}\omega} \Vert_{E_{\theta_{-pr}\omega}} < \tilde{\rho}_{2}^{\upsilon}(\omega)\right\rbrace.
				\end{align*}
				
				\item[(ii)] The tangent space at \( Y_{\omega} \) to \( U^{\upsilon}_{loc}(\omega) \) is given by
				\begin{align*}
					T_{Y_{\omega}}U^{\upsilon}_{loc}(\omega) = {U}_{\omega}:=\oplus_{i: \mu_{i}>0}H^{i}_{\omega}.
				\end{align*}
				
				\item[(iii)] For \( m \geqslant M(\omega) \),
				\begin{align*}
					U^{\upsilon}_{loc}(\omega) \subseteq \varphi^{m}_{\theta_{-mr}\omega}( U^{\upsilon}_{loc}(\theta_{-mr}\omega)).
				\end{align*}
				
				\item[(iv)] For \( 0 < \upsilon_{1} \leqslant \upsilon_{2} < \mu^{+} \),
				\begin{align*}
					U^{\upsilon_{2}}_{loc}(\omega) \subseteq U^{\upsilon_{1}}_{loc}(\omega).
				\end{align*}
				Also, for \( m \geqslant M(\omega) \),
				\begin{align*}
					U^{\upsilon_{1}}_{loc}(\omega) \subseteq \varphi^{m}_{\theta_{-mr}\omega}( U^{\upsilon_2}_{loc}(\theta_{-mr}\omega)),
				\end{align*}
				and consequently, for \( \xi_{\omega} \in U^{\upsilon}_{loc}(\omega) \),
				\begin{align*}
					\limsup_{p\rightarrow \infty} \frac{1}{p}\log\Vert \xi_{\theta_{-pr}\omega} - Y_{\theta_{-pr}\omega}\Vert_{E_{\theta_{-pr}\omega}} \leqslant -\mu^{+}.
				\end{align*}
				
				\item[(v)] 
				\begin{align*}
					\limsup_{p \rightarrow \infty} \frac{1}{p}\log\left[\sup\left\lbrace\frac{\Vert \xi_{\theta_{-pr}\omega} - \tilde{\xi}_{\theta_{-pr}\omega}\Vert_{E_{\theta_{-pr}\omega}}}{\Vert \xi_{\omega} - \tilde{\xi}_{\omega}\Vert_{E_{\omega}}}\, :\, \xi_{\omega} \neq \tilde{\xi}_{\omega},\ \xi_{\omega}, \tilde{\xi}_{\omega} \in U^{\upsilon}_{loc}(\omega)\right\rbrace\right] \leqslant -\mu^{+}.
				\end{align*}
			\end{itemize}
		\end{theorem}
		\begin{proof}
			The proof is almost the same as the proof of Theorem \ref{thm:stable_manifold}. Indeed, we must verify the same conditions and then apply \cite[Theorem 2.17]{GVR23}.
		\end{proof}
		So far, we have established the existence of stable and unstable manifolds. Analogous to the deterministic case, when one of the Lyapunov exponents is zero, we naturally expect the presence of invariant sets that neither decay nor grow exponentially. We refer to these sets as the center manifolds. The proof of such invariant sets is somewhat more subtle, and we can only guarantee their invariance up to a certain number of iterations.
		\begin{theorem}\label{thm:center_manifold}
			Assume the same setting as in Theorem \ref{thm:stable_manifold}. Suppose that at least one Lyapunov exponent is zero, denoted by $\mu_{i_c}$. Therefore, $\mu_{i_c} = 0$, $\mu^{-} = \max\{\mu_{i} : \mu_{i} < 0\} = \mu_{i_c+1}$, and if $i_{c} > 1$, then $\mu^{+} = \min\{\mu_{i} : \mu_{i} > 0\} = \mu_{i_c-1}$; otherwise, we set $\mu^{+} := \infty$. We also set $C_\omega := H^{i_c}_{\omega}$.
			Then we can find a continuous cocycle $\tilde{\varphi}$ defined on $\{E_\omega\}_{\omega \in \Omega}$. That is, for every $p, q \in \mathbb{N}$ and $\omega \in \Omega$, the map $\tilde{\varphi}^p_{\omega}: E_\omega \to E_{\theta_{rp}\omega}$ is continuous, and we have $\tilde{\varphi}^{p+q}_{\omega} = \tilde{\varphi}^p_{\theta_{qr}\omega} \circ \tilde{\varphi}^q_{\omega}$. For this new cocycle, and for every $0 < \upsilon < \min\{\mu^{+}, -\mu^{-}\}$, on a set of full measure $\tilde{\Omega}$, which is $\theta_r$-invariant, the following items hold:
			\begin{itemize}
				\item[(i)] There exists a positive random variable $\rho^c:\Omega \to (0, \infty)$ such that
				\begin{align*}
					\liminf_{m \rightarrow \pm \infty} \frac{1}{m} \log \rho^c(\theta_{mr}\omega) \geq 0,
				\end{align*}
				and if $\|\xi_{\omega} - Y_\omega\|_{E_\omega} \leq \rho^c(\omega)$, then $\tilde{\varphi}^{1}_{\omega}(Y_\omega + \xi_\omega) = \varphi^{1}_{\omega}(Y_\omega + \xi_\omega)$.
				
				\item[(ii)] We can find a map
				\begin{align*}
					h_{\omega}: C_{\omega} \rightarrow \mathcal{M}^{c,\upsilon}_{\omega},
				\end{align*}
				such that this map is a homeomorphism, Lipschitz continuous, and differentiable at zero.
				
				\item[(iii)] The manifold $\mathcal{M}^{c,\upsilon}_{\omega}$ obtained from the previous item is a topological Banach manifold modeled on $C_\omega$.
				
				\item[(iv)] $\mathcal{M}^{c,\upsilon}_{\omega}$ is $\tilde{\varphi}$-invariant, i.e., for every $m \in \mathbb{N}_0$, $\tilde{\varphi}^{m}_{\omega}(\mathcal{M}^{c,\upsilon}_{\omega}) \subset \mathcal{M}^{c,\upsilon}_{\theta_{mr}\omega}$.
				
				\item[(v)] For every $\xi_\omega \in \mathcal{M}^{c,\upsilon}_{\omega}$, we can find a sequence $\{\xi_{\theta_{-mr}\omega}\}_{m \geq 1}$ such that if we set
				\begin{align*}
					\tilde{\varphi}^{jr}_{\omega}(\xi_\omega) \coloneqq \xi_{\theta_{jr}\omega}
				\end{align*}
				for $j \leq 0$, then
				\begin{align*}
					\tilde{\varphi}^p_{\theta_{qr}\omega}(\tilde{\varphi}^q_{\omega}(\xi_\omega)) = \tilde{\varphi}^{p+q}_{\omega}(\xi_\omega)
				\end{align*}
				for all $(p, q) \in \mathbb{N} \times \mathbb{Z}$, and
				\begin{align*}
					\sup_{j \in \mathbb{Z}} \exp(-\nu |j|) \| \tilde{\varphi}^{jr}_{\omega}(\xi_\omega) - Y_{\theta_{jr}\omega} \|_{E_{\theta_{rj}\omega}} < \infty.
				\end{align*}
			\end{itemize}
		\end{theorem}
		\begin{proof}
			We must again check the same conditions as Theorem \ref{thm:stable_manifold} and Theorem \ref{thm:unstable_manifold}. Then the results follow from \cite[Theorem 2.14]{GVR24}.
		\end{proof}
		\section{\texorpdfstring{Intermezzo: An MET-approach to exponential stability of $\mathcal{C}_0$-semigroups}{}}\label{sec:MET_approach}
		%\section{Intermezzo: An MET-approach to exponential stability of $\mathcal{C}_0$-semigroups}\label{sec:MET_approach}
		We have established a general framework for analyzing the dynamics of fractional delay equations. In this section, we discuss several general results related to the spectrum of compact semigroups. Although many of these results have been previously established in other contexts, such as those detailed in \cite[Chapter 7]{HL93} and \cite[Chapter 4]{MR18}, we offer new proofs by leveraging Lyapunov exponents to provide concise and streamlined arguments. These results will later be integrated into our framework to investigate the exponential stability of rough delay equations. As noted, the generalized version of the MET is crucial for achieving these outcomes. We will first provide some background on the semigroup associated with delay equations.
		\begin{definition}
			Let $\mathcal{B}$ be a real Banach space. A strongly continuous semigroup of linear operators, or a $C_0$-semigroup, is a one-parameter family $\{T_t\}_{t \geq 0}$ of bounded linear operators on $\mathcal{B}$ that satisfies the following properties:
			
			\begin{enumerate}
				\item \textbf{Semigroup Property:} For all $t, s \geq 0$, we have
				\[
				T_{t+s} = T_t\circ T_s.
				\]
				\item \textbf{Identity at Zero:} The operator $T_0$ is the identity operator on $\mathcal{B}$, i.e.,
				\[
				T_0 = I.
				\]
				\item \textbf{Strong Continuity:} For every $x \in \mathcal{B}$, the map $t \mapsto T_t x$ is continuous from $[0, \infty)$ to $\mathcal{B}$.
			\end{enumerate}
			To every $C_0$-semigroup $\{T(t)\}_{t \geq 0}$ on a Banach space $\mathcal{B}$, we can associate an infinitesimal generator $L$ defined by
			
			\[
			{L} z = \lim_{t \to 0^+} \frac{T(t) z - z}{t},
			\]
			
			for all $z \in \mathcal{D}(L)$. The domain of the generator $L$, denoted by $\mathcal{D}(L)$, is given by
			
			\[
			\mathcal{D}(L) = \left\{ \varphi \in \mathcal{B} \mid \lim_{t \to 0^+} \frac{T(t) \varphi - \varphi}{t} \text{ exists and is finite} \right\}.
			\]
			
			In other words, $\mathcal{D}(L)$ consists of all elements $\varphi \in \mathcal{B}$ for which the limit above exists. The generator $\mathcal{L}$ is a densely defined linear operator mapping $\mathcal{D}(L)$ into $\mathcal{B}$.
		\end{definition}
		\begin{definition}
			The resolvent set $\rho(L)$ of an operator $L$ is the set of all complex numbers $\lambda$ such that the operator $\lambda I - L$ has a bounded inverse, and the domain of this inverse is dense in $\mathcal{B}$. This inverse will be denoted by $R(\lambda, L)$. The complement of $\rho(L)$ in the complex plane is called the spectrum of $L$ and is denoted by $\sigma(L)$. The spectrum of an operator may consist of three different types of points:
			\begin{enumerate}
				\item \textbf{Residual Spectrum} $\sigma_r(L)$: These are the values $\lambda \in \sigma(L)$ for which the resolvent operator $R(\lambda, L)$ exists, but its domain $\mathcal{D}(R(\lambda, L))$ is not dense in $\mathcal{B}$.
				
				\item \textbf{Continuous Spectrum} $\sigma_c(L)$: These are the values $\lambda \in \sigma(L)$ for which $\lambda I - L$ has an unbounded inverse with a dense domain.
				
				\item \textbf{Point Spectrum} $\sigma_p(L)$: These are the values $\lambda \in \sigma(L)$ for which $\lambda I - L$ does not have an inverse.
			\end{enumerate}
		\end{definition}
		The theory of semigroups is well-established, with many abstract results, particularly for \( C_0 \)-semigroups. Since this manuscript focuses on delay equations, we will primarily address this type of equation. Specifically, we aim to demonstrate how classical stability in the deterministic case can be leveraged to obtain path-wise stability for stochastic delay equations. We denote by \( C([-r,0], \mathbb{R}^n) \) the space of continuous paths from \([-r,0]\) to \(\mathbb{R}^n\). For every element \(\xi \in C([-r,0], \mathbb{R}^n)\), we recall the norm
		\begin{align*}
			\|\xi\|_{\infty,[-r,0]} := \Vert\xi\Vert_{ C([-r,0], \mathbb{R}^n)}=\sup_{\tau \in [-r,0]} |\xi_{\tau}|,
		\end{align*}
		With this norm, \( C([-r,0], \mathbb{R}^n) \) is a Banach space. Let us now consider the following classical result; its proof can be found in \cite{HL93}.
		\begin{lemma}\label{PII}
			Assume \( A \in {L}(\mathbb{R}^n \times \mathbb{R}^n, \mathbb{R}^n) \), and let \( \pi = (\pi^{i,j})_{1 \leq i,j \leq n} \) be a collection of finite signed measures.
			\begin{comment}
				Consider the following delay differential equation:
				\begin{align}\label{linear delay}
					\begin{split}
						&\mathrm{d}y_{t} = A\left(y_t, \int_{-r}^{0} y_{\theta+t} \pi(\mathrm{d}\theta)\right) \mathrm{d}t, \\
						&\quad y_{\theta} = \xi_{\theta} \ \text{for} \ \theta \in [-r,0] \ \text{and} \ \xi \in C([-r,0], \mathbb{R}^n),
					\end{split}
				\end{align}
			\end{comment}
			For $y\in C([-r,0],\mathbb{R}^n)$, assume $\int_{-r}^{0} y_{\theta+t} \pi(\mathrm{d}\theta)$ denotes the vector 
			\[
			\left(\sum_{1 \leq i \leq n} \int_{-r}^{0} y^i_{\theta+t} \pi^{i,j}(\mathrm{d}\theta)\right)_{1 \leq j \leq n}.
			\]
			Assume that the semigroup \( T_t: C([-r,0], \mathbb{R}^n) \rightarrow C([-r,0], \mathbb{R}^n) \) is defined by
			\begin{align}\label{delay-semigroup}
				(T_t\xi)(\tau) := \begin{cases}
					\xi_{t+\tau} & \text{if } t + \tau \leq 0, \\
					\xi_0 + \int_{0}^{t+\tau} A\left((T_s \xi)(0), \int_{-r}^{0} (T_s \xi)(\theta) \pi(\mathrm{d}\theta)\right) \, \mathrm{d}s & \text{if } t + \tau > 0.
				\end{cases}
			\end{align}
			Then the associated generator of this semigroup is \( L \xi := \frac{\mathrm{d}}{\mathrm{d}\theta} \xi \) on \( C([-r,0], \mathbb{R}^n) \), where
			\begin{align}\label{operator}
				\mathcal{D}(L) := \left\{ \xi \in C^1([-r,0], \mathbb{R}^n) \mid \frac{\mathrm{d}\xi}{\mathrm{d}\theta}(0) = A\left(\xi_0, \int_{-r}^{0} \xi_{\theta} \pi(\mathrm{d}\theta) \right) \right\}.
			\end{align}
			Also, \((T_t)_{t \geq 0}\) is a \(\mathcal{C}_0\)-semigroup, and for \( t \geq r \), this semigroup is compact. Additionally, \( L \) is a closed, densely defined operator on \( C([-r,0], \mathbb{R}^n) \). Also for every $\xi\in C([-r,0], \mathbb{R}^n)$, $T_{t}(\xi)$, is solution to the following equation
			\begin{align}\label{linear delay}
				\begin{split}
					&\mathrm{d}y_{t} = A\left(y_t, \int_{-r}^{0} y_{\theta+t} \pi(\mathrm{d}\theta)\right) \mathrm{d}t, \\
					&\quad y_{\theta} = \xi_{\theta} \ \text{for} \ \theta \in [-r,0] \ \text{and} \ \xi \in C([-r,0], \mathbb{R}^n).
				\end{split}
			\end{align}
		\end{lemma}
		Since we are primarily interested in stability, we will focus particularly on the spectrum of \( L \).
		\begin{remark}\label{FVCC}
			We define the characteristic equation as follows:
			\begin{align}\label{characteristic}
				\begin{split}
					&\Delta(z):\mathbb{R}^n\rightarrow\mathbb{R}^n, \\
					&\Delta(z)C = zC - A\left(C, \int_{-r}^{0} C \exp(z\theta) \pi(\mathrm{d}\theta)\right).
				\end{split}
			\end{align}
			From \cite[Lemma 2.1, Theorem 4.2 (Chapter 7)]{HL93}, for the elements of spectrum $L$, we have
			\begin{align}\label{RR}
				z\in \sigma(L)\Longleftrightarrow \det(\Delta(z))=0.
			\end{align} 
			Also, $\sigma(L)=\sigma_p(L)$.
		\end{remark}
		Remark \ref{FVCC} gives a necessary and sufficient condition to find the spectrum of the operator \(L\). We need an auxiliary Lemma to give a sufficient condition for establishing the stability. The next Proposition is a direct consequence of the Multiplicative Ergodic Theorem in Banach spaces
		\begin{proposition}\label{MEEET}
			Assume \(\mathcal{B}\) is a Banach space and \((T_{t})_{t \geq 0}\) is a \(\mathcal{C}_0\)-semigroup defined on \(\mathcal{B}\). Also assume for \(t > r\), this semigroup is also compact. Then there exists a decreasing sequence \(\{\lambda_i\}_{i \geq 1}\), \(\lambda_i \in (-\infty, \infty)\) (Lyapunov exponents) with the properties that \(\lim_{n \to \infty} \lambda_n = -\infty\) and either \(\lambda_i > \lambda_{i+1}\) or \(\lambda_i = \lambda_{i+1} = -\infty\) such that if we set
			\[ 
			F_{\lambda_{i}}:=\{ x \in \mathcal{B} : \limsup_{t \rightarrow \infty} \frac{1}{t} \log \Vert T_{t}x \Vert_{\mathcal{B}} \leq \lambda_{i} \},
			\]
			Assume \(\mathcal{B}\) is a Banach space and \((T_{t})_{t \geq 0}\) is a \(\mathcal{C}_0\)-semigroup defined on \(\mathcal{B}\). Also assume for \(t > r\), this semigroup is also compact. Then there exists a decreasing sequence \(\{\lambda_i\}_{i \geq 1}\), \(\lambda_i \in [-\infty, \infty)\) (Lyapunov exponents) with the properties that \(\lim_{n \to \infty} \lambda_n = -\infty\) and either \(\lambda_i > \lambda_{i+1}\) or \(\lambda_i = \lambda_{i+1} = -\infty\) such that if we set
			\[ 
			F_{\lambda_{i}}:=\{ x \in \mathcal{B} : \limsup_{t \rightarrow \infty} \frac{1}{t} \log \Vert T_{t}x \Vert_{\mathcal{B}} \leq \lambda_{i} \},
			\]
			then
			\begin{align}\label{nest}
				\begin{split}
					& \ldots \subset F_{\lambda_{2}} \subset F_{\lambda_{1}} = \mathcal{B},\\
					& x \in F_{\lambda_{i}} \setminus F_{\lambda_{i+1}} \quad \text{if and only if} \quad \lim_{t \rightarrow \infty} \frac{1}{t} \log \Vert T_{t} x \Vert_{\mathcal{B}} = \lambda_{i}.
				\end{split}
			\end{align}
			Also, there are numbers \(m_1, m_2, \ldots\) such that \(\operatorname{codim} F_{\lambda_j} = m_1 + \ldots + m_{j-1}\). In addition, for every \(i \geq 1\) with \(\lambda_i > \lambda_{i+1}\), there is a unique \(m_i\)-dimensional subspace \(H^i\) such that the following statements are true:
			\begin{itemize}
				\item[(i)] \textbf{(Invariance)}\ \ \(T_{t}(H^i) = H^i\) for every \(t \geq 0\).
				\item[(ii)] \textbf{(Splitting)}\ \ \(H^i \oplus F_{\lambda_{i+1}} = F_{\lambda_i}\). In particular,
				\begin{align*}
					\mathcal{B} = H^1 \oplus \cdots \oplus H^i \oplus F_{\lambda_{i+1}}.
				\end{align*}
				\item[(iii)] \textbf{('Fast-growing' subspace)}\ \ For each \(h \in H^{i} \setminus \{0\}\),
				\begin{align*}
					\lim_{t \rightarrow \infty} \frac{1}{t} \log \Vert T_t(h) \Vert_{\mathcal{B}} = \lambda_{i}
				\end{align*}
				and
				\begin{align*}
					\lim_{t \rightarrow \infty} \frac{1}{t} \log \Vert (T_t)^{-1}(h) \Vert_{\mathcal{B}} = -\lambda_{i}.
				\end{align*}
			\end{itemize}
		\end{proposition}
		\begin{proof}
			Cf.\eqref{FDCXX}.
		\end{proof}
		\begin{remark}\label{FVAACC}
			From Lemma \ref{MEEET}, for every \(j \geq 1\), we can find a unique and finite-dimensional subspace \(H^j\) such that item [iii] holds true. Note that if we fix a basis \(\{ \xi_k \}_{1 \leq k \leq m_j}\) for \(H^j\), then for an invertible matrix \(A_j\), our semigroup with respect to this basis can be represented by
			\begin{align}\label{VCXZAQ}
				\begin{split}
					&T_t: H^j \to H^j, \\
					&T_t = \exp(A_j t).
				\end{split}
			\end{align}
			In particular, if \(\tilde{\lambda}\) is a complex eigenvalue for \(A_j\), then we must have \(\operatorname{Re}(\tilde{\lambda}) = \lambda_j\)
		\end{remark}
		We can now state the following result.
		\begin{theorem}\label{LLTT}
			Assume \(\mathcal{B}\) is a Banach space, and let \((T_t)_{t \geq 0}\) be a \(\mathcal{C}_0\)-semigroup acting on \(\mathcal{B}\). Assume further that for \(t \geq r\), this semigroup is compact. Let \(L: \mathcal{D}(L) \to \mathcal{B}\) denote the infinitesimal generator of this semigroup. Then, if \(\mu \in P\sigma(L)\), there exists some \(j \geq 1\) and \(b \in \mathbb{R}\) such that \(\mu = \lambda_j + ib\), where \(\lambda_j\) is the Lyapunov exponent of the semigroup as stated in the previous proposition. Furthermore, for every \(\lambda_j \neq -\infty\), there exists \(b \in \mathbb{R}\) such that \(\lambda_j + ib \in P\sigma(L)\).
			
		\end{theorem}
		\begin{proof}
			Assume \(\mu = a + ib \in P\sigma(L)\). For two independent vectors \(\xi_1, \xi_2 \in \mathcal{D}(L)\) (if \(b \ne 0\)), with \(\xi_1 + i\xi_2 \ne 0\), we have \(L(\xi_1 + i\xi_2) = \mu (\xi_1 + i\xi_2)\). Therefore, by the definition of the semigroup, we have
			\[
			T_t(\xi_1 + i\xi_2) = \exp(\mu t)(\xi_1 + i\xi_2).
			\]
			Expanding this, we get
			\[
			T_t(\xi_1 + i\xi_2) = \exp((a + ib)t)(\xi_1 + i\xi_2) = \exp(at) \exp(ibt)(\xi_1 + i\xi_2).
			\]
			Using Euler's formula, \(\exp(ibt) = \cos(bt) + i \sin(bt)\), we can write
			\[
			T_t(\xi_1 + i\xi_2) = \exp(at) \left[(\cos(bt) + i \sin(bt)) (\xi_1 + i\xi_2)\right].
			\]
			Simplifying this, we obtain
			\[
			T_t(\xi_1 + i\xi_2) = \exp(at) \left[(\cos(bt) \xi_1 - \sin(bt) \xi_2) + i (\sin(bt) \xi_1 + \cos(bt) \xi_2)\right].
			\]
			Thus,
			\[
			T_t(\xi_1) = \exp(at) \left[\cos(bt) \xi_1 - \sin(bt) \xi_2 \right],
			\]
			\[
			T_t(\xi_2) = \exp(at) \left[\sin(bt) \xi_1 + \cos(bt) \xi_2 \right].
			\]
			Consequently,
			\[
			\lim_{t \rightarrow \infty} \frac{1}{t} \log \Vert T_t(\xi_1) \Vert = \lim_{t \rightarrow \infty} \frac{1}{t} \log \Vert T_t(\xi_2) \Vert = a.
			\]
			Therefore, from \eqref{nest}, for some \(j \geq 1\), we must have \(a = \lambda_j\). 
			Now, assume \(\lambda_j \ne -\infty\). From Remark \ref{FVAACC} for every \(j \geq 1\), we have \(H^j \subset \mathcal{D}(L)\). Assume \(\tilde{\lambda}_j = \lambda_j + ib\) is a complex eigenvalue for the matrix \(A_j\) in Remark \ref{FVAACC}. Then, for two nonzero elements \(\tilde{\xi}_1, \tilde{\xi}_2\) of \(H^j\), which are independent when \(b \ne 0\), we have
			\begin{align*}
				T_{t}(\tilde{\xi}_1+i\tilde{\xi}_2)=\exp(\lambda_jt+itb)(\tilde{\xi}_1+i\tilde{\xi}_2).
			\end{align*}
			Since $H^j\subset \mathcal{D}(L)$, this yields $L(\tilde{\xi}_1+i\tilde{\xi}_2)=(\lambda_{j}+ib)(\tilde{\xi}_1+i\tilde{\xi}_2)$.
		\end{proof}
		We then have the following corollary
		\begin{corollary}\label{TTRR}
			Assume that for every \(\lambda \in P\sigma(L)\), we have \(\operatorname{Re}(\lambda) < 0\). Then \((T_t)_{t \geq 0}\) is exponentially stable, i.e., for every \(\xi \in \mathcal{B}\), we have
			\[
			\limsup_{t \to \infty} \frac{1}{t} \log \|T_t \xi\|_{\mathcal{B}} \leq \lambda_1 < 0,
			\]
			where 
			\[
			\lambda_1 = \max \{ \operatorname{Re}(\lambda) \mid \lambda \in P\sigma(L) \}.
			\]
		\end{corollary}
		\begin{proof}
			A direct consequence of Lemma \ref{MEEET} and Theorem \ref{LLTT}.
		\end{proof}
		\begin{corollary}\label{TRFFFFD}
			Consider Equation \ref{linear delay} and recall the definition of the characteristic equation in Remark \ref{FVCC}. Assume that for every \(z \in \mathbb{C}\) such that \(\det(\Delta(z)) = 0\), we have \(\operatorname{Re}(z) < 0\). Then, for every initial value \(\xi \in C([-r,0], \mathbb{R}^n)\), we have
			\[
			\limsup_{t \to \infty} \frac{1}{t} \log \|T_t \xi\|_{\infty,[-r,0]} \leq \lambda_1 < 0,
			\]
			where \(\lambda_1\) is equal to
			\[
			\lambda_1 = \max \{ \operatorname{Re}(z) \mid \det(\Delta(z)) = 0 \}.
			\]
		\end{corollary}
		\begin{proof}
			A direct consequence of Remark \ref{FVCC} and Corollary \ref{TTRR}.
		\end{proof}
		\begin{remark}\label{ADCZXSA}
			Set \( \lambda_1 = \max \{ \operatorname{Re}(z) \mid \det(\Delta(z)) = 0 \} \), where \(\Delta\) is defined in \ref{characteristic}. Then, from Proposition \ref{MEEET} and Theorem \ref{LLTT}, \( \lambda_1 \) is actually the first Lyapunov exponent. For the first Lyapunov exponent, since our semigroup is actually a deterministic cocycle, thanks to \cite[Theorem 4.17]{GVRS22}, we have
			\begin{align}
				\lambda_1 = \lim_{t \rightarrow \infty} \frac{1}{t} \log \Vert T_{t} \Vert_{L(\mathcal{B},\mathcal{B})}.
			\end{align}
			Therefore, from this simple observation, we can even enhance Corollary \ref{TTRR}.
		\end{remark}
		
		As previously mentioned, the recent results are well-established and classical. However, our contribution is to provide new proofs, to the best of our knowledge, using the Multiplicative Ergodic Theorem, which are both very short and concise. In the next section, we will apply these results to our delay equations and, in particular, explore how our approach leads to several non-obvious pathwise stability results.
		
		\section{Pathwise stability results}\label{sec:stability}
		In this section, we obtained several exponential stability results for both linear and nonlinear cases. It is natural to expect that when the drift part of an equation is exponentially stable, the solution will remain stable under stochastic perturbations from the diffusion term, provided the diffusion is not too large. Demonstrating this stability, especially for non-Markovian equations and even for standard stochastic equations in finite-dimensional settings, is highly non-trivial. For delay equations, the situation is even more complex due to the infinite-dimensional Banach spaces in which we work. When the noise is driven by Brownian motion, we can use classical Itô integration, with many established tools for studying the asymptotic behavior of solutions, such as Lyapunov functions. However, a limitation of the classical approach is its lack of pathwise results. In this article, we explored a pathwise approach, raising the natural question of how these two approaches relate, particularly in the context of Brownian motion. Therefore, before presenting our stability results, we aim to bridge the gap between these approaches. This connection allows us to enlarge the space of initial values and obtain stronger stability results in a pathwise manner.
		\begin{assumption}
			In this section, unless explicitly stated otherwise, the noise \( B = (B^1, \ldots, B^d) \colon \mathbb{R} \to \mathbb{R}^d \) refers to a two-sided Brownian motion valued in \(\mathbb{R}^d\). This process is defined on a probability space \((\Omega, \mathcal{F}, \mathbb{P})\) and is adapted to a two-parameter filtration \((\mathcal{F}^t_s)_{s \leq t}\). Specifically, \((B_{t+s} - B_s)_{t \geq 0}\) constitutes a standard \((\mathcal{F}_s^{t+s})_{t \geq 0}\)-Brownian motion for every \(s \in \mathbb{R}\), with \(B_0 = 0\) almost surely. We also assume that \((\Omega, \mathcal{F}, \mathbb{P}, \theta)\) forms the ergodic metric dynamical system that exists according to Theorem \ref{thm:B_delayed_cocycle} for \(\mathbf{B}^{\text{It\=o}}\). Additionally, we will use \(\mathbf{B}^{\text{It\=o}}\) to solve our equations. According to \cite[Proposition 2.3]{GVRS22}, using \(\mathbf{B}^{\text{It\=o}}\) to compute the integrals yields results that align with classical It\=o integration.
		\end{assumption}
		\begin{lemma}\label{rough de}
			Assume $A$ and $\pi$ satisfy the same conditions as in Lemma \ref{PII}, and let either $G \in {L}(\mathbb{R}^n\times \mathbb{R}^n,L(\mathbb{R}^d,\mathbb{R}^n))$ or $G\in C(\mathbb{R}^n\times \mathbb{R}^n,L(\mathbb{R}^d,\mathbb{R}^n))$ be a Lipschitz function. Consider the following equation:
			\begin{align}\label{linear_delay_stochastic}
				\begin{split}
					\mathrm{d}y_{t} &= A\left(y_t, \int_{-r}^{0} y_{\theta + t} \pi(\mathrm{d}\theta)\right) \mathrm{d}t + G\left(y_{t}, y_{t-r}\right) \mathrm{d}B_t(\omega), \\
					y_{\theta} &= \xi_{\theta} \quad \text{for} \quad \theta \in [-r, 0], \quad \xi \in C([-r, 0], \mathbb{R}^n),
				\end{split}
			\end{align}
			where the integral is defined in the classical It\=o sense.
			We use \((y^{\xi}_{t}(\omega))_{t \geq 0}\) to denote the solution. For every initial value \(\xi \in C([-r,0], \mathbb{R}^n)\), let \((y_t^\xi(\omega), (y^{\xi}_t)^{\prime}(\omega)) = (y^\xi_t(\omega), G(y^{\xi}_t(\omega), y^{\xi}_{t-r}(\omega)))\). Then, on a set of full measure \(\Omega_{\xi}\), and for every \(\frac{1}{3}<\gamma < \frac{1}{2}\),
			\begin{align*}
				(y_{t+2r}^\xi(\omega), (y^{\xi})^{\prime}_{t+2r}(\omega))_{-r \leq t \leq 0} \in \mathscr{D}_{\mathbf{B}(\theta_{r}\omega)}^{\gamma}([-r,0]).
			\end{align*}
			In addition, \(\Vert (y_{.+2r}^\xi(\omega), (y^{\xi})^{\prime}_{.+2r}(\omega))\Vert_{\mathscr{D}_{\mathbf{B}(\theta_{2r}\omega)}^{\gamma}([-r,0])} \in {L}^{p}(\Omega)\) for every \(p \geq 1\).
		\end{lemma}
		\begin{proof}
			First, note that the existence and uniqueness of the solution for this equation follow from standard results in stochastic analysis. Given our assumption on \(G\), the Burkholder-Davis-Gundy inequality implies the existence of a constant \(\alpha_{m,1}\), which depends on \(A\), \(r\), \(\pi\), and \(G\), such that
			\begin{align*}
				\forall t\in [0,r]:\ \ \ & \mathbb{E}\left(\sup_{0 \leq \tau \leq t} \vert y_{\tau}^{\xi} \vert^{2m}\right) 
				\\&\leq \alpha_{m,1}\left(1 + \mathbb{E}\left(\sup_{\theta \in [-r,0]} \vert y_{\theta}^{\xi} \vert^{2m}\right) + \mathbb{E}\left(\int_{0}^{t} \sup_{0 \leq \tau \leq \sigma} \vert y_{\tau}^{\xi} \vert \, \mathrm{d}\sigma \right)^{2m} + \mathbb{E}\left(\int_{0}^{t} \sup_{0 \leq \tau \leq \sigma} \vert y_{\tau}^{\xi} \vert^2 \, \mathrm{d}\sigma\right)^{m}\right).
			\end{align*}
			
			By applying Grönwall's lemma and performing some basic calculations, we obtain that, for a constant \(\alpha_{m,2}\),
			\begin{align}\label{UNI}
				\mathbb{E}\left(\sup_{\theta \in [0,r]} \vert y_{\theta}^{\xi} \vert^{2m}\right) 
				\leq \alpha_{m,2}\left(1 + \mathbb{E}\left(\sup_{\theta \in [-r,0]} \vert y_{\theta}^{\xi} \vert^{2m}\right)\right) 
				= \alpha_{m,2}\left(1 + \Vert \xi \Vert_{\infty,[-r,0]}^{2m}\right).
			\end{align}
			Therefore, by repeating the same argument over the interval \([r, 2r]\) and using the previous inequality, for a constant \(\alpha_{m,3}\), we obtain
			\begin{align}\label{UNI2}
				\mathbb{E}\left(\sup_{\theta \in [r,2r]} \vert y_{\tau}^{\xi} \vert^{2m}\right) 
				\leq \alpha_{m,2}\left(1 + \mathbb{E}\left(\sup_{\theta \in [0,r]} \vert y_{\theta}^{\xi} \vert^{2m}\right)\right) 
				\leq \alpha_{m,3}\left(1 + \Vert \xi \Vert_{\infty,[-r,0]}^{2m}\right).
			\end{align}
			
			Again, from the Burkholder-Davis-Gundy inequality, for a constant \(\beta_{m,1}\), we have
			\begin{align}\label{AAATTG}
				\begin{split}
					\forall s,t \in [0,2r]: \quad &\mathbb{E}\left(\vert (\delta y^\xi)_{s,t} \vert^{2m}\right) 
					\\	&\leq \beta_{m,1}\left(\mathbb{E}\left(\int_{s}^{t} \sup_{\theta \in [-r,0]} \vert y^{\xi}_{\sigma+\theta} \vert \, \mathrm{d}\sigma\right)^{2m} + \mathbb{E}\left(\int_{s}^{t} \left(1 + \sup_{\theta \in [-r,0]} \vert y^{\xi}_{\sigma+\theta} \vert^2 \right) \, \mathrm{d}\sigma\right)^{m}\right).
				\end{split}
			\end{align}
			
			The latter inequality, combined with \eqref{UNI2}, yields that, for a constant \(\beta_{m,2}\),
			\begin{align}\label{NMMMLIO}
				\forall s,t \in [0,2r]: \quad \mathbb{E}\left(\vert (\delta y^\xi)_{s,t} \vert^{2m}\right) 
				\leq \beta_{m,2} (t-s)^{m} \left(1 + \Vert \xi \Vert_{\infty,[-r,0]}^{2m}\right).
			\end{align}
			
			Again, from the Burkholder-Davis-Gundy inequality, for a constant \(\gamma_{m,1}\),
			\begin{align*}
				\forall s,t \in [r,2r]: \quad &\mathbb{E}\left(\left\vert (\delta y^\xi)_{s,t} - G(y^{\xi}_s, y^{\xi}_{s-r})(\delta B)_{s,t}\right\vert^{2m}\right)
				\\&\leq \gamma_{m,1} \left(\mathbb{E}\left(\int_{s}^{t} \sup_{\theta \in [-r,0]} \vert y^{\xi}_{\sigma+\theta} \vert \, \mathrm{d}\sigma\right)^{2m}  \quad + \mathbb{E}\left(\int_{s}^{t} \left(\vert (\delta y^{\xi})_{s,\sigma} \vert^2 + \vert (\delta y^{\xi})_{s-r,\sigma-r} \vert^2\right) \, \mathrm{d}\sigma\right)^m\right).
			\end{align*}
			By applying \eqref{UNI2} and \eqref{NMMMLIO}, for a constant \({\gamma}_{m,2}\),
			\begin{align}\label{TTR2}
				\forall s,t\in [r,2r] :\mathbb{E}\left(\left\vert (\delta y^\xi)_{s,t}-  G(y^{\xi}_s,y^{\xi}_{s-r})(\delta B)_{s,t}\right\vert^{2m}\right)\leq {\gamma}_{m,2}(t-s)^{2m}(1+\Vert\xi\Vert_{\infty,[-r,0]}^{2m})
			\end{align}
			Here we are fixing a \(\xi \in C([-r,0], \mathbb{R}^n)\) and \(m \in \mathbb{N}\) can be selected arbitrarily large. Set \((y^{\xi})^{\#}_{s,t} := (\delta y^{\xi})_{s,t} - G(y^{\xi}_s, y^{\xi}_{s-r})(\delta B)_{s,t}\) and \((y^{\xi})^{\prime}_s := G(y^{\xi}_s, y^{\xi}_{s-r})\). Then, by the Kolmogorov continuity theorem, for a probability space \(\Omega_{\xi}\) with full measure and $\theta_r$-invariant, we can conclude that for every $\omega\in\Omega_{\xi} $
			\begin{align*}
				(y_{t+2r}^\xi(\omega), (y^{\xi})^{\prime}_{t+2r}(\omega))_{-r \leq t \leq 0} \in \mathscr{D}_{\mathbf{B}(\theta_{2r} \omega)}^{\gamma}([-r,0]).
			\end{align*} 
		\end{proof}
		\begin{remark}
			Note that in the case where \( G(0,0) = 0 \), for the inequalities \eqref{UNI2}, \eqref{NMMMLIO}, and \eqref{TTR2}, on their right-hand sides, we can replace \((1 + \Vert \xi \Vert_{\infty,[-r,0]}^{2m})\) with \(\Vert \xi \Vert_{\infty,[-r,0]}^{2m}\).
		\end{remark}
		\begin{comment}
			In practice, it is more feasible to obtain moment stability with respect to the underlying probability spaces for delay equations (see \cite{LM07, BGY16}). Our aim is to demonstrate that if moment stability is achieved, then pathwise stability also holds. Let us first start with the following simple lemma:
		\end{comment}
		\begin{lemma}\label{deteministic}
			Consider the equation \eqref{linear delay} and recall the definition of $\Delta$ from \eqref{characteristic}. Set
			\begin{align}\label{TRFVV}
				\lambda_1 = \max \{ \operatorname{Re}(z) \mid \det(\Delta(z)) = 0 \}.
			\end{align}
			Then, for every $\xi \in C([-r,0], \mathbb{R}^n)$, we have:
			\begin{align}\label{GBVZCXS}
				\limsup_{t \to \infty} \frac{1}{t} \log \| T_{t} \xi \|_{C^1([-r,0], \mathbb{R}^n)} \leq \lambda_1.
			\end{align}
			Here, \((T_t)_{t \geq 0}\) denotes the semigroup solution as defined in \eqref{delay-semigroup}, and \(\|\cdot\|_{C^1([-r,0], \mathbb{R}^n)}\) is the usual \(C^1([-r,0], \mathbb{R}^n)\) norm, i.e., the supremum norm of the path and its derivative. Also 
			\begin{align}\label{GBVZCXS1}
				\lim_{t \geq r, t\rightarrow\infty}\frac{1}{t}\log\sup_{\|\xi\|_{\infty,[-r,0]}=1}\log \| T_{t} \xi \|_{C^1([-r,0], \mathbb{R}^n)} = \lambda_{1}.
			\end{align}
		\end{lemma}
		\begin{proof}
			First note that from Theorem \ref{LLTT} and Remark \ref{ADCZXSA}, 
			\begin{align}\label{BNM<LO}
				\limsup_{t \to \infty} \frac{1}{t} \log \|T_t \xi\|_{\infty,[-r,0]} \leq \lim_{t \to \infty}\frac{1}{t}\log\Vert T_{t}\Vert_{L(C([-r,0], \mathbb{R}^n),C([-r,0], \mathbb{R}^n))}=\lambda_1 .
			\end{align}
			Also for $t\geq r$,  and $\sigma\in [-r,0]$
			\begin{align}
				\frac{\mathrm{d} T_{t}(\xi)}{\mathrm{d}\sigma}(\sigma)=A\big((T_{t}\xi)(\sigma),\int_{-r}^{0}(T_{t}\xi)(\sigma+\theta)\pi(\mathrm{d}\theta)\big).
			\end{align}
			This in particular implies that for a constant \( M \),
			\begin{align}\label{AZA8}
				\sup_{\sigma \in [-r,0]} \left| \frac{\mathrm{d} T_{t}(\xi)}{\mathrm{d} \sigma}(\sigma) \right| \leq M \left( \Vert T_{t} \xi \Vert_{\infty,[-r,0]} + \Vert T_{t-r} \xi \Vert_{\infty,[-r,0]} \right).
			\end{align}
			Therefore, from \eqref{BNM<LO},
			\begin{align*}
				\limsup_{t \to \infty} \frac{1}{t} \log \sup_{\sigma \in [-r,0]} \left| \frac{\mathrm{d} T_{t}(\xi)}{\mathrm{d} \sigma}(\sigma) \right| \leq \lambda_1.
			\end{align*}
			Consequently, \ref{GBVZCXS} is proved. Similarly, we can argue to obtain \ref{GBVZCXS1}.
		\end{proof}
		In the following result, we obtain a pathwise stability result in the linear case.
		\begin{proposition}\label{BBSTAB}
			Assume that \( A \) and \( \pi \) satisfy the same conditions as in Lemma \ref{PII}, and let \( G \in {L}(\mathbb{R}^n \times \mathbb{R}^n, {L}(\mathbb{R}^d, \mathbb{R}^n)) \). Consider the following linear stochastic differential equation:
			\[
			\begin{aligned}
				\mathrm{d}y_t &= A\left( y_t, \int_{-r}^{0} y_{\theta+t} \, \pi(\mathrm{d}\theta) \right) \mathrm{d}t + \epsilon G(y_t, y_{t-r}) \, \mathrm{d}B_t(\omega), \\
				y_{\theta} &= \xi_{\theta} \quad \text{for} \quad \theta \in [-r, 0], \quad \xi \in C([-r, 0], \mathbb{R}^n),
			\end{aligned}
			\]
			where \( y^{\xi, \epsilon}_t \) denotes the solution. Let \( \lambda_1 \) be as defined in \eqref{TRFVV}, and assume \( \lambda_1 < 0 \). Then, there exists \( \epsilon_0 > 0 \) such that for every \( 0 < \epsilon < \epsilon_0 \), there exists \( \lambda^\epsilon < 0 \) such that for any \( \xi \in C([-r, 0], \mathbb{R}^n) \), the following holds:
			\[
			\limsup_{n \to \infty} \frac{1}{t} \log | y^{\xi, \epsilon}_t(\omega) |\leq \lambda^\epsilon, \quad \text{almost surely}.
			\]
			This yields in particular a pathwise stability result on $C([-r,0],\mathbb{R}^n)$
		\end{proposition}
		\begin{proof}
			By Lemma \eqref{rough de}, for every $\xi\in C([-r,0],\mathbb{ R}^n)$ and $\frac{1}{3}<\beta<\gamma<\frac{1}{2}$, on set of full measure, 
			\begin{align*}
				(y_{t+2r}^{\xi,\epsilon}(\omega), G(y_{t+2r}^{\xi,\epsilon}(\omega),\xi_t)_{-r \leq t \leq 0} \in \mathscr{D}_{\mathbf{B}^{\text{It\=o}}(\theta_{2r} \omega)}^{\gamma}([-r,0])\subset \mathscr{D}_{\mathbf{B}^{\text{It\=o}}(\theta_{2r} \omega)}^{\gamma,\beta}([-r,0])\ .
			\end{align*} 
			Therefore, since we are interested in the asymptotic behavior of the solution, we can assume that our initial value is controlled by the Brownian noise, and we solve the equation using the rough path approach by considering \(\mathbf{B}^{\text{It\=o}}\). As \eqref{TVBN}, we use $(\varphi_{\omega}^{m,\epsilon}(\xi))_{m\in\N}$, to denote our solution, where $\varphi^{m,\epsilon}_{\omega}(\xi):=\left(y_{(m-1)r+t}^{\xi,\epsilon}(\omega)\right)_{0\leq t\leq r}$. In this case, our linear equations now satisfy the assumptions of the Multiplicative Ergodic Theorem, cf. Proposition \ref{MEEETTTT}. Let us use $\lambda^\epsilon$, to denote the first Lyapunov exponent which is deterministic and finite. Therefore, be definition of first Lyapunov exponent on set of full measure $\Omega^\xi$ and for $E_{\omega}:=\mathscr{D}_{\mathbf{B}^{\text{It\=o}}(\omega)}^{\gamma,\beta}([-r,0])$
			\begin{align*}
				\limsup_{t\rightarrow\infty}\frac{1}{t} \log | y^{\xi, \epsilon}_t(\omega) | \leq \lim_{m\rightarrow \infty}\frac{1}{mr}\log\Vert \varphi_{\omega}^{m,\epsilon}(\xi)\Vert_{E_{\theta_{mr} \omega}}\leq \lambda^\epsilon.
			\end{align*}
			From \cite[ Theorem 3.3.2]{Arn98},  and \cite[Theorem 4.17]{GVRS22},
			\begin{align}\label{FR_1}
				\lambda^\epsilon= \inf_{m \geq 0} \frac{1}{mr} \int_{\Omega}\log \Vert \varphi_{\omega}^{m,\epsilon}\Vert_{L(E_{ \omega},E_{\theta_{mr} \omega})}\ \mathbb{P}(\mathrm{d}\omega)=\lim_{m\rightarrow\infty}\frac{1}{mr} \int_{\Omega}\log \Vert \varphi_{\omega}^{m,\epsilon}\Vert_{L(E_{ \omega},E_{\theta_{mr} \omega})}\ \mathbb{P}(\mathrm{d}\omega).
			\end{align}
			Note that for every \( m \in \mathbb{N} \), \( \log \Vert \varphi_{\omega}^{m,\epsilon} \Vert_{L(E_{ \omega},E_{\theta_{mr} \omega})} \) depends continuously on \( \epsilon \). Additionally, we can find a uniformly integrable bound that dominates \( \log \Vert \varphi_{\omega}^{m,\epsilon} \Vert_{L(E_{ \omega},E_{\theta_{mr} \omega})} \) for every \( \epsilon \in [0,1] \). Therefore, by the Dominated Convergence Theorem for every $m\in\N$,
			\begin{align}\label{DSCXVSVSA}
				\lim_{\epsilon\rightarrow 0}\frac{1}{mr} \int_{\Omega}\log \Vert \varphi_{\omega}^{m,\epsilon}\Vert_{L(E_{ \omega},E_{\theta_{mr} \omega})}\ \mathbb{P}(\mathrm{d}\omega)=\frac{1}{mr}\int_{\Omega}\log \Vert T_{mr}\Vert_{L(E_{ \omega},E_{\theta_{mr} \omega})}\ \mathbb{P}(\mathrm{d}\omega),
			\end{align}
			where \( (T_{t})_{t \geq 0} \) is the semigroup generated by \( A \). Recall that for \( t \geq r \), we have \( T_{t}\left(C([-r,0],\mathbb{R}^n)\right) \subset C^{1}([-r,0],\mathbb{R}^n) \). Therefore, for any \( \xi \in E_\omega \), by definition, there exists a deterministic constant \( R \) such that
			\[
			\| T_{mr}\xi \|_{E_{\theta_{mr}\omega}} \leq R \| T_{mr}\xi \|_{C^{1}([-r,0],\mathbb{R}^n)},
			\]
			for every \( \omega \in \Omega \). Also, it is evident that \( \|\xi\|_{\infty,[-r,0]} \leq \|\xi\|_{E_\omega} \) holds for every \( \omega \in \Omega \). Therefore, we can conclude that
			\[
			\log \| T_{mr}\|_{L(E_{ \omega},E_{\theta_{mr} \omega})} \leq \log\left(R \sup_{\|\xi\|_{\infty,[-r,0]}=1} \| T_{mr} \xi \|_{C^1([-r,0], \mathbb{R}^n)}\right).
			\]
			In particular, from \eqref{GBVZCXS1},
			\[
			\limsup_{m \rightarrow \infty} \frac{1}{mr} \int_{\Omega} \log \| T_{mr}\|_{L(E_{ \omega},E_{\theta_{mr} \omega})}\ \mathbb{P}(\mathrm{d}\omega) \leq \lambda_1.
			\]
			Therefore, since \( \lambda_1 < 0 \), we can choose \( m_0 \) large enough such that \( \frac{1}{m_0r} \int_{\Omega} \log \| T_{m_0r}\|_{L(E_{ \omega},E_{\theta_{m_0r} \omega})}\ \mathbb{P}(\mathrm{d}\omega) < 0 \). Thanks to \eqref{FR_1} and \eqref{DSCXVSVSA}, by considering \( \epsilon_0 \) small, we can argue that for every \( \epsilon \leq \epsilon_0 \),
			\[
			\lambda^{\epsilon} \leq \frac{1}{m_0r} \int_{\Omega} \log \| \varphi_{\omega}^{m_0,\epsilon}\|_{L(E_{ \omega},E_{\theta_{mr} \omega})}\ \mathbb{P}(\mathrm{d}\omega) < 0.
			\]
			This finishes the proof.
		\end{proof}
		So far, we have considered Brownian motion. Note that since we have the Markov property, we can obtain similar results using classical tools, such as the Borel-Cantelli theorem, in stochastic analysis and probability. However, an advantage of our theory is that it can also be applied when the linear equation is driven by a fractional Brownian motion, where the Markov property is lost. In this case, since we initially solved our equation path-wise, to compensate for this issue, we must take the initial value to be smoother. We can state the following results, which are analogous to the previous result.
		\begin{proposition}\label{ASsSCAAS}
			Assume \( \frac{1}{3}<\beta< \gamma< H\) and let \(\mathbf{B}\) be the delayed \(\gamma\)-rough path cocycle constructed in Theorem \ref{thm:B_delayed_cocycle}. Consider the same conditions on \(A\), \(\pi\), and \(G\) as specified in Proposition \ref{BBSTAB}. We consider the following linear rough delay equation:
			\begin{equation}
				\begin{aligned}
					\mathrm{d}y_t &= A \left( y_t, \int_{-r}^{0} y_{\theta+t} \, \pi(\mathrm{d}\theta) \right) \mathrm{d}t + \epsilon G(y_t, y_{t-r}) \, \mathrm{d}\mathbf{B}_t(\omega), \\
					y_{\theta} &= \xi_{\theta} \quad \text{for} \quad \theta \in [-r, 0], \quad \xi \in \mathscr{D}_{\mathbf{B}(\omega)}^{\beta, \gamma}([-r,0]),
				\end{aligned}
			\end{equation}
			where \( y^{\xi, \epsilon}_t(\omega) \) denotes the solution. Let \( \lambda_1 \) be as defined in \eqref{TRFVV}, and assume \( \lambda_1 < 0 \). Then, there exists \( \epsilon_0 > 0 \) such that for every \( 0 < \epsilon < \epsilon_0 \), there exists \( \lambda^\epsilon < 0 \) such that for any \( \xi \in  \mathscr{D}_{\mathbf{B}(\omega)}^{\beta, \gamma}([-r,0]) \), the following holds:
			\begin{equation}
				\limsup_{t \to \infty} \frac{1}{t} \log | y^{\xi, \epsilon}_t(\omega) | \leq \lambda^\epsilon.
			\end{equation}
			In particular, our stability result holds for every \(\xi \in C^{2\gamma}([-r, 0], \mathbb{R}^n)\), where \(\frac{1}{3} < \gamma < H\) and \(\gamma\) can be chosen arbitrarily close to \(\frac{1}{3}\).
		\end{proposition}
		\begin{proof}
			The proof is similar to that of Proposition \ref{BBSTAB}, with the primary difference being the initial values. In the previous result, due to the Markov property, we could take the initial values in \(C([-r, 0], \mathbb{R}^n)\). However, for fractional Brownian motions, the noise is less regular and we do not have the Markov property, so we take the initial values in \(\mathscr{D}_{\mathbf{B}(\omega)}^{\beta, \gamma}([-r,0])\). Note that our final claim is obvious since \(C^{2\gamma}([-r, 0], \mathbb{R}^n) \subset \mathscr{D}_{\mathbf{B}(\omega)}^{\beta, \gamma}([-r,0])\).
		\end{proof}
		
		\subsection{The nonlinear case}
		We can now obtain similar stability results for the nonlinear case; however, in this situation, we can only derive a local result. The key component here is our stable manifold theorem (cf. Theorem \ref{thm:stable_manifold}), which applies locally and leads to Corollary \ref{STTTAASDSF}. Unlike the linear case, where the stable manifold exists globally, in this scenario, particularly for the Brownian case, we cannot extend the space of initial values to achieve exponential stability results in \(C([-r,0],\mathbb{R}^n)\). However, when the initial values are sufficiently smooth(i.e. ${C^{\frac{2}{3}+\epsilon}([-r,0],\R^n)}$), local stability can still be established.
		\begin{proposition}\label{AAASSW}
			Assume \( \frac{1}{3} < \beta < \gamma < H \), and let \(\mathbf{B}\) be the delayed \(\gamma\)-rough path cocycle constructed in Theorem \ref{thm:B_delayed_cocycle}. Consider the following equation:
			\[
			\begin{aligned}
				\mathrm{d}y_t &= A\left( y_t, \int_{-r}^{0} y_{\theta+t} \, \pi(\mathrm{d}\theta) \right) \mathrm{d}t + G(y_t, y_{t-r}) \, \mathrm{d}\mathbf{B}_t(\omega), \\
				y_{\theta} &= \xi_{\theta}, \quad \text{for } \theta \in [-r, 0], \quad \xi \in \mathscr{D}_{\mathbf{B}(\omega)}^{\beta, \gamma}([-r,0]),
			\end{aligned}
			\]
			where we assume that the conditions in Theorem \ref{ASXXZA} regarding \(A\), \(\pi\), and \(G\) are satisfied, and that \(G(0,0) = 0\). Let \( \lambda_1 \) be as defined in \eqref{TRFVV}. Then, there exists \(l > 0\) such that if \(\Vert D_{(0,0)}G \Vert < l\), the zero path is locally exponentially stable. More specifically, we can find a random variable \(R^{l}: \Omega \to (0, \infty)\) such that
			\[
			\liminf_{p \to \infty} \frac{1}{p} \log R^{l}(\theta_{pr} \omega) \geq 0,
			\]
			and for every \(\xi \in \mathscr{D}_{\mathbf{B}(\omega)}^{\beta, \gamma}([-r,0])\) with
			\[
			\Vert \xi \Vert_{\mathscr{D}_{\mathbf{B}(\omega)}^{\beta, \gamma}([-r,0])} \leq R^{l}(\omega),
			\]
			the following holds:
			\[
			\limsup_{t \to \infty} \frac{1}{t} \log | y^{\xi}_t(\omega) | \leq \lambda^\epsilon < 0,
			\]
			where $y^{\xi}_{t}(\omega)$ is our solution. In particular, our local stability result holds for every \(\xi \in C^{2\gamma}([-r, 0], \mathbb{R}^n)\), where \(\frac{1}{3} < \gamma < H\) and \(\gamma\) can be chosen arbitrarily close to \(\frac{1}{3}\).
		\end{proposition}
		\begin{proof}
			Fisrt note that by our assumption the zero path is stationary solution to our equation. In addition the linearzed equation around the  zero path is given by
			\begin{align}\label{aaaASX}
				\begin{split}
					\mathrm{d}y_t &= A\left( y_t, \int_{-r}^{0} y_{\theta+t} \, \pi(\mathrm{d}\theta) \right) \mathrm{d}t + D_{(0,0)}G(y_t, y_{t-r}) \, \mathrm{d}\mathbf{B}_t(\omega), \\
					y_{\theta} &= \xi_{\theta} \quad \text{for } \theta \in [-r, 0], \quad \xi \in \mathscr{D}_{\mathbf{B}(\omega)}^{\beta, \gamma}([-r,0]).
				\end{split}
			\end{align}
			Our claim follows from Corollary \ref{STTTAASDSF}, if we show that the first Lyapunov exponent that arises from Equation \eqref{aaaASX} is negative when \(\Vert D_{(0,0)}G \Vert_{L(\mathbb{R}^n \times \mathbb{R}^n, \mathbb{R}^n)}\) is small, which follows from Proposition \ref{ASsSCAAS}.
		\end{proof}
		
		As we have mentioned, all the results of this paper, including Proposition \ref{AAASSW}, hold when our process is a Brownian motion and we use the rough path approach by considering \(\mathbf{B}^{\text{Ito}}\), which coincides with classical stochastic analysis. Our stability results, regardless of considering fractional Brownian motions, also apply to the Brownian case, showing that the solutions decay pathwise exponentially in the norm \(\Vert \cdot \Vert_{\mathscr{D}_{\mathbf{B}(\omega)}^{\beta, \gamma}([-r,0])}\), which is stronger compared to the \(C([-r,0], \mathbb{R}^d)\) norm. Obtaining similar results using only classical It\=o calculus in the Brownian case seems to be very challenging.
		
		\appendix

		\section{Elements of Malliavin Calculus and fractional calculus}
		In this section, we provide some necessary definitions and theorems from Malliavin's calculus and fractional integrals. Most of these results, along with their proofs, can be found in \cite{YM08} and \cite{Nua05}. Throughout, we will assume that $0<H<\frac{1}{2}$.
		Let's begin with fractional calculus. For $0<\alpha<1$, the right and left-sided Riemann–Liouville fractional integrals on $\mathbb{R}$ are defined as follows:
		\begin{align}\label{fractional integral}
			&(I^{\alpha}_{-}f)(x):=\frac{1}{\Gamma(\alpha)}\int_{x}^{\infty}f(t)(t-x)^{\alpha-1}\mathrm{d}t,\\
			&(I^{\alpha}_{+}f)(x):=\frac{1}{\Gamma(\alpha)}\int_{-\infty}^{x}f(t)(x-t)^{\alpha-1}\mathrm{d}t.
		\end{align}
		receptively.
		We say that $f\in D(I^{\alpha}_{-})$, if \eqref{fractional integral}is well-defined for a.s  $x\in\mathbb{R}$. It is known that for $0\leq p<\frac{1}{\alpha}$, we have $L^{p}(\mathbb{R})\subset D(I^{\alpha}_{-}) $.
		% $. 
		The right and left-sided Riemann– Liouville fractional derivatives of order $\alpha$, are defined by
		\begin{align*}
			&(I^{-\alpha}_{-}f)(x):=\frac{-1}{\Gamma(1-\alpha)}\frac{\mathrm{d}}{\mathrm{d}x}\int_{x}^{\infty}f(t)(t-x)^{-\alpha}\mathrm{d}t,\\
			&(I^{-\alpha}_{+}f)(x):=\frac{1}{\Gamma(1-\alpha)}\frac{\mathrm{d}}{\mathrm{d}x}\int_{-\infty}^{x}f(t)(x-t)^{-\alpha}\mathrm{d}t
		\end{align*}
		repetitively.
		For $p\geq 1$, by $I^{\alpha}_{-}(L^{p}(\mathbb{R}))$ ($I^{\alpha}_{+}(L^{p}(\mathbb{R}))$), we mean the class of functions such that $f=I^{\alpha}_{-}(\phi)$ for some $\phi\in L^{p}(\mathbb{R})$. If $p>1$, then for every $f\in I^{\alpha}_{-}(L^{p}(\mathbb{R}))$ ($f\in I^{\alpha}_{+}(L^{p}(\mathbb{R}))$), the right and left-sided Riemann–Liouville derivatives are equal to the Marchaud fractional derivatives, which are defined by:
		\begin{align*}
			&(\mathcal{D}^{\alpha}_{-}f)(x):=\frac{\alpha}{\Gamma(1-\alpha)}\int_{0}^{\infty}\frac{f(x)-f(x+t)}{t^{1+\alpha}}\mathrm{d}t,\\
			&	(\mathcal{D}^{\alpha}_{+}f)(x):=\frac{\alpha}{\Gamma(1-\alpha)}\int_{0}^{\infty}\frac{f(x)-f(x-t)}{t^{1+\alpha}}\mathrm{d}t .
		\end{align*}
		It is easy to check that $\langle I^{-\alpha}_{-}f,g\rangle=\langle f,I^{-\alpha}_{+}g\rangle$ .
		Obviously if $ \text{supp}(f)\subset (-\infty,b) $, then $\text{supp}(\mathcal{D}^{\alpha}_{-}f)\subset (-\infty,b)$  and 
		\begin{align*}
			x\in(-\infty,b): \ \ (\mathcal{D}^{\alpha}_{-}f)(x)=\frac{f(x)}{\Gamma(1-\alpha)(b-x)^{\alpha}}+\frac{\alpha}{\Gamma(1-\alpha)}\int_{x}^{b}\frac{f(x)-f(t)}{(t-x)^{1+\alpha}}\mathrm{d}t  .
		\end{align*}
		For $\alpha=\frac{1}{2}-H$ and $C_H=\big(\int_{0}^{\infty}((1+s)^{-\alpha}-s^{-\alpha})^2\mathrm{d}s+\frac{1}{2H}\big)^{-\frac{1}{2}}\Gamma(H+\frac{1}{2})$,  we set
		\begin{align*}
			\mathcal{M}^{H}_-(f):=C_{H}I^{-\alpha}_{-}(f).
		\end{align*}
		We define \( \mathcal{H}=L_{2}^{H}(\mathbb{R}) := \{ f : \mathcal{M}^{H}_-(f) \in L^{2}(\mathbb{R}) \} \). For \( H \in \left(0, \frac{1}{2}\right) \), the operator \( \mathcal{M}^{H} \) is an isometric isomorphism from \(  \mathcal{H} \) to \( L^{2}(\mathbb{R}) \). In fact, \(  \mathcal{H} \) is a Hilbert space, and it can be characterized as the closure of the set of step functions with respect to the norm induced by \(  \mathcal{H} \) . In more details,
		suppose $\mathcal{E}$ be the set of step-functions from $\mathbb{R}$ to $\mathbb{R}^d$. Define the Hilbert space $ \mathcal{H}^{ d} $ as the closure of step-functions with the following inner product
		\begin{align*}
			\langle(\chi_{[s_{1},t_{1}]},...,\chi_{[s_{d},t_{d}]}),&(\chi_{[u_{1},v_{1}]},...,\chi_{[u_{d},v_{d}]})\rangle=\\ &\sum_{1\leqslant i\leqslant d}\big{(}R_{H}(s_{i},u_{i})+R_{H}(t_{i},v_{i})-R_{H}(s_{i},v_{i})-R_{H}(t_{i},u_{i})\big{)} .
		\end{align*}
		Where
		\begin{align*}
			R_{H}(s,t)=\frac{1}{2}(\vert t\vert^{2H}+\vert s\vert^{2H}-\vert t-s\vert^{2H})\ .
		\end{align*}
		We also can define the isonormal stochastic Gaussian  process  $B=\lbrace B(\phi)=(B^{i}(\phi^i))_{1\leq i\leq d}, \phi=(\phi^{i})_{1\leq i\leq d}\in\mathcal{H}^d\rbrace $ on a complete probability space $ (\Omega,\mathcal{F},\mathbb{P}) $ such that 
		\begin{align*}
			\E(\langle B(\phi),B(\psi)\rangle)=\langle\phi,\psi\rangle_{\mathcal{H}^{ d}}=\sum_{1\leq i\leq d}\langle\phi^i,\psi^i\rangle_{\mathcal{H}}, \ \ \phi,\psi\in\mathcal{H}^{ d} .
		\end{align*}
		In particular, $ B(\chi_{[s_{1},t_{1}]},...,\chi_{[s_{d},t_{d}]}) =(B^{i}(t_{i})-B^{i}(s_{i}))_{1\leqslant i\leqslant d}$ .
		Note that for $\phi=(\phi^i)_{1\leq i\leq d}$ and  $\psi=(\psi^i)_{1\leq i\leq d}$, we have
		\begin{align}\label{FD}
			\phi,\psi\in\mathcal{H}^{ d}: \ \ 	\langle\phi,\psi\rangle_{\mathcal{H}^{ d}}=\langle D^{\frac{1}{2}-H}_{-}\phi,D^{\frac{1}{2}-H}_{-}\psi\rangle_{L^{2}(\mathbb{R}^{d})}:=\sum_{1\leq i\leq d}\langle D^{\frac{1}{2}-H}_{-}\phi^{i},D^{\frac{1}{2}-H}_{-}\psi^{i}\rangle_{L^{2}(\mathbb{R})}  .
		\end{align}
		Consequently, for $ \phi\in \mathcal{H}^{ d} $ with $ \text{supp}(\phi)\subset (a,b)$
		\begin{align}\label{AA_1}
			\Vert\phi\Vert^{2}_{\mathcal{H}^{ d}}\leq M\left[\int_{a}^{b}\frac{\vert\phi(u)\vert^2}{(b-u)^{2\alpha}}\mathrm{d}u+\int_{-\infty}^{a}\left\vert\int_{a}^{b}\frac{\phi(x)}{(x-u)^{1+\alpha}}\mathrm{d}x\right\vert^{2}\mathrm{d} u +\int_{a}^{b}\left(\int_{u}^{b}\frac{\vert\phi(u)-\phi(x)\vert}{(x-u)^{1+\alpha}}\mathrm{d}x\right)^{2}\mathrm{d} u\right].
		\end{align}
		where $M$ is a constant.
		We set
		\begin{align}\label{BB_1}
			\Vert \phi\Vert^{2}_{\vert\mathcal{H}\vert^{d}}:=\int_{a}^{b}\frac{\vert\phi(u)\vert^2}{(b-u)^{2\alpha}}\mathrm{d}u+\int_{-\infty}^{a}\left|\int_{a}^{b}\frac{\phi(x)}{(x-u)^{1+\alpha}}\mathrm{d}x\right|^{2}\mathrm{d} u +\int_{a}^{b}\left(\int_{u}^{b}\frac{\vert\phi(u)-\phi(x)\vert}{(x-u)^{1+\alpha}}\mathrm{d}x\right)^{2}\mathrm{d} u .
		\end{align}
		%	\begin{align}\label{AA}
			%		\begin{split}
				%			\E(\Vert\phi\Vert^{2}_{\mathcal{H}})\leqslant M\big{[}\int_{a}^{b}\frac{\E(\Vert\phi(u)\Vert^{2})}{(b-u)^{2\alpha}}\mathrm{d}u+&\int_{-\infty}^{a}\big{(}\int_{a}^{b}\frac{(\E(\Vert\phi(x)\Vert^{2}))^{\frac{1}{2}}}{(x-u)^{1+\alpha}}dx\big{)}^{2}\mathrm{d}u+\\&\int_{a}^{b}\big{(}\int_{u}^{b}\frac{(\E(\Vert\phi(u)-\phi(x)\Vert^{2}))^{\frac{1}{2}}}{(x-u)^{1+\alpha}}dx\big{)}^{2}\mathrm{d}u\big{]}.
				%		\end{split}
			%	\end{align}
		%%%%%%%%%%%%%%%%%%%%%%%%%%%%%%%%%%%%%%%%%%%%%%%%%%%%%%%%%%%%%%%%%%%%%%%%%%%%%%%5
		Let ${\mathcal{G}} $ be the set of smooth cylindrical random variables
		\begin{align*}
			F=f(B(\phi_{1}),...,B(\phi_{k})), \ \ \ \phi_{i}\in\mathcal{H}^{ d}, \ \ 1\leqslant i\leqslant k,
		\end{align*}
		where $ f \in C^{\infty}(\R^{dm},\R)$ is bounded with bounded derivatives. Define the derivative operator by
		\begin{align*}
			D^{B}F=(D^{B^i}F)_{1\leq i\leq d}=\sum_{1\leqslant j\leqslant k}\frac{\partial f}{\partial x_{j}}(B(\phi_{1}),...,B(\phi_{k}))\phi_{j}\in\mathcal{H}^{ d}.
		\end{align*}
		%	Where $D^{B}_{r}F:=(D^{B}F)(r)$.
		Note that this operator is closable from $ L^{p}(\Omega) $ into $ L^{p}(\Omega,\mathcal{H}^d) $. 
		Also, we define $\mathbb{D}^{1,2}(\mathcal{H}^{ d})$ as the closure of real smooth cylindrical random variables with respect to the following norm
		\begin{align*}
			\Vert F\Vert_{\mathbb{D}^{1,2}(\mathcal{H}^{ d})}:=\E(\vert F\vert^{2})+\E(\Vert D^{B}F\Vert_{\mathcal{H}^{ d}}^{2})  .
		\end{align*}
		Similarly, \( \mathbb{D}^{1,2}(\lvert \mathcal{H} \rvert^{d}) \) can be defined. We also use \( \delta^{B} \) to denote the adjoint operator of \( D^{B} \). The domain of \( \delta^{B} \) is defined by
		\begin{align*}
			Dom (\delta^{B}):=\lbrace u\in L^{2}(\Omega,\mathcal{H}^{ d}): \sup_{F\in\mathbb{D}^{1,2}}\frac{\vert \E(\langle D^{B}F,u\rangle_{\mathcal{H}^{ d}})\vert}{\Vert F\Vert_{\mathbb{D}^{1,2}}(\mathcal{H}^{ d})}<\infty\rbrace .
		\end{align*}
		Also, for $ G\in Dom(\delta^{B}) $ and $ F\in\mathbb{D}^{1,2}(\mathcal{H}^{ d}) $
		\begin{align}\label{MAL_A}
			\E( F\delta^{B}(G))=\E(\langle D^{B}F,G\rangle_{\mathcal{H}^{ d}}) .
		\end{align}
		$\delta^{B}$ is called Skorohod integral. We use $\delta^{B}$ to denote the Skorohod integral. In particular, for $u = (u^{i})_{1 \leq i \leq d}\in\mathcal{H}^d$, we write
		\[
		\delta^{B}(u) = \sum_{1 \leq i \leq d}	\delta^{B^{i}}(u^i)=\int_{\mathbb{R}} u(x) \, \mathrm{d}B_{x} =\sum_{1 \leq i \leq d}=  \int_{\mathbb{R}} u^i(x) \, \mathrm{d}B^i_x.
		\]
		Also, $\delta^{B}\otimes(u):=\delta^{B^i}(u^j)e_{i}\otimes e_j$ .
		We have the following relations between $D^{B}$ and $\delta^{B}$:
		\begin{itemize}
			\item Let $u\in Dom (\delta^{B})$ and $ F\in\mathbb{D}^{1,2}(\mathcal{H}^{ d}) $ and $Fu\in L^{2}(\Omega,\mathcal{H}^{ d})$, then
			\begin{align*}
				\delta^{B} (Fu)=F\delta^{B}(u)-\langle D^{B}F,u \rangle_{\mathcal{H}^{ d}}.
			\end{align*}
			\item Let $u: \mathbb{R}\rightarrow \mathcal{H}^{ d}$ such that $\E(\Vert u\Vert_{\mathcal{H}^{ d}}^2)+\E(\Vert D^{B}u\Vert^{2}_{\mathcal{H}^{ d}\otimes \mathcal{H}^{ d}})<\infty$ and  $\delta^{B}(D^{B}u(x))\in L^2(\Omega,\mathcal{H}^{ d})$ holds for every $x\in\mathbb{ R}$. Then for $(D^{B}u)(x)(z):=(D^{B}u(z))(x)$
			\begin{align}\label{MAL_B}
				D^{B}(\delta^{B}u)(x)=u(x)+\delta^{B}_{}\big((D^{B}u)(x)\big).
			\end{align}
		\end{itemize}
		
		\subsection{Stochastic integrals}
		Suppose $B$ is a fractional Brownian motion with $H<\frac{1}{2}$. To define the stochastic integral with respect to $B$, it is more natural to work with the symmetric integral introduced by Russo and Vallois.
		\begin{definition}\label{symmetric_integral}
			For a given process $u=(u^i)_{1\leq i\leq d}:[s,t]\rightarrow \mathbb{R}^d$, the symmetric integral is defined as 
			\begin{align*}
				\int_{s}^{t}u(x) \mathrm{d}^{\circ} B_{x}:=L^{2}(\Omega)-\sum_{1\leq i\leq d}\lim_{\epsilon\rightarrow 0}\int_{s}^{t}u^i(x)\frac{B_{r+\epsilon}^i-B_{x-\epsilon}^i}{2\epsilon} \mathrm{d}r.
			\end{align*} 
			when it exists. We also set
			\begin{align*}
				\int_{s}^{t}u(x)\otimes \mathrm{d}^{\circ} B_{x}:=L^{2}(\Omega)-\sum_{1\leq i,j\leq d}\lim_{\epsilon\rightarrow 0}\int_{s}^{t}u^i(x)\frac{B_{r+\epsilon}^j-B_{x-\epsilon}^j}{2\epsilon} \mathrm{d}r\  e_{i}\otimes e_j,
			\end{align*}  
			when it exists.
		\end{definition}
		\begin{remark}
			Assume that $u:[s,t]\rightarrow \mathbb{R}^d$ be a $C^1$ function, then it is straightforward to see that
			\begin{align*}
				\int_{s}^{t}u(x)\otimes \mathrm{d} B_{x}=\int_{s}^{t}u(x)\otimes \mathrm{d}^{\circ} B_{x},  \ \ \text{in}\ \  L^{2}(\Omega,\R^d).
			\end{align*}
			Where the left integral is defined pathwise.
		\end{remark}
		The following proposition relates the symmetric integral to the Skorohod integral.
		\begin{proposition}\label{MAL}
			Let $\Phi: [s,t]\rightarrow \mathbb{R}^d$ be a stochastic process such that $\phi\in\mathcal{H}^d$ and
			\begin{align*}
				\E (\Vert \Phi\Vert_{\vert\mathcal{H}^d\vert}^2)+
				\sum_{1\leq i\leq d}\int_{s}^{t}\E(\Vert D^{B}\Phi^{i}(r)\Vert_{\vert\mathcal{H}^d\vert}^2) \mathrm{d}r<\infty .
			\end{align*}
			Assume that \( \text{Tr}_{[s,t]} D^B \Phi =\sum_{{1\leq i\leq d}} \text{Tr}_{[s,t]}D^{B^i}\Phi^{i}\) exists, where
			\begin{align}\label{Trace}
				\text{Tr}_{[s,t]} D^{B^j} \Phi^i := \lim_{\epsilon \rightarrow 0} \int_{s}^{t} \frac{\langle D^{B^j} \Phi^i(x), \chi_{[x-\epsilon, x+\epsilon] \cap [s,t]} \rangle_{\mathcal{H}}}{2\epsilon} \, \mathrm{d}x,
			\end{align}
			where \( e_k \) denotes the usual unit vector in \(\mathbb{R}^d\).
			Then the symmetric integral is well-defined and
			\begin{align*}
				\int_{s}^{t}\Phi(x)\mathrm{d}^{\circ}B_{x}=\delta^{B}(\Phi)+Tr_{[s,t]} D^{B}\Phi .
			\end{align*}
		\end{proposition}
		\begin{proof}
			Cf.\cite[Proposition 5.2.4]{Nua05}.
		\end{proof}
		\begin{remark}
			The norm we defined in \eqref{AA_1} is stronger than the norm defined in \cite{Nua05}. However, if $\text{supp}(\phi)\subset (a_1,b)$ and $a-a_1>0$, it is easy to see that in \eqref{AA_1}, the second term can be dominated by the first and third term. Nonetheless, we used the fractional derivative to represent the fractional Brownian motion because we found it easier to follow and more direct compared to the Volterra representation on a compact interval. Therefore, we prefer to use our norm derived from the fractional derivative without losing any generality.
		\end{remark}
		\section{\texorpdfstring{Proof of Proposition \eqref{MEEET}}{}}\label{FDCXX}
		%\section{Proof of Proposition \eqref{MEEET}}\label{FDCXX}
		\begin{proof}
			The proof is a direct consequence of \cite[Theorem 1.21]{GVR23} and the semigroup property with some minor modifications. First note that since our transformations are deterministic, we can relax the separability assumption of the bases space, i.e $\mathcal{B}$. Also, in \cite[Theorem 1.21]{GVR23}, the theorem is stated in discrete version, so we can apply on this theorem by iterating the semogroup with a fix step. i.e
			\begin{align}\label{AAAASWA}
				T_{nt_0}=T_{t_0}\circ...\circ T_{t_0}
			\end{align}
			$T_{nt_0}=T_{t_0}\circ...\circ T_{t_0}$. However one can easily verify for every $t_0>0$
			\begin{align*}
				F_{\lambda}:=\lbrace x\in\mathcal{B}:\ \limsup_{n\rightarrow\infty}\frac{1}{nt_0}\log \Vert T_{nt_0}x\Vert_{\mathcal{B}}\leq \lambda\rbrace
			\end{align*}
			is independent from the choice of $t_{0}$ (Because of the semigroup property).
			The complementary spaces introduced in \cite[Theorem 1.21]{GVR23}, denoted by \(\{ H^i \}_{i \geq 1}\), are indeed unique. Therefore, if we choose \( t_{0}^{m} = \frac{1}{2^m} \) as a fixed step in \eqref{AAAASWA}, the complementary spaces we obtain remain consistent. By a simple continuity argument and the fact that the set \(\left\{ \frac{p}{2^m} \mid m, p \in \mathbb{N} \right\}\) is dense in \([0, \infty)\), we can extend the proof to show that items \((\text{ii})\) and \((\text{iii})\) hold for every \( t > 0 \), thereby proving the claim. For further details, see \cite[Theorem 3.3 and Lemma 3.4]{LL10}, where the authors derived the continuous version of the Multiplicative Ergodic Theorem from the discrete case in a random setting. Although their problem is more complex due to the presence of randomness, they assumed injectivity, which cannot be directly applied to our situation.
		\end{proof}
		\begin{comment}
			\begin{remark}\label{REEEW}\todo{Are you sure you want to keep this remark in the appendix? I would rather put it in the main text.}
				From this theorem, For every $j\geq 1$, we can find a unique and finite subspace $H^j$, such that item [iii] holds true. Note that if we fixed a base $\lbrace \xi_{k} \rbrace_{1\leq k\leq m_j}$ for $H^j$, then for an invertible matrix $A_j$, our semigroup respect to this base can be represented by 
				\begin{align*}
					&T_{t}: H^j\rightarrow H^j,\\
					&T_t=\exp(A_jt)
				\end{align*}
				In particular, if $\tilde{\lambda}$ be a complex eigenvalue for $A_j$ , then we must have $Re(\tilde{\lambda})=\lambda_j$ .
			\end{remark}
		\end{comment}
		\subsection*{Acknowledgements}
		\label{sec:acknowledgements}
		Mazyar Ghani Varzaneh acknowledges financial support by the DFG via Research Unit DFG CRC/TRR 388 “Rough Analysis, Stochastic Dynamics and Related Topics”.
		\bibliographystyle{alpha}
		\bibliography{refs}

	\end{document}